\documentclass[11pt,reqno,titlepage]{amsart}
\usepackage[letterpaper,margin=1in,footskip=0.25in]{geometry}
\usepackage[only,llbracket,rrbracket]{stmaryrd}
\usepackage{amssymb}
\usepackage{microtype}
\usepackage{mathtools}
\usepackage{etoolbox}
\usepackage{enumitem}
\PassOptionsToPackage{dvipsnames}{xcolor}
\usepackage{tikz-cd}
\usepackage{trimclip}
\usepackage[noadjust]{cite}
\renewcommand{\citepunct}{;\penalty\citemidpenalty\ }
\let\counterwithout\relax
\let\counterwithin\relax
\usepackage{chngcntr}

\PassOptionsToPackage{pdfusetitle,colorlinks,pagebackref,linktoc=all}{hyperref}
\usepackage{bookmark}
\hypersetup{citecolor=OliveGreen,linkcolor=Mahogany,urlcolor=Plum}
\usepackage[hyphenbreaks]{breakurl}

\newtheorem{theorem}[subsubsection]{Theorem}
\newtheorem{proposition}[subsubsection]{Proposition}
\newtheorem{corollary}[subsubsection]{Corollary}
\newtheorem{lemma}[subsubsection]{Lemma}
\theoremstyle{definition}
\newtheorem{definition}[subsubsection]{Definition}
\newtheorem{remark}[subsubsection]{Remark}
\newtheorem{example}[subsubsection]{Example}
\newtheorem{question}[subsubsection]{Question}

\newtheoremstyle{remlem}{.5\baselineskip\@plus.2\baselineskip\@minus.2\baselineskip}{.5\baselineskip\@plus.2\baselineskip\@minus.2\baselineskip}{\itshape}{}{\itshape}{\itshape .}{5pt plus 1pt minus 1pt}{\thmname{#1}\normalfont\thmnumber{ #2}\thmnote{ (#3)}}
\theoremstyle{remlem}
\newtheorem{remlem}{Lemma}[subsubsection]

\newtheoremstyle{cited}{.5\baselineskip\@plus.2\baselineskip\@minus.2\baselineskip}{.5\baselineskip\@plus.2\baselineskip\@minus.2\baselineskip}{\itshape}{}{\bfseries}{\bfseries .}{5pt plus 1pt minus 1pt}{\thmname{#1}\thmnumber{ #2}\thmnote{ \normalfont#3}}
\theoremstyle{cited}
\newtheorem{citedthm}[subsubsection]{Theorem}
\newtheorem{citedlem}[subsubsection]{Lemma}
\newtheorem{citedprop}[subsubsection]{Proposition}

\newtheoremstyle{citeddef}{.5\baselineskip\@plus.2\baselineskip\@minus.2\baselineskip}{.5\baselineskip\@plus.2\baselineskip\@minus.2\baselineskip}{}{}{\bfseries}{\bfseries .}{5pt plus 1pt minus 1pt}{\thmname{#1}\thmnumber{ #2}\thmnote{ \normalfont#3}}
\theoremstyle{citeddef}
\newtheorem{citeddef}[subsubsection]{Definition}
\newtheorem{citedconstr}[subsubsection]{Construction}

\numberwithin{equation}{subsubsection}

\DeclareMathOperator{\Hom}{Hom}
\DeclareMathOperator{\grHom}{\prescript{*}{}{Hom}}
\DeclareMathOperator{\Frac}{Frac}
\DeclareMathOperator{\coker}{coker}
\DeclareMathOperator{\Spec}{Spec}
\DeclareMathOperator*{\colim}{colim}
\DeclareMathOperator{\im}{im}
\DeclareMathOperator{\Ann}{Ann}
\newcommand{\ev}{\mathrm{eval}}
\newcommand{\fg}{\mathrm{fg}}
\newcommand{\gr}{\mathrm{gr}}
\newcommand{\GR}{\mathrm{GR}}
\newcommand{\red}{\mathrm{red}}
\newcommand{\NN}{\mathbf{N}}
\newcommand{\QQ}{\mathbf{Q}}
\newcommand{\RR}{\mathbf{R}}
\newcommand{\ZZ}{\mathbf{Z}}
\newcommand{\FF}{\mathbf{F}}
\newcommand{\fm}{\mathfrak{m}}
\newcommand{\fn}{\mathfrak{n}}
\newcommand{\fp}{\mathfrak{p}}
\newcommand{\fq}{\mathfrak{q}}
\newcommand{\id}{\mathrm{id}}
\newcommand{\perf}{\mathrm{perf}}
\newcommand{\Tr}{\mathrm{Tr}}

\DeclarePairedDelimiterX\abs[1]\lvert\rvert{\ifblank{#1}{\:\cdot\:}{#1}}
\DeclarePairedDelimiterX\norm[1]\lVert\rVert{\ifblank{#1}{\:\cdot\:}{#1}}
\providecommand\given{}
\newcommand\SetSymbol[1][]{\nonscript : \allowbreak \nonscript \mathopen{}}
\DeclarePairedDelimiterX\Set[1]\{\}{\renewcommand\given{\SetSymbol[\delimsize]} #1}

\newcommand{\hooklongrightarrow}{\lhook\joinrel\longrightarrow}
\makeatletter
\renewcommand{\twoheadrightarrow}{\mathrel{\text{\two@rightarrow}}}
\newcommand{\longtwoheadrightarrow}{\mathrel{\text{\longtwo@rightarrow}}}
\newcommand{\two@rightarrow}{\sbox0{$\m@th\rightarrow$}\smash{\rlap{\kern0.1\wd0 \clipbox{{.3\width} {-\height} 0pt {-\height}}{$\m@th\rightarrow$}}}$\m@th\rightarrow$}
\newcommand{\longtwo@rightarrow}{\sbox0{$\m@th\longrightarrow$}\smash{\rlap{\kern0.175\wd0 \clipbox{{.3\width} {-\height} 0pt {-\height}}{$\m@th\longrightarrow$}}}$\m@th\longrightarrow$}
\makeatother

\makeatletter
\def\l@subsection{\@tocline{2}{0pt}{2pc}{6pc}{}}
\makeatother

\makeatletter
\let\@wraptoccontribs\wraptoccontribs
\patchcmd{\@setauthors}
  {,\penalty-3 \space \@setcontribs}
  {\\\@setcontribs}
  {}{}
\makeatother

\begin{document}
\title{Excellence, \emph{F}-singularities, and solidity}


\author{Rankeya Datta}
\address[R.~Datta]{Department of Mathematics, Statistics and Computer Science\\University
of Illinois at Chicago\\Chicago, IL 60607-7045\\USA}
\email{\href{mailto:rankeya@uic.edu}{rankeya@uic.edu}}
\urladdr{\url{https://rankeya.people.uic.edu/}}

\thanks{The first author was supported 
by an AMS-Simons Travel Grant, administered by the American Mathematical Society 
with support from the Simons Foundation}

\author{Takumi Murayama}
\address[T.~Murayama]{Department of Mathematics\\Princeton University\\Princeton, NJ
08544-1000\\USA}
\email{\href{mailto:takumim@math.princeton.edu}{takumim@math.princeton.edu}}
\urladdr{\url{https://web.math.princeton.edu/~takumim/}}

\thanks{The second author was supported by the National Science
Foundation under Grant No.\ DMS-1902616.}

\contrib[With an appendix by]{Karen E. Smith}
\address[K.~E.~Smith]{Department of Mathematics\\University of Michigan\\Ann Arbor, 
MI 48109-1043\\USA}
\email{\href{mailto:kesmith@umich.edu}{kesmith@umich.edu}}
\urladdr{\url{http://www.math.lsa.umich.edu/~kesmith/}}

\thanks{The third author was supported by the National Science Foundation
under Grant No.\ DMS-1801697.}

\makeatletter
\hypersetup{
  pdfauthor={Rankeya Datta and Takumi Murayama, with an appendix by Karen E.
  Smith},
  pdfsubject=\@subjclass,
  pdfkeywords={Excellent ring, Japanese ring, solid algebra, F-singularities, F-solidity, big Cohen--Macaulay algebra, Tate algebra, absolute integral closure, perfect closure, test ideal}
}
\makeatother

\begin{abstract}
An $R$-algebra $S$ is $R$-solid if there
exists a nonzero $R$-linear map $S \rightarrow R$. In characteristic $p$, the study of 
$F$-singularities such as Frobenius splittings
implicitly rely on the $R$-solidity of $R^{1/p}$.
Following recent results of the first two authors on the Frobenius 
non-splitting of certain
excellent $F$-pure rings, in this paper we use the notion of solidity to systematically study the notion of
excellence, with an emphasis on $F$-singularities.
We show that for rings $R$ essentially
of finite type over complete local rings of characteristic $p$,
reducedness implies the $R$-solidity of $R^{1/p}$, $F$-purity implies
Frobenius splitting, and $F$-pure regularity implies split $F$-regularity.
We demonstrate that Henselizations and completions are not solid, providing
obstructions for
the $R$-solidity of $R^{1/p}$ for arbitrary
excellent rings.
This also has negative consequences 
for the solidity of big Cohen--Macaulay algebras,
an important example of which are absolute integral closures of excellent 
local rings in prime characteristic. 
We establish a close
relationship between the solidity of absolute integral closures and the notion of Japanese rings. 
Analyzing the Japanese property reveals that Dedekind domains $R$
for which $R^{1/p}$ is $R$-solid are excellent, despite our recent
examples of excellent Euclidean domains with no nonzero $p^{-1}$-linear
maps.
Additionally, we show that while perfect closures are often solid
in algebro-geometric situations, 
there exist
locally excellent domains with solid perfect closures 
whose absolute integral closures are not solid.
In an appendix, Karen E. Smith uses the solidity of absolute integral
closures to characterize the test ideal for a large class of Gorenstein domains of prime
characteristic.
\end{abstract}

\maketitle

\setcounter{tocdepth}{1}
{\hypersetup{hidelinks}\tableofcontents}

\section{Introduction}
\counterwithout{subsubsection}{subsection}
\counterwithin{subsubsection}{section}

\par In the 1950s, Nagata constructed various examples of Noetherian rings
that behave badly under taking completions and integral closures
 (see \cite[App.\ A1]{Nag75}).
In order to avoid such pathologies, Grothendieck and Dieudonn\'e introduced
excellent rings, which form a large class of rings for which deep theorems of
algebraic geometry such as 
resolutions of singularities are conjectured to hold \cite[D\'ef.\ 7.8.2 and
Rem.\ 7.9.6]{EGAIV2}.
In prime characteristic $p > 0$, Smith and the first author characterized
excellent domains $R$ in terms of
the existence of nonzero $R$-linear maps $F_*R \to R$ when
the fraction field $K$ of $R$ satisfies $[K:K^p] < \infty$, that is, when $K$ is
$F$-finite \cite{DS18}. 
On the other hand, the first two authors recently constructed examples of excellent regular
domains for which there exist no nonzero $R$-linear maps $F_*R \to R$ when one relaxes
$F$-finiteness of the fraction field \cite{DM}.
Here, $F_*R$ denotes the ring $R$ with the $R$-algebra structure given by
restricting scalars via the Frobenius ($p$-th power) map $F\colon R \to F_*R$.
We will refer to $R$-linear maps of the form $F_*R \to R$ as
\emph{$p^{-1}$-linear maps}.
\medskip
\par Our goal in this paper is to systematically study how the excellence of a
 ring $R$ interacts with the existence of nonzero $p^{-1}$-linear maps in prime
characteristic, and more generally, with the existence of nonzero $R$-linear maps
$S \to R$ for certain $R$-algebras $S$ in arbitrary characteristic. Hochster utilized such
 $R$-algebras to propose a closure operation called \emph{solid closure} as a characteristic independent
 substitute for tight closure \cite{Hoc94}. Following Hochster \cite[Def.\ 1.1]{Hoc94}, an $R$-algebra $S$ is
\emph{$R$-solid} if there exists a nonzero $R$-linear map $S \to R$.
For example, for a ring $R$ of prime characteristic $p > 0$, the existence of a
nonzero $p^{-1}$-linear map such as a Frobenius splitting is precisely the
assertion that $F_*R$ is a solid
$R$-algebra. Thus, the notion of solid algebras provides a convenient 
framework to investigate questions related to the non-triviality of the dual space of certain algebras.

 \bigskip

\par We focus first on the case of prime characteristic $p > 0$.
In this setting, the theory of $F$-singularities studies how the properties of
the Frobenius map $F\colon R \to F_*R$ are related to the singularities of a Noetherian ring $R$ of characteristic $p$.
This theory began with the work of Kunz \cite{Kun69, Kun76}, and was subsequently
developed in parallel in the theories of $F$-purity and tight closure in commutative algebra 
(see, for example, \cite{HR76,HH90})
and in the theory of Frobenius splittings in algebraic geometry 
(see, for example, \cite{MR85}).
One differentiating aspect between these approaches is that Hochster and Roberts
focused on the \emph{purity} of the Frobenius map, while Mehta and Ramanthan
focused on the \emph{splitting} of the Frobenius map.
Their approaches were known to coincide in the setting of Noetherian $F$-finite rings and for complete local rings, but this does not seem to have
been known in general, even for rings of finite type over an arbitrary field of 
positive characteristic.
\medskip

\par Our first main result says that for the purposes of classical algebraic
geometry, the approaches of Hochster--Roberts and of Mehta--Ramanathan coincide.

\begin{theorem}[see Theorem \ref{thm:splittingwithgamma}]
\label{thmintro:splittingwithgamma}
An $F$-pure (resp.\ strongly $F$-regular) ring essentially of finite type over a
complete local ring of prime characteristic $p >0$ is Frobenius split (resp.\ split $F$-regular). 
\end{theorem}

\noindent In particular, $F$-purity and Frobenius splitting coincide
for rings essentially of finite type over fields of positive characteristic that are not 
necessarily $F$-finite.
Here, the notion of strong $F$-regularity that we use is one defined in terms 
of tight
closure, following Hochster,
 while split
$F$-regularity is defined in terms of splittings of maps $R \to F^e_*R$ sending
$1$ to certain elements $c \in F^e_*R$. Both variants (see Definition \ref{def:freg}) coincide with the usual notion of strong $F$-regularity in the
 setting of Noetherian $F$-finite rings, and Theorem 
\ref {thmintro:splittingwithgamma} shows that the notions coincide for a large 
class of non-$F$-finite Noetherian rings as well.
\medskip

\par Despite the equivalence of the related notions of singularity proved in Theorem
\ref{thmintro:splittingwithgamma}, the first two authors recently constructed examples which exhibit that
$F$-purity is distinct
from Frobenius splitting for excellent rings in general \cite[Thm.\ A]{DM}.
Using these same examples arising from rigid analytic geometry, we show in the present paper that while 
$F$-purity behaves well under
completions for excellent rings (and more generally for faithfully flat maps
with sufficiently nice fibers), Frobenius splittings do not in general.

\begin{proposition}[see Propositions \ref{prop:ascent-regular-maps} and \ref{prop:faithfully-flat-descent}]
\label{prop:intro-permanence}
Let $\varphi\colon R \rightarrow S$ be a map of Noetherian rings of prime characteristic $p > 0$.
\begin{enumerate}[label=$(\roman*)$,ref=\roman*]
	\item\label{prop:blah} If $\varphi$ is flat and has geometrically regular fibers, then $F$-purity 
	ascends from $R$ to $S$.
	\item\label{prop:blahblah} If $\varphi$ is faithfully flat (more generally, is pure), then $F$-purity
	descends from $S$ to $R$.
\end{enumerate}
However, there exist examples of excellent regular Henselian local rings $R$ and $S$ for
which Frobenius splitting fails to ascend under regular maps, and fails to descend under 
faithfully  flat maps.
\end{proposition}

\noindent Note that properties $(\ref{prop:blah})$ and 
$(\ref {prop:blahblah})$ for $F$-purity are well-known 
(see \cite[Prop.\ 2.4.4]{Has10} and \cite[Prop.\ 5.13]{HR76}, respectively).
The new contents in Proposition \ref{prop:intro-permanence} are the
examples which illustrate that $F$-purity is a better behaved notion of
singularity than 
Frobenius splitting with respect to permanence properties for excellent rings
 that are not essentially of finite type
over complete local rings.
At the
same time, although Frobenius splittings do not ascend 
under arbitrary flat maps with geometrically regular fibers, we show that
the property
behaves well under the closely related notion of smooth maps 
even in a non-$F$-finite setting (see Proposition \ref{prop:ascent-regular-maps}$(\ref {prop:Fsplit-ascent})$). This indicates that as a notion of 
singularity, Frobenius splitting behaves best under additional hypotheses
such as being of finite type.
\medskip

\par Just as Frobenius splitting behaves well only under additional hypotheses,
we demonstrate that the two most natural variants of the usual notion of strong
$F$-regularity, namely \emph{split $F$-regularity} 
and \emph{$F$-pure
regularity}, do not coincide for arbitrary excellent rings (see
Theorem \ref{thm:F-pure-not-split-F} and Corollary
\ref{cor:F-regular-different}). More specifically, we prove that Tate
algebras and convergent power series rings
over complete non-Archimedean fields of prime characteristic are always
$F$-pure regular, while \cite{DM} provides examples of non-Archimedean fields
over which such rings are not always split $F$-regular. This is despite the fact
that Tate algebras and convergent power series rings are the rigid
analytic analogues of polynomial rings, their localizations, and the 
completions thereof, all three of which are
split $F$-regular by Theorem \ref{thmintro:splittingwithgamma}.

\bigskip
\par An interesting aspect of the authors' examples in \cite{DM} is that not
only are the excellent regular rings $R$ therein not split $F$-regular, 
but the rings $R$ are not
even \emph{$F$-solid}, that is, $R$ does not admit any nonzero $R$-linear
maps $F_*R \to R$.
Given the connection between $F$-solidity and excellence
established in \cite{DS18}, we show that
a large class of excellent Noetherian domains are nevertheless
$F$-solid, without assuming that their fraction fields are $F$-finite.

\begin{theorem}[see Theorem \ref{thm:F-solid-eft} and Proposition \ref{prop:non-F-solid}]
\label{thm:F-solid-intro}
A reduced ring which is essentially of finite type over a complete 
local ring of prime characteristic $p > 0$ is $F$-solid. Nevertheless,
for each integer $n > 0$,
there exist excellent Henselian regular local rings of characteristic $p > 0$ 
and Krull dimension $n$ that are not $F$-solid.
\end{theorem}

\noindent Theorem \ref {thm:F-solid-intro} should be viewed as a natural
extension of Theorem \ref {thmintro:splittingwithgamma} because a Frobenius
splitting is a special instance of a nonzero $p^{-1}$-linear map.
To better explain the difference in behavior between rings of finite type over
complete local rings and more general excellent rings such as Tate algebras
from the point of view of $F$-solidity, we prove some obstructions for descending $F$-solidity from the complete to the non-complete cases
in the following result.

\begin{theorem}[see Theorems \ref{thm:Henselization-not-solid} and
  \ref{thm:completion-not-solid}]
\label{thmintro:Henselization-completion-not-solid}
Let $R$ be a Noetherian domain of arbitrary characteristic.
\begin{enumerate} [label=$(\roman*)$,ref=\roman*]
  \item If $R$ is a non-Henselian local ring whose Henselization $R^h$ is a
    domain, then $R^h$ is not a solid $R$-algebra.
	\item If $I$ is an ideal of $R$ and $R$ is not $I$-adically complete, then
	the $I$-adic completion $\widehat{R}^I$ is not a solid $R$-algebra.
\end{enumerate}
\end{theorem}

\noindent In particular, Theorem
\ref{thmintro:Henselization-completion-not-solid} illustrates that there is no general
method to reduce the question of  $F$-solidity of rings essentially of finite 
type over excellent local rings to that of $F$-solidity of rings that are 
essentially of finite type over a complete local
ring. This provides another perspective on the non-$F$-solid excellent local
examples obtained in \cite{DM}.

\bigskip
\par Theorem \ref{thmintro:Henselization-completion-not-solid} has consequences for
Hochster and Huneke's theory of big Cohen--Macaulay algebras.
In particular, while big Cohen--Macaulay algebras for complete local rings are
always solid, if $(R,\fm)$ is a mixed characteristic Noetherian domain that is
not $\fm$-adically complete, then Theorem \ref{thmintro:Henselization-completion-not-solid}
implies that all known constructions of big Cohen--Macaulay
$R$-algebras are non-$R$-solid  (see Propositions \ref{prop:BCM-solid}
and \ref{prop:BCM-not-solid}).\medskip

\par One of the principal examples of a big
Cohen--Macaulay algebra in prime characteristic is the absolute integral 
closure of an excellent local domain.
Recall that if $R$ is a domain, then the \emph{absolute integral closure} $R^+$
of $R$ is the integral closure of $R$ in an algebraic closure of its fraction
field.
Since big Cohen--Macaulay algebras in prime and mixed characteristics 
are often constructed as extensions of $R^+$,
it is natural to ask whether $R^+$ itself is a solid $R$-algebra.
The following result shows that the $R$-solidity of $R^+$ is closely
connected with an important property that is implied by excellence, namely
 that of being Japanese (see Definition \ref{def:Japanese-Nagata}).

\begin{theorem}[see Theorem \ref{thm:solid-N2} and
  Corollary \ref{cor:N-1-char0}]\label{thmintro:A+-solid}
  Let $R$ be a Noetherian domain of arbitrary characteristic.
  \begin{enumerate}[label=$(\roman*)$,ref=\roman*]
    \item If the absolute integral closure $R^+$ is a solid $R$-algebra, then
      $R$ is a Japanese ring.
      The converse holds if $R$ contains the rational numbers $\QQ$.
    \item Suppose $R$ has prime characteristic.
      If the perfect closure $R_\perf$ is a solid $R$-algebra, or more
      generally, if $R$ is $F$-solid, then $R$ is N-1 if and only if
      $R$ is Japanese.
  \end{enumerate}
\end{theorem}

\noindent The perfect closure $R_\perf$ of $R$ is the universal perfect ring
with respect to ring homomorphisms from $R$ to a perfect ring (see
\eqref{eq:defperfectclosure}), while the
N-1 condition is a generalization of the Japanese property (see Definition
\ref {def:Japanese-Nagata}).
The N-1 and Japanese properties are well-known to be equivalent when the
 fraction field $K$ of $R$
has characteristic zero due to the absence of inseparability phenomena (see 
\cite[(31.B), Cor.\ 1]{Mat80}). Thus, Theorem \ref{thmintro:A+-solid} can be
interpreted as saying
that in prime characteristic, $F$-solidity is an antidote to inseparability.

\bigskip
\par One of the more surprising aspects of the constructions in \cite{DM} are the examples of
excellent Euclidean domains which illustrate the failure of $F$-solidity 
even for prime characteristic excellent rings of Krull dimension one. 
However, by carefully analyzing the N-1 and Japanese properties, we show
in this paper that for most rings of Krull dimension one, the $R$-solidity of
$R^+$ or of $F_*R$ in fact implies that $R$ is excellent. Our result does not
impose any $F$-finiteness restrictions on the fraction field of $R$, which was the case in \cite{DS18}.

\begin{theorem}[see Theorem \ref{thm:excellent-dvr}]\label{thmintro:excellent-dvr}
Let $R$ be a Noetherian N-1 domain of
Krull dimension one. 
Suppose $R^+$
is a solid $R$-algebra (in which case one does not need $R$ to be N-1), 
or that $R$ is an $F$-solid domain of prime characteristic
$p > 0$. 
Then, $R$ is excellent. In particular, a Frobenius split (or $F$-solid) 
Dedekind domain is always excellent.
\end{theorem}

\noindent Furthermore, Theorem \ref{thmintro:excellent-dvr} yields a characterization of
excellence for DVRs of prime characteristic in terms of $F$-solidity.
This extends \cite[Thm.\ 4.1]{DS18} to a larger class of DVRs, and provides
an almost complete classification for which DVRs are $F$-solid.

\begin{corollary}[see Corollary \ref{cor:Fsolid-DVRs}]
\label{cor:intro-dvrs}
Let $(R,\fm)$ be a DVR of prime characteristic
$p > 0$. Consider the following statements:
\begin{enumerate}[label=$(\roman*)$,ref=\roman*]
	\item\label{corintro:dvr-maps} $R$ is $F$-solid.
	\item\label{corintro:dvr-F-split} $R$ is Frobenius split.
	\item\label{corintro:dvr-split-F} $R$ is split $F$-regular.
	\item\label{corintro:excellent-dvr} $R$ is excellent.
\end{enumerate}
Then, $(\ref{corintro:dvr-maps})$, $(\ref{corintro:dvr-F-split})$, and 
$(\ref{corintro:dvr-split-F})$ are equivalent, and always imply
$(\ref{corintro:excellent-dvr})$. All four assertions are equivalent
if $R$ is essentially of finite type over a complete local ring,
or the fraction field $K$ of $R$ is such that $K^{1/p}$ is
countably generated over $K$. However, there exist excellent
Henselian DVRs that are not $F$-solid.
\end{corollary}

\par A natural question that arises from Theorem \ref {thmintro:excellent-dvr}
and Corollary \ref{cor:intro-dvrs} is whether there exist non-excellent
domains that are $F$-solid. We end the main body
of the paper by addressing this question using a meta-construction of
Hochster. Hochster's construction (see Theorem \ref{thm:meta-Mel}) provides
an abundance of examples of biequidimensional locally excellent Noetherian
domains that are Frobenius split (see Corollary \ref{cor:non-exc-Fsplit}), and
consequently $F$-solid, but
not excellent. In fact, the construction of Corollary \ref{cor:non-exc-Fsplit}, 
when specialized to Krull dimension one, provides examples of locally 
excellent Noetherian domains $R$ of prime characteristic whose perfect closures $R_\perf$ are solid, 
but whose absolute integral closures $R^+$ are
not (see Example \ref {ex:F-solid-not-N1}). This illustrates that solidity of 
$R_\perf$ is more common than solidity of $R^+$, further evidenced by the
 fact that most algebro-geometric rings of prime characteristic have solid perfect
closures (see Proposition \ref {prop:perf-solid}).
On the other hand,
surprisingly little seems to be known about the solidity of absolute integral closures of Noetherian 
rings that are not complete local, even for polynomial
rings over fields of positive characteristic (see Question \ref {q:absolute-solid}). 
We list this and other related open questions in \S\ref{sect:openquestions}.
\bigskip

\par Despite the mystery surrounding solidity of absolute integral 
closures, in Appendix \ref{Karen-appendix}, Karen E. Smith proves the following 
theorem by reinterpreting results
from her fundamental thesis \cite{Sm93}, thereby partially answering
a question (Question \ref {q:image-ev1}) raised by the first two authors in \S\ref{sect:openquestions}:

\begin{theorem}[see Theorem \ref{theorem:Smith-thesis}]
\label{thm:intro-smith}
If $R$ is a complete Noetherian Gorenstein local domain of prime
characteristic $p > 0$, then the image of the evaluation at $1$ map
\[
\Hom_R(R^+, R) \xrightarrow{\ev @ 1} R
\]
is the test ideal $\tau(R)$ of $R$.
\end{theorem}

\noindent Smith obtains a similar characterization of the test ideal $\tau(R)$
when $R$ is an $\NN$-graded Gorenstein ring of finite type over $R_0$, when $R_0$
is a 
field of prime characteristic (see Theorem \ref{theorem:Smith-thesis2}). For 
this graded characterization, Smith replaces 
the absolute integral closure $R^+$ by its graded analogues $R^{+\GR}$ and $R^{+\gr}$ (see
Definition \ref{def:graded-abs-int-cl}) and $\Hom_R$ by its graded analogue $\grHom_R$ in the statement of 
Theorem \ref {thm:intro-smith}.

\subsection*{Acknowledgments}
We would first like to thank Karen E. Smith for contributing the appendix.
This paper is a continuation of work the first author began 
with Karen, so we also thank her for all of her insights and for helpful
conversations regarding absolute integral closures. 
The first author would also like to thank Karl Schwede from whom he first 
learned of some of the questions addressed in this paper at the
2015 AMS Math Research Communities in commutative algebra. Both authors are 
very grateful to Kevin Tucker for numerous insightful conversations on the
topic of solidity, especially on solidity of completions, 
and for his interest in non-$F$-finite rings. 
We are also grateful to
Melvin Hochster for answering questions about solidity, absolute
integral closures,
and for sharing with us some of the history behind Lemma \ref{lem:fed12}. 
In addition, we would like to thank Linquan Ma for discussions on the solidity
of absolute integral closures, and Bhargav Bhatt for conversations about relatively perfect maps. 
Finally, we are grateful to Eric Canton, Naoki Endo,
Shubhodip Mondal, Alapan Mukhopadhyay, Mircea Musta\c{t}\u{a}, and Matthew Stevenson for
helpful conversations.

\section{Notation and preliminaries}
\counterwithout{subsubsection}{section}
\counterwithin{subsubsection}{subsection}
All rings will be commutative with identity. Given a domain $R$, the integral
closure of $R$ in its fraction field (that is, the normalization of $R$) will often
be denoted as $R^N$.
Most rings in this paper will be Noetherian.
On the other hand, if $R$ is a Noetherian domain, we will consider its
\emph{absolute integral closure} $R^+$, which is the integral closure of $R$ in
an algebraic closure of its fraction field.
Similarly, if $R$ is a Noetherian domain of prime characteristic, we will
consider the \emph{perfect closure}
\begin{equation}\label{eq:defperfectclosure}
  R_\perf \coloneqq \colim_{e \in \ZZ_{>0}}
  \Bigl( R \overset{F}{\longrightarrow} R
  \overset{F}{\longrightarrow} \cdots \Bigr),
\end{equation}
where $F\colon R \to R$ is the Frobenius map (see Definition \ref{def:frob}).
Both rings $R^+$ and $R_\perf$ are not Noetherian unless $R$ is a
field.

\par A \emph{discrete valuation ring} (abbreviated DVR) is a Noetherian regular local ring of Krull
dimension one, or equivalently, a Noetherian normal ring with a unique nonzero
maximal ideal that is principal.

\subsection{Frobenius}
We fix our notation for the Frobenius map.
\begin{definition}\label{def:frob}
  Let $R$ be a ring of prime characteristic $p > 0$.
  The \emph{Frobenius map} is the ring homomorphism
  \[
    \begin{tikzcd}[column sep=1.475em,row sep=0]
      \mathllap{F_R\colon} R \rar & F_{R*}R.\\
      r \rar[mapsto] & r^p
    \end{tikzcd}
  \]
  We often consider the target copy of $R$ as an $R$-algebra via $F_R$, in which
  case we denote it by $F_{R*}R$. In other words, $F_{R*}R$ is the
  same underlying ring as $R$, but for $r \in R$ and $x \in F_{R*}R$, the
  $R$-algebra structure is given by
  \[r\cdot x = r^px.\]
    The $e$-th iterate of the Frobenius map is denoted $F^e_{R*}\colon R 
  \rightarrow F^e_{R*}R$. If $R$ is a domain, then $F^e_{R*}R$ will
  occasionally be identified with the ring $R^{1/p^e}$ of $p^e$-th roots
  of elements of $R$. Under this identification, the map $F^e_R$ corresponds
  to the inclusion $R \hookrightarrow R^{1/p^e}$.
  An $R$-linear map of the form
  $F^e_{R*}R \rightarrow R$ will often be called a \emph{$p^{-e}$-linear
  map}.
  We say $R$ is \emph{$F$-finite} if the Frobenius map is 
  finite (equivalently, of finite type).
\end{definition}

\par We will sometimes drop the subscript $R$ from $F^e_{R*}R$ and just write $F^e_*R$ instead
when $R$ is clear from context or is too cumbersome to write as a subscript.
We hope this does not cause any confusion.

\subsection{Excellent rings}
A focus of this paper will be to better understand various notions of 
singularities defined via the Frobenius map in the setting of excellent rings.
These are a class of rings, introduced by Grothendieck and Dieudonn\'e and defined below,
to which deep results in algebraic geometry such as resolution of singularities
are expected to extend.

\begin{definition}
\label{def:excellent}
Let $R$ be a Noetherian ring. We say that $R$ is \emph{excellent} if it satisfies
the following axioms:
\begin{enumerate}[label=$(\arabic*)$,ref=\arabic*]
  \item\label{def:excellentuc}
    $R$ is \emph{universally catenary}, that is, for every finite type
	$R$-algebra $S$, if $\fp \subsetneq \fq$ are two prime ideals of $S$,
	then all maximal chains of primes ideals from $\fp$ to $\fq$ have the same
	length.
  \item\label{def:excellentj2}
    $R$ is \emph{J-2}, that is, for every finite type $R$-algebra $S$, the
    regular locus in $\Spec(S)$
	is open.
  \item\label{def:excellentgring}
    $R$ is a \emph{$G$-ring}, that is, for every prime ideal $\fp$ of $R$, the $\fp R_\fp$-adic completion
	map $R_\fp \rightarrow \widehat{R_\fp}$ has geometrically regular fibers. 
\end{enumerate}
We say that $R$ is \emph{quasi-excellent} if it
satisfies axioms $(\ref{def:excellentj2})$ and
$(\ref{def:excellentgring})$, but not necessarily $(\ref{def:excellentuc})$.
We say that $R$ is
\emph{locally excellent} if for all prime (equivalently, maximal) ideals
$\fp$ of $R$, the local ring $R_\fp$ is excellent.
\end{definition}

\begin{example}\label{ex:excellence}\leavevmode
\begin{enumerate}[label=$(\arabic*)$,ref=\arabic*]
\item\label{ex:localgringsareqe}
  Local $G$-rings are quasi-excellent by \cite[Thm.\ 76 and footnote on p.\ 252]{Mat80}.
\item\label{ex:qeeft}
If $R$ is an excellent ring, then any ring which is essentially of 
finite type over $R$ is also excellent. In particular, since a complete local 
Noetherian ring is excellent, any algebra that is essentially of finite type
over a complete local ring is also excellent.
The same applies for quasi-excellent rings.
See \cite[(34.A) Def.]{Mat80} or \cite[Prop.\ 5.6.1$(b)$, Rem.\ 5.6.3$(i)$, Thms.\ 6.12.4 and 7.4.4$(ii)$]{EGAIV2}.
\item A Noetherian $F$-finite ring of prime characteristic is excellent by
 \cite[Thm.\ 2.5]{Kun76}.
 
\item\label{ex:F-finite-not-local} 
Examples of locally excellent rings that are not excellent are 
abundant. See Theorem \ref{thm:meta-Mel} for a construction
due to Hochster that provides many such examples.
There are also examples of locally $F$-finite rings of prime characteristic that
are not $F$-finite. See \cite{DI}, which uses a construction of Nagata
\cite[\S5]{Nag59}.
\end{enumerate}
\end{example}

\subsection{Pure maps}

We will study the notion of a pure map of modules that we introduce next.

\begin{definition}
\label{def:pure-map}
Let $R$ be a ring. A map of $R$ modules $\varphi\colon M \rightarrow N$
is \emph{pure} if for all $R$-modules $P$, the induced map
\(
  \id_P \otimes \varphi \colon P \otimes_R M \rightarrow P \otimes_R N
\)
is injective.
\end{definition}

\begin{remark}\label{rem:puremaps}\leavevmode
\begin{enumerate}[label=$(\arabic*)$,ref=\arabic*]
	\item A pure map of modules is automatically injective (by taking $P = R$).
	
	\item Since any $R$-module $P$ is a filtered colimit of its finitely generated
	submodules, and since tensor products commute with filtered colimits, 
	it follows that to check that an $R$-linear map $\varphi\colon M \rightarrow N$ is
	pure, it suffices to check that for all finitely generated $R$-modules $P$, 
	the map $\id_P \otimes \varphi$ is injective.
	
	\item Any map of $R$-modules $M \rightarrow N$ that admits an $R$-linear
	left inverse is pure.
	
	\item Any faithfully flat ring homomorphism $R \rightarrow S$ is pure
    as a map of $R$-modules \cite[Ch.\ I, $\mathsection3$, n\textsuperscript{o}
    5, Prop.\ 9]{BouCA}.\label{rem:puremapsfflat}
\end{enumerate}
\end{remark}

We will use the following result about pure maps of complete local rings.
The ``in particular'' statement was communicated to Hochster by Auslander,
although it may be older.

\begin{lemma}[{cf.\ \cite[Cor.\ 6.24]{HH90}}]\label{lem:fed12}
  Let $(R,\fm,\kappa)$ be a Noetherian local ring, and let $f\colon R \to
  M$ be a map of $R$-modules.
  Then, the following are equivalent:
  \begin{enumerate}[label=$(\roman*)$,ref=\roman*]
    \item $f$ is pure.\label{lem:fed12pure}
    \item There exists an $R$-linear map $g\colon M \to \widehat{R}$ such that
      $g \circ f$ is the canonical map $R \to \widehat{R}$.\label{lem:fed12split}
  \end{enumerate}
  In particular, if $R$ is complete local, then
  every pure map $f\colon R \to M$ splits as a map of $R$-modules.
\end{lemma}

\begin{proof}
  We first prove $(\ref{lem:fed12split}) \Rightarrow (\ref{lem:fed12pure})$.
  The canonical map $R \to \widehat{R}$ is faithfully flat, hence pure.
  Therefore, assuming $(\ref{lem:fed12split})$, the composition $g \circ f$ is a
  pure map of $R$-modules, hence so is $f$.
  \par We now show show $(\ref{lem:fed12pure}) \Rightarrow
  (\ref{lem:fed12split})$.
  Denoting by $E_R(\kappa)$ the injective hull of the residue field $\kappa$,
  we claim we have the following commutative diagram with exact rows, where the
  vertical homomorphisms are isomorphisms:
  \[
    \begin{tikzcd}[column sep=3em]
      \Hom_R\bigl(E_R(\kappa) \otimes_R M,E_R(\kappa)\bigr)
      \rar{(\id_{E_R(\kappa)} \otimes f)^*}\arrow{d}[sloped,above]{\sim} &
      \Hom_R\bigl(E_R(\kappa)\otimes_RR,E_R(\kappa)\bigr)
      \arrow{d}[sloped,above]{\sim} \rar & 0\\
      \Hom_R\bigl(M,\Hom_R\bigl(E_R(\kappa),E_R(\kappa)\bigr)\bigr) \rar{f^*} &
      \Hom_R\bigl(R,\Hom_R\bigl(E_R(\kappa),E_R(\kappa)\bigr)\bigr) \rar & 0\\
      \Hom_R(M,\widehat{R}) \rar{f^*}\arrow{u}[sloped,below]{\sim} &
      \Hom_R(R,\widehat{R}) \arrow{u}[sloped,below]{\sim} \rar & 0
    \end{tikzcd}
  \]
  The top horizontal arrow is the Matlis dual of $\id_{E_R(\kappa)}
  \otimes f$, and is surjective by the purity of $f$.
  The middle row is obtained from the top row by tensor-hom adjunction, and
  the last row is obtained using the isomorphism
  $\Hom_R(E_R(\kappa),E_R(\kappa)) \simeq \widehat{R}$ from \cite[Thm.\
  18.6$(iv)$]{Mat89}.
  Since the last row is exact by the commutativity of the diagram,
  there exists a map $g \in \Hom_R(M,\widehat{R})$ such
  that $f^*(g) = g \circ f$ is the completion homomorphism $R \to \widehat{R}$.
  \par Finally, the ``in particular'' statement follows since the map $g$ in
  $(\ref{lem:fed12split})$ is a splitting for $f$ when $R = \widehat{R}$.
\end{proof}

\begin{remark}
Lemma \ref{lem:fed12}
implies that a
pure map $R \rightarrow M$
from a Noetherian local ring $R$ is split precisely when among
the non-empty set of $\widehat{R}$-linear splittings of the
induced map $\widehat{R} \rightarrow M \otimes_R \widehat{R}$, 
there exists a splitting $f\colon M \otimes_R \widehat{R} \rightarrow \widehat{R}$ 
such that the image of the composition
$M \rightarrow M \otimes_R \widehat{R} \xrightarrow{f} \widehat{R}$
lands inside the canonical image of $R$ in $\widehat{R}$.
\end{remark}

\subsection{\textit{F}-singularities}
We now recall various notions of $F$-singularities
that will be
studied subsequently in this paper.

\begin{definition}\label{def:freg}
  Let $R$ be a Noetherian ring of prime characteristic $p > 0$, and 
$R^\circ$ denote the set of elements of $R$ not in any minimal prime.
  For every $c \in R$ and every integer $e > 0$, we denote by $\lambda^e_c$ the
  composition
  \[
    R \overset{F^e}{\longrightarrow} F_{R*}^eR \xrightarrow{F_{R*}^e(-\cdot c)}
    F_{R*}^eR.
  \]
  In other words, $\lambda^e_c$ is the unique $R$-linear map $R \rightarrow F^e_{R*}R$
  that maps $1$ to $c$.
  If $c \in R$, then following \cite[Def.\ 6.1.1]{DS16}, we say that $R$ is
  \emph{$F$-pure along $c$} if $\lambda^e_c$ is pure for some $e > 0$, and
  that $R$ is \emph{Frobenius split along $c$} if $\lambda^e_c$ splits as an
  $R$-module homomorphism for some $e >0$.
  We then say that
  \begin{enumerate}[label=$(\alph*)$,ref=\alph*]
    \item $R$ is \emph{split $F$-regular} if $R$ is Frobenius split along every $c
      \in R^\circ$ \cite[Def.\ 5.1]{HH94};\label{def:fsingssplitreg}
    \item $R$ is \emph{$F$-pure regular} if $R$ is $F$-pure along every $c \in
      R^\circ$ \cite[Rem.\ 5.3]{HH94};\label{def:fsingspureregds}
    \item $R$ is \emph{strongly $F$-regular} if for every $R$-module $M$ 
	and every submodule $N$ of $M$, $N$ is tightly closed in $M$
     \cite[Def.\ 3.3]{Has10};\label{def:fsingspureregapp}
    \item $R$ is \emph{Frobenius split} if $R$ is Frobenius split along $1 \in
      R$ \cite[Def.\ 1]{MR85}; and
    \label{def:F-split}
  \item $R$ is \emph{$F$-pure} if $R$ is $F$-pure along $1 \in R$ \cite[p.\
    121]{HR76}.\label{def:F-pure}
  \end{enumerate}
  The definition in $(\ref{def:fsingspureregapp})$ is due to Hochster. See
  \cite[Def.\ 8.2]{HH90} for the definition of tight closure for modules.
  \par Note that $(\ref{def:fsingssplitreg})$ is the usual definition of strong
  $F$-regularity in the $F$-finite setting.
  The terminology in $(\ref{def:fsingssplitreg})$ and
  $(\ref{def:fsingspureregds})$ is from \cite[Defs.\ 6.6.1 and 6.1.1]{DS16}.
  $F$-pure regular rings are called \emph{very strongly $F$-regular} in
  \cite[Def.\ 3.4]{Has10}.
\end{definition}
\begin{remark}\label{rem:fsingimplications}
  We have the following implications between the classes of
  singularities above.
  
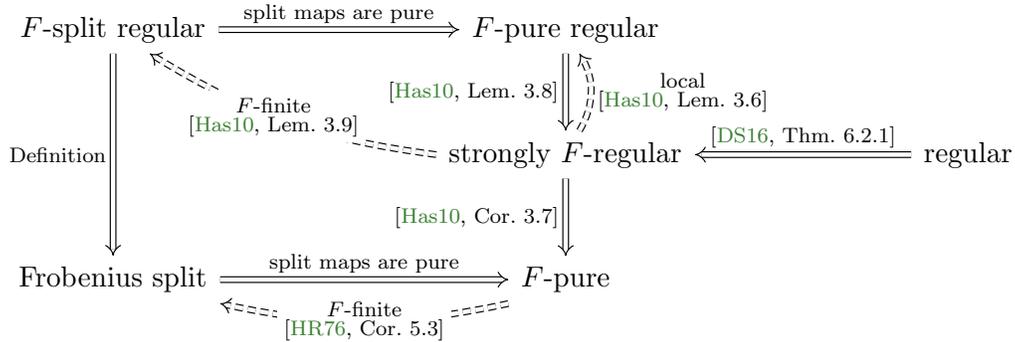
\[
    \begin{tikzcd}[column sep=7.5em,row sep=large]
      \text{$F$-split regular} \rar[Rightarrow]{\text{split maps are pure}}
      \arrow[Rightarrow]{dd}[swap]{\text{Definition}}
      & \text{$F$-pure regular}
      \dar[Rightarrow,swap]{\text{\cite[Lem.\ 3.8]{Has10}}}\\
      & \text{strongly $F$-regular} \uar[Rightarrow,dashed,bend
      right=30,swap]{\substack{\text{local}\\\text{\cite[Lem.\ 3.6]{Has10}}}}
      \arrow[start anchor=west,Rightarrow,dashed,bend left=12]{ul}[description,
      pos=0.55]{\substack{\text{$F$-finite}\\\text{\cite[Lem.\ 3.9]{Has10}}}}
      \arrow[Rightarrow]{d}[swap]{\text{\cite[Cor.\ 3.7]{Has10}}}
      & \text{regular}
      \arrow[Rightarrow]{l}[swap]{\text{\cite[Thm.\ 6.2.1]{DS16}}}\\
      \text{Frobenius split} \rar[Rightarrow]{\text{split maps are pure}}
      & \text{$F$-pure}
      \arrow[Rightarrow,dashed,bend left=12,start anchor=south west,end
      anchor=south east]{l}[description]{\substack{\text{$F$-finite}\\\text{\cite[Cor.\ 5.3]{HR76}}}}
    \end{tikzcd}
  \]
  For the implication ``regular $\Rightarrow$ strongly $F$-regular,'' the
  reference \cite[Thm.\ 6.2.1]{DS16} proves that regular local rings are
  $F$-pure regular.
  This shows that ``regular $\Rightarrow$ strongly $F$-regular'' in general,
  since a ring $R$ of prime characteristic is strongly $F$-regular if and only
  if every localization at a prime ideal is $F$-pure regular \cite[Lem.\
  3.6]{Has10}.
  We expect that all excellent regular rings are $F$-pure regular (see Theorem
\ref {eft-G-ring} and  Question
  \ref{q:regular-F-pure}).
\end{remark}

\subsection{The gamma construction}
We give a brief account of the gamma construction of Hochster and Huneke
\cite{HH94}.
\begin{citedconstr}[{\cite[(6.7) and (6.11)]{HH94}}]\label{constr:gamma}
  Let $(R,\fm,\kappa)$ be a Noetherian complete local ring of prime
  characteristic $p > 0$.
  By the Cohen structure theorem, we may identify the residue field $\kappa$ with a
  coefficient field $\kappa \subseteq R$, and by Zorn's lemma, there exists a
  \emph{$p$-basis} $\Lambda$ for $\kappa$, which is a subset $\Lambda \subseteq \kappa$
  such that $\kappa = \kappa^p(\Lambda)$, and such that for every finite subset $\Sigma
  \subseteq \Lambda$ with $s$ elements, we have $[\kappa^p(\Sigma) : \kappa^p] = p^s$.
  See \cite[p.\ 202]{Mat89}.
  \par Now for every cofinite subset $\Gamma \subseteq \Lambda$ and for every
  integer $e \ge 0$, we consider the subfield
  \[
    \kappa_e^\Gamma = \kappa\bigl[\lambda^{1/p^e}\bigr]_{\lambda \in \Gamma}
    \subseteq \kappa_{\perf}
  \]
  of a perfect closure $\kappa_{\perf}$ of $\kappa$.
  These subfields of $\kappa_\perf$ form an ascending chain, and we set
  \[
    R^\Gamma \coloneqq \varinjlim_e\Bigl(\kappa_e^\Gamma\llbracket R \rrbracket\Bigr),
  \]
  where $\kappa_e^\Gamma\llbracket R \rrbracket$ is the completion of $\kappa_e^\Gamma
  \otimes_k R$ at the extended ideal $\mathfrak{m}\cdot(\kappa_e^\Gamma \otimes_\kappa
  R)$.
  Finally, if $S$ is an $R$-algebra essentially of finite type, we set
  $S^\Gamma \coloneqq S \otimes_R R^\Gamma$.
\end{citedconstr}

The gamma construction satisfies the following properties that will be
important for us.

\begin{lemma}\label{lem:gammafacts}
  Fix notation as in Construction \ref{constr:gamma}.
  \begin{enumerate}[label=$(\roman*)$,ref=\roman*]
    \item The rings $R^\Gamma$ and $S^\Gamma$ are Noetherian and
      $F$-finite.\label{lem:gammaffin}
    \item If $S$ is reduced, then there exists a cofinite subset $\Gamma_0
      \subseteq \Lambda$ such that $S^\Gamma$ is reduced for every cofinite
      subset $\Gamma \subseteq \Gamma_0$.\label{lem:gammared}
  \end{enumerate}
\end{lemma}

\begin{proof}
  $(\ref{lem:gammaffin})$ follows from \cite[(6.11)]{HH94} and the fact that
  algebras essentially of finite type over $F$-finite rings are $F$-finite.
  $(\ref{lem:gammared})$ is \cite[Lem.\ 6.13$(a)$]{HH94}.
\end{proof}

\addtocontents{toc}{\protect\setcounter{tocdepth}{2}}
\addtocontents{toc}{\protect\smallskip}
\section{\texorpdfstring{$F$}{F}-purity, Frobenius splitting, and
variants of strong \texorpdfstring{$F$}{F}-regularity}
\label{sec:F-purity-and-splitting}
In \cite{DM}, the authors constructed examples of excellent rings
of prime characteristic that are $F$-pure but not Frobenius split, thereby answering
the long-standing open question of whether excellent $F$-pure rings are 
always Frobenius split. The origins
of this question can be traced to Hochster and Roberts's observation that 
$F$-purity and Frobenius splitting coincide for Noetherian $F$-finite rings, 
and to Auslander's lemma (Lemma \ref {lem:fed12}) which shows that the 
same result holds for complete local rings. Our goal in this section is to show 
that, despite the examples of \cite{DM}, $F$-purity and Frobenius splitting coincide for most rings that arise in classical algebraic geometry 
(Theorem \ref{thm:splittingwithgamma}), even in a non-$F$-finite setting. 
In Theorem \ref{thm:splittingwithgamma}, we will also prove an analogous result for the 
variants of strong $F$-regularity defined in Definition \ref{def:freg}. This
in turn will allow us to establish that most excellent regular rings arising in
arithmetic and geometry are $F$-pure regular (Theorem \ref{eft-G-ring}).
At the same time, we will show that the constructions considered in
\cite{DM} are always $F$-pure regular, but not necessarily split $F$-regular
(Theorem \ref{thm:F-pure-not-split-F}), giving the first examples of excellent
regular local rings (even DVRs) where these two notions of singularity do not coincide.
We will then end this section with a characterization of split $F$-regularity
for DVRs (Proposition \ref{thm:split-F-regular-DVR}). 

\par Our explorations of the various notions of $F$-singularities in this section and Section \ref {sec:permanence} will motivate the discussion of solidity of algebra extensions, which is a topic that we
will formally introduce in Section \ref{sec:F-solidity} and study in 
considerable depth thereafter.
\subsection{\textit{F}-singularities for rings of finite type over complete local
rings}
We show that for rings to which the gamma construction applies, the
notions of $F$-purity and Frobenius splitting coincide, as do the variants of strong
$F$-regularity defined in Definition \ref{def:freg}.

\begin{theorem}\label{thm:splittingwithgamma}
  Let $S$ be a ring essentially of finite type over a Noetherian complete local
  ring $(R,\fm,\kappa)$ of prime characteristic $p > 0$.
  If $S$ is strongly $F$-regular (resp.\ $F$-pure), then $S$ is split
  $F$-regular (resp.\ Frobenius split). Consequently, the notions
  $(\ref{def:fsingssplitreg})$, $(\ref{def:fsingspureregds})$, and
  $(\ref{def:fsingspureregapp})$ (resp.  $(\ref{def:F-split})$ and $(\ref{def:F-pure})$) in Definition \ref{def:freg}
  are equivalent for such $S$.
  \end{theorem}

\noindent This extends \cite[Lem.\ 1.2]{Fed83}, which proves the case when $S = R$, that is, when $S$ is complete local.

\begin{proof}
  By the gamma construction \ref{constr:gamma} (see also Lemma
  \ref{lem:gammafacts}) and \cite[Thm.\ 3.4]{Mur}, there exists a
  faithfully flat ring extension $R \hookrightarrow R^\Gamma$ such that
  $S^\Gamma \coloneqq S \otimes_R R^\Gamma$ is strongly $F$-regular (resp.\
  $F$-pure) and $F$-finite.
  By $F$-finiteness, the ring $S^\Gamma$ is split $F$-regular by
  \cite[Lem.\ 3.9]{Has10} (resp.\ $F$-split by \cite[Cor.\ 5.3]{HR76}).
  Now consider the commutative diagram
  \[
    \begin{tikzcd}
      R \rar\dar[hook] & S
      \rar{F^e_S}\dar[hook] & F^e_{S*}S \rar{F^e_{S*}(-\cdot c)}\dar[hook]
      &[3.6em]
      F^e_{S*}S\dar[hook]\\
      R^\Gamma \rar & S^\Gamma \rar{F^e_{S^\Gamma}} & F^e_{S^\Gamma*}S^\Gamma
      \rar{F^e_{S^\Gamma*}(-\cdot (c\otimes 1))} & F^e_{S^\Gamma*}S^\Gamma
    \end{tikzcd}
  \]
  for every $c \in S^\circ$ and every integer $e > 0$, where the left square
  is cocartesian by definition of $S^\Gamma$.
  Note that if $c \in S^\circ$, then $c \otimes 1 \in (S^\Gamma)^\circ$, since
  $S \to S^\Gamma$ satisfies Going Down by faithful flatness \cite[Thm.\ 9.5]{Mat89}, and so, minimal primes contract to minimal primes.

  \par Since the inclusion $R \hookrightarrow R^\Gamma$ is faithfully flat, it is
  pure, and hence splits as an $R$-module homomorphism by Lemma \ref{lem:fed12}.
  By base change, this implies the inclusion $S \hookrightarrow S^\Gamma$ splits
  as an $S$-module homomorphism.
  For both split $F$-regularity and $F$-splitting of $S$, it then suffices to note that
  if 
  \[
    F^e_{S^\Gamma*}\bigl(-\cdot (c\otimes 1)\bigr) \circ F^e_{S^\Gamma}\colon
    S^\Gamma \longrightarrow F^e_{S^\Gamma*}S^\Gamma
  \]
  splits for some $c \in
  S^\circ$ and for some $e > 0$, then restricting this splitting to 
  $F^e_{S*}S$
  and composing with a splitting
  of $S \hookrightarrow S^\Gamma$ gives a splitting of $F^e_{S*}(-\cdot c) \circ
  F^e_S$ by the commutativity of the diagram above. Thus, $S$ is split 
$F$-regular if it is strongly $F$-regular. The proof that $S$ is Frobenius split if
it is $F$-pure follows by taking $c = 1$ and $e = 1$ in the aformentioned
argument.
\end{proof}

As a corollary, we answer a question raised by Karl Schwede which inspired
this paper.

\begin{corollary}
\label{cor:divisorial-Fsplit}
Let $k$ be an arbitrary field of prime characteristic $p > 0$. Let $X$ be a
normal $k$-variety
and let $D$ be a prime divisor on $X$. Then, the divisorial valuation
ring $\mathcal{O}_{X,D}$ is Frobenius split.
\end{corollary}

\begin{proof}
The local ring $\mathcal{O}_{X,D}$ is essentially of finite type over $k$
and is $F$-pure (since it is regular). Therefore $\mathcal{O}_{X,D}$ is Frobenius split
by Theorem \ref{thm:splittingwithgamma}.
\end{proof}

\begin{remark}
The equivalence of $F$-purity and Frobenius splitting for a Noetherian ring
which is essentially of finite type over a complete local ring of prime characteristic
is in some sense the best result one can hope for. Indeed, using fundamental constructions
from rigid analytic geometry, the authors have shown recently that this equivalence
fails even for excellent Henselian DVRs \cite[Cor.\ C]{DM}. We will
have more to say about the constructions from \cite{DM} in Subsection 
\ref{subsec:split-F-vs-F-pure}.
\end{remark}

\subsection{On \textit{F}-pure regularity of regular rings} 
In this subsection, we further explore the connections between the various
variants of strong $F$-regularity (see Definition \ref {def:freg}), which we showed are
 all equivalent for
rings essentially of finite type over a complete local ring in Theorem 
\ref {thm:splittingwithgamma}. Somewhat surprisingly, these connections
are not completely understood even for the class of regular rings.
Regular rings are always strongly $F$-regular (Remark 
\ref {rem:fsingimplications}). However, even a regular 
\emph{local} ring may not
be split $F$-regular \cite[Cor.\ 4.4]{DS18}, although such a local ring is 
always
$F$-pure regular \cite[Thm.\ 6.2.1]{DS16}.  The main result of this subsection is that
most prime characteristic regular rings arising in arithmetic and geometric
settings are $F$-pure regular even in the non-local setting.

\begin{theorem}
\label{eft-G-ring}
Let $(R,\fm,\kappa)$ be a Noetherian local $G$-ring
of prime characteristic $p > 0$. If $S$ is a regular ring
that is essentially of finite type over $R$, then $S$ is $F$-pure regular. 
\end{theorem}

\begin{proof}
By assumption, the canonical map $R \rightarrow \widehat{R}$ is a 
regular homomorphism.
Since $S$ is essentially of finite type over $R$, the ring 
$\widehat{R} \otimes_R S$ is also Noetherian and the map
\[
S \longrightarrow \widehat{R} \otimes_R S
\]
is regular\footnote{A flat homomorphism of Noetherian
rings with geometrically regular fibers is called a regular map.} \cite[Prop.\ 6.8.3$(iii)$]{EGAIV2} and faithfully flat by base
change.
Thus, $S_{\widehat{R}} \coloneqq \widehat{R} \otimes_R S$
is regular because $S$ is regular \cite[Thm.\ 23.7$(ii)$]{Mat89}.
The regular ring $S_{\widehat{R}}$ is strongly $F$-regular by Remark
\ref{rem:fsingimplications}, and hence $F$-pure regular by
Theorem \ref{thm:splittingwithgamma}.
Thus, $S$ is $F$-pure regular by
faithfully flat descent of $F$-pure regularity \cite[Lem.\ 3.14]{Has10} 
(see also \cite[Prop.\ 6.1.3$(d)$]{DS16}).
\end{proof}

\begin{remark}\label{rem:whyexcellentinq1}\leavevmode
\begin{enumerate}[label=$(\arabic*)$,ref=\arabic*]
	\item\label{rem:quasi-exc} 
	Although there exist local $G$-rings that are not excellent, any regular ring
	that is essentially of finite type over a local $G$-ring is automatically 
	excellent. Indeed, regular rings are Cohen--Macaulay, and hence universally
	catenary \cite[Thm.\ 33]{Mat80}. Moreover, since local $G$-rings are
  quasi-excellent (Example \ref{ex:excellence}$(\ref{ex:localgringsareqe})$)
	and since quasi-excellence is preserved under essentially of finite type maps
  (Example \ref{ex:excellence}$(\ref{ex:qeeft})$), it
	follows that a regular ring that is essentially of finite type over a local
	$G$-ring is quasi-excellent and universally catenary, that is, such a ring is
	excellent. Thus, a large class of excellent regular rings of prime characteristic
  are $F$-pure regular by Theorem \ref{eft-G-ring}, which suggests that the same
  statement holds for \emph{all} excellent regular rings of prime characteristic
  (see Question \ref{q:regular-F-pure}).
	At the same time, one should note that $F$-pure regularity of 
regular rings often holds even without the excellence hypothesis. For example, any regular
 local ring is always $F$-pure regular (Remark \ref{rem:fsingimplications}).

  	\item\label{lem:generalizing-Hash} 
	Quasi-excellence of rings essentially of finite type over local $G$-rings
	can be used to show that if $S$ is essentially of finite type over a local $G$-ring
	$R$ of prime characteristic $p > 0$ and is strongly $F$-regular, then $S$ is $F$-pure regular. This generalizes Theorem \ref {eft-G-ring}, and
	gives an analogue of a similar
  	result for $F$-rationality due to V\'elez \cite[Thm.\ 3.8]{Vel95}.
  	The only point where the proof differs from the proof of
  	Theorem \ref{eft-G-ring} is that one has to show that the ring $S_{\widehat{R}}$ is
  	strongly $F$-regular, and to then use Theorem \ref{thm:splittingwithgamma} to 
	conclude that $S_{\widehat{R}}$ is $F$-pure regular.
  	To show strong $F$-regularity of $S_{\widehat{R}}$, one can use the following 
	lemma essentially due to Hashimoto which establishes permanence 
properties of strong $F$-regularity and $F$-pure regularity under regular
maps:

  	\begin{remlem}[cf.\ {\cite[Lem.\ 3.28]{Has10}}]\label{lem:has328}
    Let $A \to B$ be a regular map of Noetherian rings of prime characteristic
    $p > 0$.
    Assume that $A$ is $F$-pure regular (resp.\ strongly $F$-regular) and
    quasi-excellent, and that $B$ is essentially of finite type over a
    Noetherian local $G$-ring (resp.\ is a $G$-ring).
    Then, $B$ is $F$-pure regular (resp.\ strongly $F$-regular).
  	\end{remlem}
  	\noindent One applies Lemma \ref {lem:has328} to $A = R$ and 
$B = S$ in our situation and notation. Note that Hashimoto states Lemma \ref{lem:has328} under the
  	stronger assumption that $A$ is excellent. We need $A$ to be quasi-excellent
  	since the ring $S$ in the proof of Theorem \ref{eft-G-ring} is
  	quasi-excellent (see $(\ref{rem:quasi-exc})$), but not
  	a priori excellent. Of course, once one proves that $S_{\widehat{R}}$ is
	strongly $F$-regular, it then follows that $S_{\widehat{R}}$ is
  Cohen--Macaulay by \cite[Thm.\ 3.4$(c)$]{HH94} and \cite[Cor.\ 1.2]{Kaw02}, 
	and consequently, so is $S$ by faithfully flat descent. This then implies
	$S$ is excellent as in $(\ref{rem:quasi-exc})$.

\medskip

 \noindent The proof of Lemma \ref{lem:has328} proceeds as in \cite{Has10}, with the
  	following changes:
  	\begin{itemize}
    \item \cite[Lem.\ 3.27]{Has10} only uses the assumption that $R$ is a
      $G$-ring in Hashimoto's notation.
    \item In the $F$-pure regular case, \cite[Lem.\ 3.28]{Has10} does not use
      the assumption that $A$ is universally catenary.
    \item In the strongly $F$-regular case, \cite[Lem.\ 3.28]{Has10} reduces to
      the $F$-pure regular case by localization. Indeed if $B$ is a $G$-ring, 
      then as in the proof of  \cite[Lem.\ 3.28]{Has10}, for every maximal ideal
      $\fm$ of $B$, if $\fn = A \cap \fm$, then $B_\fm$ is a local $G$-ring, 
      $A_\fn$ is strongly $F$-regular hence also $F$-pure regular, and $A_\fn \rightarrow
      B_\fm$ is a regular map. Then $B_\fm$
      is $F$-pure regular by Lemma \ref{lem:has328} for the $F$-pure regular case,
      and so, $B$ is strongly $F$-regular since strong $F$-regularity is a local property \cite[Lem.\ 3.6]{Has10}.
  \end{itemize}
\end{enumerate}  
\end{remark}
We expect, though cannot prove, that an arbitrary excellent 
regular ring
of prime characteristic is $F$-pure regular (see Question
\ref{q:regular-F-pure}). However, using flatness of Frobenius
for regular rings, we can give a purely ideal-theoretic criterion for when a regular
ring is $F$-pure regular.

\begin{proposition}
\label{prop:ideal-theoretic}
Let $R$ be a ring of prime characteristic $p > 0$. Consider the following statements.
\begin{enumerate}[label=$(\roman*)$,ref=\roman*]
	\item\label{prop:Fpure-regular} $R$ is $F$-pure regular.
	\item\label{prop:nzd} For every nonzerodivisor $c \in R$, there exists $e \in \ZZ_{>0}$ such 
	that for all maximal ideals $\mathfrak{m}$, we have $c \notin \mathfrak{m}^{[p^e]}$.
	\item\label{prop:intersection} 
	For every countable collection of maximal ideals $\{\fm_e\}_{e \in \ZZ_{>0}}$,
  we have
	$\bigcap_e \fm_e^{[p^e]}  = 0$.
\end{enumerate}
We then have the following implications:
\[
  \begin{tikzcd}
    (\ref{prop:Fpure-regular}) \rar[Rightarrow]
    & (\ref{prop:nzd}) \lar[start anchor=west,end anchor=east,shift left=2,dashed,bend left=10,Rightarrow]{\text{$R$ regular}}
    \arrow[start anchor=east,end anchor=west,shift right=2,dashed,bend right=10,Rightarrow]{r}[swap]{\text{$R$ domain}}
    & \lar[Rightarrow] (\ref{prop:intersection}).
  \end{tikzcd}
\]
\end{proposition}

\begin{proof}
We first show $(\ref{prop:Fpure-regular}) \Rightarrow (\ref{prop:nzd})$. 
Let $c \in R$ be a nonzerodivisor and let $e \in \mathbf{Z}_{>0}$ such that 
the map
\begin{align*}
  \lambda^e_c\colon R &\longrightarrow F^e_{R*}R
  \intertext{mapping $1$ to $c$ is pure. In particular, for every maximal ideal
  $\mathfrak{m}$ of $R$, the map}
  (R/\mathfrak{m}) \otimes_R \lambda^e_c \colon 
  R/\mathfrak{m} &\longrightarrow F^e_{R*}(R/\mathfrak{m}^{[p^e]})
\end{align*}
is injective. But this latter map sends $1 + \mathfrak{m}$ to $c + \mathfrak{m}^{[p^e]}$, 
and so, $c \notin \mathfrak{m}^{[p^e]}$.

\smallskip
\par We next show $(\ref{prop:nzd}) \Rightarrow (\ref{prop:Fpure-regular})$ when
$R$ is regular. 
Suppose there exists $e \in \ZZ_{>0}$ satisfying $(\ref{prop:nzd})$. 
We will show that the map
\begin{align*}
\lambda^e_c\colon R &\longrightarrow F^e_{R*}R
\intertext{mapping $1$ to $c$ is pure.
We first show that for every proper
ideal $I \subseteq R$, the map}
  (R/I) \otimes_R \lambda^e_c\colon R/I &\longrightarrow F_{R*}^e(R/I^{[p^e]})\\
  r + I &\longmapsto r^{p^e}c + I^{[p^e]}
\end{align*}
is injective.
We have
\[r^{p^e}c + I^{[p^e]} = 0 \implies r^{p^e}c \in I^{[p^e]} \implies c \in
(I^{[p^e]}: r^{p^e}) = (I:r)^{[p^e]},\]
where the last equality follows by flatness of the Frobenius map on $R$ since $R$ 
is regular \cite[Thm.\ 2.1]{Kun69}. Since $c \notin \mathfrak{m}^{[p^e]}$ for any maximal ideal 
$\mathfrak{m}$ of $R$, it follows that $(I:r)$ is not contained in any maximal 
ideal of $R$. Thus, $(I:r) = R$, which shows $r \in I$. Consequently, 
$(R/I) \otimes_R \lambda^e_c$ is injective, proving the claim.

\par The injectivity of $(R/I) \otimes_R \lambda^e_c$ for every proper ideal $I
\subseteq R$ implies that $\lambda^e_c$ is pure by \cite{Hoc77} (see Remark
\ref{rem:appgor}) using the notion of approximately Gorenstein rings, 
but we give a direct proof as follows.
To show that $\lambda^e_c$ is a pure map of $R$-modules, it suffices to show that 
for every finitely generated $R$-module $M$, the map 
\begin{align*}
  \mu_M\colon M &\longrightarrow M \otimes_R F^e_{R*}R\\
  m &\longmapsto m \otimes c
\end{align*}
is injective. Recall that every finitely generated module over a Noetherian 
ring has a finite filtration
\[
  M = M_n \supsetneq M_{n-1} \supsetneq \dots \supsetneq M_1 \supsetneq M_0 = 0
\]
by cyclic modules. We now 
proceed by induction on the length $n$ of this filtration. If $n = 1$, then 
$M$ is a nonzero cyclic module and $\mu_M$ is injective by the claim. 
Otherwise, consider the short exact sequence
\[0 \longrightarrow M_{n-1} \longrightarrow M \longrightarrow \frac{M}{M_{n-1}} \longrightarrow 0.\]
Here, $M/M_{n-1}$ is cyclic, and $M_{n-1}$ is a module with a filtration of
length $n-1$. 
By the flatness of $F_{R*}^eR$ \cite[Thm.\ 2.1]{Kun69}, we get a commutative diagram of short 
exact sequences
\[
\begin{tikzcd}
0 \arrow{r} 
  & M_{n-1}\arrow[hookrightarrow]{r}\arrow[hookrightarrow]{d}{\mu_{M_{n-1}}}
  & M\arrow{r}{}\arrow{d}{\mu_M} 
  & M/M_{n-1}\arrow{r}\arrow[hookrightarrow]{d}{\mu_{M/M_{n-1}}} 
  & 0 \\
0 \arrow{r} 
  & M_{n-1} \otimes_R F_{R*}^eR\arrow[hookrightarrow]{r}[swap]{} 
  & M \otimes_R F_{R*}^eR\arrow{r}[swap]{} 
  & (M/M_{n-1}) \otimes_R F_{R*}^eR\arrow{r}
  & 0 
\end{tikzcd}
\]
The map $\mu_{M_{n-1}}$ is injective by the induction hypothesis, and the map 
$\mu_{M/M_{n-1}}$ is injective by our claim since $M/M_{n-1}$ is cyclic. Then $\mu_M$ 
must also be injective.

\smallskip
\par Finally, the negation of $(\ref{prop:nzd})$ implies the existence of a collection of maximal ideals
$\{\fm_e\}_{e \in \mathbf{Z}_{>0}}$ such that $\bigcap_e \fm_e^{[p^e]} \neq 0$,
and hence $(\ref{prop:intersection}) \Rightarrow (\ref{prop:nzd})$.
The converse 
$(\ref{prop:nzd}) \Rightarrow (\ref{prop:intersection})$ holds when $R$ is a domain, since a
nonzerodivisor on $R$ is
precisely a nonzero element of $R$.
\end{proof}

\begin{remark}\label{rem:appgor}
  One can also use Hochster's notion of approximately Gorenstein rings \cite{Hoc77}
  to give another proof of $(\ref{prop:nzd}) \Rightarrow (\ref{prop:Fpure-regular})$.
  Indeed, since $R$ is
  regular it is Gorenstein and hence approximately Gorenstein (see \cite[Def.\
  1.3 and Cor.\ 2.2$(b)$]{Hoc77}). Thus, the purity of $(R/I) \otimes_R \lambda^e_c$ for every proper ideal 
  $I \subseteq R$ implies that $\lambda^e_c$ is pure by \cite[Thm.\ 2.6$(iv)$]{Hoc77}.
  However, our argument is simpler.
\end{remark}

\begin{corollary}
\label{cor:F-pure-regularity-notlocal}
Let $R$ be a regular ring of prime characteristic $p > 0$. Then 
$R$ is $F$-pure regular if $R$ is semi-local or $R$ is a domain
such that every nonzero element is contained in only finitely 
many maximal ideals (for example, a Dedekind domain).
\end{corollary}

\begin{proof}
Suppose $R$ is semi-local. Let $\fm_1, \fm_2,\dots, \fm_n$
be the maximal ideals of $R$. If $c \in R$ is a nonzerodivisor,
then $c$ is a nonzerodivisor in each $R_{\fm_i}$. By Krull's
intersection theorem for local rings, there exists $e \gg 0$
such that $c \notin \fm_i^{[p^e]}R_{\fm_i}$. Then $c \notin \fm_i^{e}$
for each $i$, so $R$ is $F$-pure regular by 
Proposition \ref{prop:ideal-theoretic}$(\ref{prop:nzd})$. Similarly, 
if $R$ is a domain where every nonzero element is contained in only finitely 
many maximal ideals, $R$ satisfies 
Proposition \ref{prop:ideal-theoretic}$(\ref{prop:nzd})$ by an application
of Krull's intersection theorem for domains.
\end{proof}

\begin{example}
As an application of Corollary \ref{cor:F-pure-regularity-notlocal},
Nagata's example of an infinite-dimensional
Noetherian regular domain \cite[App.\ A1, Ex.\ 1]{Nag75} is $F$-pure
regular, since every nonzero element of this ring is only contained
in finitely many maximal ideals by construction.
\end{example}

\subsection{Excellent \textit{F}-pure regular rings are not always split 
\textit{F}-regular}
\label{subsec:split-F-vs-F-pure}
Let $R$ be a Noetherian ring of prime characteristic $p > 0$. While the
notions of $F$-pure regularity and split $F$-regularity coincide
if $R$ is $F$-finite by \cite[Cor.\ 5.2]{HR76} or if $R$ is
essentially of finite type over a complete local ring by Theorem 
\ref{thm:splittingwithgamma}, the two notions are not equivalent
in general, even when $R$ is a regular local ring. The first such 
examples were obtained in \cite[Ex.\ 4.5.1]{DS16}, where the first
author and Smith constructed DVRs in the function
field of $\mathbf{P}^2_{\mathbf{F}_p}$ that are not Frobenius split.
However, the rings in \cite[Ex.\ 4.5.1]{DS16} are not excellent, and
a general expectation since Hochster and Huneke's work on tight closure
\cite{HH90, HH94} is
that the various notions of $F$-singularities
are well-behaved only in the setting of excellent rings. 
The goal of this subsection is to use the constructions of \cite{DM}
to give examples of excellent regular (local) rings for which
$F$-pure regularity and split $F$-regularity do not coincide. 
These are the first examples which illustrate that the two notions
do not coincide for excellent rings.

Before we can state our main result, we need to introduce the rigid
analytic analogue of a polynomial ring over a field and its local variant.

\begin{definition}
\label{def:Tate-alg}
Let $(k, \abs{})$ be a complete non-Archimedean field.
For every positive integer $n > 0$, the \emph{Tate algebra} in $n$
  indeterminates over $k$ is the $k$-subalgebra
  \[
    T_n(k) \coloneqq k\{ X_1,X_2,\ldots,X_n \}
    \coloneqq
    \Set[\Bigg]{\sum_{\nu \in \ZZ_{\ge0}^n} a_\nu X^{\nu} \given
      \begin{array}{@{}c@{}}
        a_\nu \in k\ \text{and}\ \abs{a_\nu} \to 0\\
        \text{as}\ \nu_1 + \nu_2 + \cdots + \nu_n \to \infty
      \end{array}}
     \]
  of 
   the formal power
  series $k\llbracket X_1,X_2,\ldots,X_n \rrbracket$ in $n$ indeterminates over $k$.
  Similarly, for every positive integer $n > 0$, the \emph{convergent power series ring}
  in $n$ indeterminates over $k$ is the $k$-subalgebra
  \[
    K_n(k) \coloneqq k\langle X_1,X_2,\ldots,X_n \rangle
    \coloneqq
    \Set*{\sum_{\nu \in \ZZ_{\ge0}^n} a_\nu X^{\nu} \given
      \begin{array}{@{}c@{}}
        a_\nu \in k\ \text{and there exists}\\ r_1,r_2,\ldots,r_n, M \in
        \RR_{>0}\ \text{such that}\\
        \ \abs{a_\nu}\,r^{\nu_1}_1r^{\nu_2}_2\cdots r^{\nu_n}_n \leq M\ 
        \text{for all}\ \nu \in \ZZ_{\ge0}^n
      \end{array}}
      \]
  of
  $k\llbracket X_1,X_2,\ldots,X_n \rrbracket$.
\end{definition}

For the notion of a non-Archimedean field, which we will not
define in this paper, we refer the reader to 
\cite[\S2.1]{Bos14}. The main properties of Tate algebras and convergent power series rings
of relevance to this paper are summarized in the following result.

\begin{proposition}
\label{prop:properties}
Let $(k,\abs{})$ be a complete non-Archimedean field. Then for any integer
$n > 0$, we have the following:
\begin{enumerate}[label=$(\roman*)$]
	\item $T_n(k)$ and $K_n(k)$ are Noetherian and regular of dimension $n$.
	\item $T_n(k)$ and $K_n(k)$ are excellent.
	\item $K_n(k)$ is local and Henselian.
\end{enumerate}
\end{proposition}

\begin{proof}[Indication of proof]
For precise references for all these properties and additional ones, 
we refer the reader to \cite[Thms.\ 2.7 and 4.3]{DM}.
\end{proof}

We also prove the following preliminary result.
\begin{lemma}\label{lem:gvffinite}
  Let $(k, \abs{})$ be a complete non-Archimedean field of characteristic $p >
  0$ such that $[k^{1/p}:k] < \infty$.
  Then, $T_n(k)$ and $K_n(k)$ are $F$-finite for every integer $n > 0$.
\end{lemma}
\noindent Although the statement for $T_n(k)$ is contained in the proof of
\cite[Thm.\ 12]{GV74}, we give an elementary proof.
\begin{proof}
 Fix a basis $\{a_1,a_2,\dots,a_m\}$ of $F_{k*}k$ over $k$. Since $k$ is complete,
and $F_{k*}k$ is an algebraic extension of $k$ via the Frobenius map $F_k$, 
there exists a unique extension
of the norm $\abs{}$ on $k$ to a norm on $F_{k*}k$ \cite[App.\ A, Thm.\ 3]{Bos14}.
Denoting the norm on $F_{k*}$ also by $\abs{}$, uniqueness and
\cite[App.\ A, Thm.\ 1]{Bos14} then gives us that
for $x = x_1a_1 + x_2a_2 + \cdots + x_ma_m \in F_{k*}k$,
\begin{equation}
\label{eq:max-norm}
\abs{x} = \max_{1 \leq i \leq m} \bigl\{\abs{x_i}\bigr\}.
\end{equation}
Now consider the formal power series ring $A \coloneqq k\llbracket
X_1,X_2,\dots,X_n\rrbracket$. 
Since $k$ is $F$-finite, $F_{A*}A$ is a free $A$-module with basis
given by 
\[
  \mathcal{B} \coloneqq \Set[\big]{a_iX^{\beta_1}_1X^{\beta_2}_2\cdots X^{\beta_n}_n \given
  1 \leq i \leq m, \, 0 \leq \beta_j \leq p-1}.
\]
\par We next claim that $\mathcal{B}$ is also a basis of $F_{T_n*}T_n$ over $T_n$.
Indeed, let $f = \sum_{\nu \in \ZZ_{\geq 0}^n} b_\nu X^\nu$ 
be an element of $F_{T_n*}T_n$. Expressing $f$, considered as an element of $A$,
in terms of the basis $\mathcal{B}$, to get the coefficient of 
$a_iX^{\beta_1}_1X^{\beta_2}_2\cdots X^{\beta_n}_n$, we first collect all the terms in the expansion
of $f$ for which the power of $X_i$ is congruent to $\beta_i$ modulo $p$,
for all $1 \leq i \leq n$. This gives us a power series of the form
\[
\sum_{\nu \in \ZZ_{\geq 0}^n} b_{p\nu+\beta} 
X^{p\nu + \beta},
\]
where $\beta \coloneqq (\beta_1,\beta_2,\dots,\beta_n)$. Note that we still have $\abs{b_{p\nu+\beta}}
\rightarrow 0$ as $\nu_1 + \nu_2 + \dots + \nu_n \rightarrow \infty$. Expressing each
$b_{p\nu+\beta}$ in terms of the basis $\{a_1,a_2,\dots,a_m\}$ of $F_{k*}k$ over $k$,
we find that the coefficients $b_{i,p\nu+\beta}$ of $a_i$ also satisfy
$\abs{b_{i,p\nu+\beta}} \rightarrow 0$ as $\nu_1 + \nu_2 + \dots + \nu_n \rightarrow \infty$
by $(\ref{eq:max-norm})$. Therefore, the coefficient of
$a_iX^{\beta_1}_1X^{\beta_2}_2\cdots X^{\beta_n}_n$
in the expansion of $f$ with respect to the basis $\mathcal{B}$, which is precisely
\begin{equation}\label{eq:bipnubeta}
  \sum_{\nu \in \ZZ_{\geq 0}^n} b_{i,p\nu+\beta} X^{p\nu + \beta},
\end{equation}
is also an element of $T_n$ by the discussion above. This implies that
$\mathcal{B}$ is also a free $T_n$-basis of $F_{T_n*}T_n$, and so, $T_n$ is
$F$-finite. 
\par Finally, we show that $\mathcal{B}$
is also a free basis of $F_{K_n*}K_n$ over $K_n$. To see this, note that for
\[ 
f = \sum_{\nu \in \ZZ_{\geq 0}^n} b_\nu X^\nu \in F_{K_n*}K_n
\] 
considered as an element of
$F_*(k\llbracket X_1,X_2,\dots,X_n\rrbracket)$, the coefficient of 
$a_iX^{\beta_1}_1X^{\beta_2}_2\cdots X^{\beta_n}_n$ is again given by
\eqref{eq:bipnubeta},
where $b_{i,p\nu+\beta}$ is the coefficient of $a_i$ when 
$b_{p\nu+\beta}$ is expressed in terms of the basis $\{a_1,a_2,\dots,a_m\}$ of $F_{k*}k$ over $k$.
Now choose $r_1,r_2,\ldots,r_n, M \in \RR_{>0}$ such that for all 
$\abs{b_\nu}\,r^{\nu_1}_1r^{\nu_2}_2\cdots r^{\nu_n}_n \leq M$ for all $\nu \in \ZZ_{\ge0}^n$.
Then 
\[
\abs{b_{i,p\nu+\beta}}\,r^{p\nu_1 + \beta_1}_1r^{p\nu_2 + \beta_2}_2\cdots r^{p\nu_n + \beta_n}_n \leq
\abs{b_{p\nu+\beta}}\,r^{p\nu_1 + \beta_1}_1r^{p\nu_2 + \beta_2}_2\cdots r^{p\nu_n + \beta_n}_n \leq M
\]
for all $\nu \in \ZZ_{\geq 0}^n$, where the first inequality again follows by
(\ref{eq:max-norm}). Thus, the same choice of $r_1,r_2,\ldots,r_n, M$ shows that
$\sum_{\nu \in \ZZ_{\geq 0}^n} b_{i,p\nu+\beta} X^{p\nu + \beta}$ is also an element
of $K_n$. Thus $\mathcal{B}$ is also a free $K_n$-basis of $F_{K_n*}K_n$.
\end{proof}

With these preliminaries, we can now state and prove the main result of this subsection.

\begin{theorem}
\label{thm:F-pure-not-split-F}
Let $(k, \abs{})$ be a complete non-Archimedean field of characteristic $p > 0$. 
Then we have the following:
\begin{enumerate}[label=$(\roman*)$,ref=\roman*]
	\item\label{thm:Tate-F-regular} 
	For every integer $n > 0$, $T_n(k)$ and $K_n(k)$ are $F$-pure regular.
	\item\label{thm:Tate-not-F-regular} 
	There exists a complete non-Archimedean field $(k,\abs{})$ of
	characteristic $p > 0$
	such that for every integer $n > 0$, 
	$T_n(k)$ and $K_n(k)$ are not Frobenius split, and consequently, not
	split $F$-regular.
\end{enumerate}
\end{theorem}

\begin{proof}
We first prove $(\ref{thm:Tate-F-regular})$. Note that since $K_n(k)$ is local
by Proposition \ref{prop:properties}, it is $F$-pure regular by Corollary
\ref{cor:F-pure-regularity-notlocal}. To show $F$-pure regularity of $T_n(k)$,
we first assume that $k$ satisfies $[k^{1/p}:k] < \infty$, that is, $k$ is 
$F$-finite.
Since $T_n(k)$ is $F$-finite by Lemma \ref{lem:gvffinite} and is regular by Proposition \ref{prop:properties},
it follows that $T_n$ is split $F$-regular, and hence $F$-pure regular, when
 $k$ is $F$-finite.

\par Now suppose $(k,\abs{})$ is an arbitrary complete non-Archimedean field of
characteristic $p > 0$.
Let $\ell$ be the completion of the algebraic closure of $k$, where the completion
is taken with respect to the metric induced by the unique norm on the
algebraic closure of $k$ that extends the norm on $k$. Then by Krasner's lemma
\cite[App.\ A, Lem.\ 6]{Bos14},
$\ell$ is an algebraically closed non-Archimedean field, hence in particular,
$F$-finite. Therefore $T_n(\ell)$ is $F$-pure regular by the previous paragraph.
Since the canonical inclusion
\[T_n(k) \hooklongrightarrow T_n(\ell)\]
is faithfully flat by \cite[Lem.\ 2.1.2]{Ber93} and \cite[App.\ B, Prop.\ 5]{Bos14}, 
and since $F$-pure regularity descends under faithfully flat maps \cite[Lem.\
3.14]{Has10} (see also \cite[Prop.\ 6.1.3$(d)$]{DS16}), it follows that $T_n(k)$ is $F$-pure regular.
This completes the proof of $(\ref{thm:Tate-F-regular})$.

\par $(\ref{thm:Tate-not-F-regular})$ follows from \cite[Thm.\ A]{DM} and 
\cite[Rem.\ 5.5]{DM} by working over a complete non-Archimedean field $(k,\abs{})$
of characteristic $p > 0$ such that $k$ admits no nonzero continuous linear 
functionals $k^{1/p} \rightarrow k$. An explicit example of such a field
is constructed based on ideas of Gabber in \cite[Thm.\ 5.2]{DM}.
\end{proof}

As a consequence, we obtain examples of excellent Henselian $F$-pure regular local
rings that are not split $F$-regular.

\begin{corollary}
\label{cor:F-regular-different}
There exists an excellent Henselian DVR of prime characteristic
$p > 0$ that is not split $F$-regular.
\end{corollary}

\begin{proof}
Choose a non-Archimedean field $(k,\abs{})$ as in 
Theorem \ref{thm:F-pure-not-split-F}$(\ref{thm:Tate-not-F-regular})$.
Then the convergent power series ring $K_1(k)$ is an excellent Henselian
discrete valuation ring by Proposition \ref{prop:properties} that is
not even Frobenius split by 
Theorem \ref{thm:F-pure-not-split-F}$(\ref{thm:Tate-not-F-regular})$,
hence also not split $F$-regular.
\end{proof}

\begin{remark}
\label{rem:F-pure-strong-F}
As far as we are aware, it is not known if Hochster's  
notion of strong $F$-regularity defined via tight closure
(Definition \ref{def:freg}$(\ref{def:fsingspureregapp})$) coincides with
$F$-pure regularity for excellent Noetherian rings (see Question
\ref{q:regular-F-pure}). Note that if a 
counterexample exists, then it necessarily has to be non-local by
\cite[Lem.\ 3.8]{Has10} (see also \cite[Prop.\ 6.3.2]{DS16}). We
expect the two notions to not be equivalent in general even for regular
rings, because we expect there to exist regular rings that do not
satisfy the ideal-theoretic characterization of $F$-pure regularity
given in Proposition \ref{prop:ideal-theoretic}. Note that any regular
ring is always strongly $F$-regular by \cite[Lem.\ 3.6]{Has10} because
strong $F$-regularity is a local property and regularity localizes.
\end{remark}

In Lemma \ref{lem:gvffinite},
we showed that
if $(k,\abs{})$ is an $F$-finite non-Archimedean field, then
$T_n(k)$ is $F$-finite, and hence, split $F$-regular for each integer $n > 0$.
However, using some non-Archimedean functional analysis, 
one can also show that Tate algebras are often split $F$-regular, even if $k$ is
not $F$-finite.
To do so, we will use the following:

\begin{lemma}\label{lem:completedperfisperf}
Let $(k,\abs{})$ be a complete non-Archimedean field of
  characteristic $p > 0$. Let $k_\perf$ be a perfect closure 
  of $k$ and let $\ell$ be the completion of $k_\perf$
  with respect to the unique norm on $k_\perf$ that extends 
  the norm on $k$.
  Then, $\ell$ is perfect.
\end{lemma}
\begin{proof}
Let $a \in \ell$ be a nonzero element, and choose a sequence 
$(b_n)_n$ of elements in $k_\perf$ such that $b_n \rightarrow a$
as $n \rightarrow \infty$. Such a sequence exists because $k_\perf$
is dense in $\ell$. Since $k_\perf$ is perfect, the sequence
$(b_n^{1/p})_n$ also consists of elements in $k_\perf$. Moreover,
$(b_n^{1/p})_n$ is a Cauchy sequence because $(b_n)_n$ is Cauchy. 
Let $a' = \lim_{n \rightarrow \infty} b^{1/p}_n$ be the
limit of $(b_n^{1/p})_n$ in $\ell$. Then $(a')^p$ is the limit of
$(b_n)_n$, and so, $(a')^p = a$ because sequences have unique limits
in a metric space. Thus, any element of $\ell$ has a $p$-th root in $\ell$,
that is, $\ell$ is perfect. 
\end{proof}

We now show that Tate algebras are often split $F$-regular.
\begin{proposition}
\label{prop:split-F-regular-Tate}
Let $(k,\abs{})$ be a complete non-Archimedean field of
  characteristic $p > 0$. Let $k_\perf$ be a perfect closure 
  of $k$ and let $\ell$ be the completion of $k_\perf$
  with respect to the unique norm on $k_\perf$ that extends 
  the norm on $k$.
Then for every $n > 0$, $T_n(k)$ and $K_n(k)$ are split $F$-regular in the
following cases:
\begin{enumerate}[label=$(\roman*)$,ref=\roman*]
	\item\label{cor:F-split-usually-spher} $(k,\abs{})$ is spherically complete.
	\item\label{cor:F-split-usually-count} $k^{1/p}$ has a dense $k$-subspace $V$ 
	which has a countable $k$-basis.
	\item\label{cor:F-split-polar} $\abs{k^\times}$ is not discrete,
	and the norm on $k_\perf$ is polar with respect to the norm on $k$.
\end{enumerate}
\end{proposition}

\noindent For the definitions of spherically complete non-Archimedean fields and polar norms, 
we refer
the reader to \cite[Def.\ 2.13]{DM} and \cite[Def.\ 2.12]{DM} respectively.

\begin{proof}
\par We first claim that if $(\ref{cor:F-split-usually-count})$ holds,
then $\ell$ has a dense $k$-subspace which is countably generated over
$k$. Since $k_\perf$ is dense in $\ell$, to prove the claim it suffices
to show that $k_\perf$ has a dense $k$-subspace which is countably
generated over $k$. Let $S$ be a countable generating set for 
the dense $k$-subspace $V$ of $k^{1/p}$. We will inductively show that for all $e > 0$,
$k^{1/p^e}$ has a dense $k$-subspace that is countably generated over $k$.
The case $e = 1$ holds by the assumption of $(\ref{cor:F-split-usually-count})$.
For $e > 1$, by the inductive hypothesis, let $T_{e-1}$
be a countable subset of $k^{1/p^{e-1}}$ that generates a dense
$k$-linear subspace. 
Define
\[
  S^{1/p^{e-1}} \coloneqq \Set[\big]{x^{p/p^e} \given x \in S} \subseteq k^{1/p^e}.
\]
Then note that the $k^{1/p^{e-1}}$-linear subspace $V_{e-1}$ of $k^{1/p^e}$
generated by $S^{1/p^{e-1}}$ is dense in $k^{1/p^e}$ because the $k$-linear
space $V$ generated by $S$ is dense in $k^{1/p}$. 
Now consider the countable
set
\[S_e \coloneqq \Set[\big]{xy \in k^{1/p^e} \given x \in T_{e-1},\, y \in S^{1/p^{e-1}}}.\]
To prove the induction statement, it suffices to show that 
the $k$-linear subspace $k \cdot \{S_e \}$ of $k^{1/p^e}$ 
generated by $S_e$ is dense
in $k^{1/p^e}$. For any $z \in k^{1/p^e}$ and any real number 
$\epsilon > 0$, let
$B_\epsilon(z) \subseteq k^{1/p^e}$ be the ball of radius $\epsilon$
centered at $z$. Since $V_{e-1}$ is dense in $k^{1/p^e}$, there exists
$y \in V_{e-1} \cap B_\epsilon(z)$. Choose
$y_1,y_2,\dots,y_n \in S^{1/p^{e-1}}$ and $a_1,a_2,\dots,a_n \in k^{1/p^{e-1}}$
such that
\[y = a_1y_1 + a_2y_2 + \cdots + a_ny_n.\]
We may assume each $y_i \neq 0$. Since $k\cdot\{ T_{e-1} \}$ is
dense in $k^{1/p^{e-1}}$, for all $1 \leq i \leq n$, choose 
$x_i \in k\cdot \{T_{e-1} \}$ such that 
$\abs{a_i - x_i} < \epsilon/\abs{y_i}$. Then 
$x_1y_1 + x_2y_2 + \cdots + x_ny_n \in k \cdot \{ S_e \}$,
and
\begin{align*}
  \MoveEqLeft[4]\abs{z - x_1y_1 + x_2y_2 + \cdots + x_ny_n}\\
&\leq \max\bigl\{\abs{z-y},\abs{y-x_1y_1 +
x_2y_2 + \cdots + x_ny_n}\bigr\}\\
&= \max\bigl\{\abs{z-y},\abs{(a_1-x_1)y_1 + (a_2-x_2)y_2 + \cdots + (a_n-x_n)y_n}\bigr\}\\
&\leq \max\bigl\{\abs{z-y},\abs{(a_1-x_1)y_1},\abs{(a_2-x_2)y_2},\ldots,\abs{(a_1-x_1)y_1}\bigr\}\\
&\leq \epsilon.
\end{align*}
Here the first and penultimate inequalities follow 
by the non-Archimedean triangle inequality.
Thus, $x_1y_1 + x_2y_2 + \ldots + x_ny_n \in B_\epsilon(z)$, and since $z$
and $\epsilon$ were chosen arbitrarily, this implies 
$k\langle S_e \rangle$ is dense in $k^{1/p^e}$, as desired. Finally,
because $k_\perf = \bigcup_{e > 0} k^{1/p^e}$, it follows that
$k_\perf$ has a dense $k$-subspace that is countably generated over $k$
by \cite[Thm.\ 4.2.13$(iv)$]{PGS10}, proving the claim.

\par We can now show that $T_n(k)$ and $K_n(k)$ are split $F$-regular.
The ring $T_n(\ell)$ is 
split $F$-regular  by Lemma \ref{lem:gvffinite}
because $\ell$ is $F$-finite by Lemma \ref{lem:completedperfisperf}. Therefore, it suffices to show that for
$k$ satisfying any of the three conditions of this Proposition,
the inclusion
$T_n(k) \hookrightarrow T_n(\ell)$
splits. This is because a direct summand of a split $F$-regular ring is
split $F$-regular. However, if $k$ satisfies any of the conditions 
$(\ref{cor:F-split-usually-spher})$,
$(\ref{cor:F-split-usually-count})$, or $(\ref{cor:F-split-polar})$, then variants of the
Hahn-Banach theorem for normed spaces over $\RR$ and $\mathbf{C}$ 
also hold for normed spaces
over $k$; see for example \cite[Thm.\ 2.15]{DM}. When $k$ satisfies 
$(\ref{cor:F-split-usually-count})$,
in order to apply \cite[Thm.\ 2.15]{DM} one needs 
the observation made in the above claim that
$\ell$ has a countably generated dense $k$-subspace. When
$k$ satisfies $(\ref{cor:F-split-polar})$, then in 
order to apply \cite[Thm.\ 2.15]{DM} one needs that the norm
on $\ell$ is polar with respect to the norm on $k$. But this
follows by the hypothesis of $(\ref{cor:F-split-polar})$ 
and \cite[Thm.\ 4.4.16$(ii)$]{PGS10} because the norm on $k_\perf$
is polar and $k_\perf$ is a dense $k$-subspace of $\ell$.
The upshot is that the identity
map $\id_k\colon k \rightarrow k$ lifts to a continuous $k$-linear functional 
\[\phi\colon \ell \longrightarrow k\]
when $k$ satisfies $(\ref{cor:F-split-usually-spher})$,
$(\ref{cor:F-split-usually-count})$, or $(\ref{cor:F-split-polar})$.
By continuity, $\phi$ maps sequences in $\ell$ whose norms converge to $0$
to sequences in $k$ whose norms converge to $0$ \cite[Lem.\ 2.11]{DM}.
Therefore, the induced map
\begin{align*}
    \widetilde{\phi}\colon T_n(\ell) &\longrightarrow T_n(k)\\
    \sum_{\nu \in \ZZ^n_{\geq0}} a_\nu X^\nu &\longmapsto 
    \sum_{\nu \in \ZZ^n_{\geq0}} \phi(a_\nu)X^\nu
  \end{align*}
is a $T_n(k)$-linear splitting of $T_n(k) \hookrightarrow T_n(\ell)$, as desired.

\par Similarly, $K_n(\ell)$ is split $F$-regular because $\ell$ is $F$-finite
by Lemma \ref{lem:gvffinite}. Now, 
for $\phi\colon \ell \rightarrow k$ as above, we get an induced 
map
\begin{align*}
    \phi'\colon K_n(\ell) &\longrightarrow K_n(k)\\
    \sum_{\nu \in \ZZ^n_{\geq0}} a_\nu X^\nu &\longmapsto 
    \sum_{\nu \in \ZZ^n_{\geq0}} \phi(a_\nu)X^\nu.
  \end{align*}
The reason why $\phi(a_\nu)X^\nu$ is an element $K_n(k)$ is 
because by continuity of $\phi$, there exists a real number $B > 0$
such that for all $a \in \ell$, $\abs{\phi(a)} \leq B\,\abs{a}$ 
\cite[Lem.\ 2.11]{DM}. 
Hence if $r_1,\dots,r_n, M \in \RR_{> 0}$ are such that
$\abs{a_\nu}\,r^{\nu_1}_1\ldots r^{\nu_n}_n \leq M$ for all
$\nu = (\nu_1,\ldots,\nu_n) \in \ZZ_{\geq 0}$, then
$\abs{\phi(a_\nu)}\,r^{\nu_1}_1\ldots r^{\nu_n}_n \leq BM$ for
all $\nu \in \ZZ^n_{\geq 0}$, which shows that 
$\sum_\nu \phi(a_\nu)X^\nu \in K_n(k)$
by Definition \ref{def:Tate-alg}. Thus, $K_n(k)$ is a direct summand
of the split $F$-regular ring $K_n(\ell)$, and consequently, also split $F$-regular.
\end{proof}

\begin{remark}
\label{rem:split-Freg-TnKn}
Let $(k, \abs{})$ be a complete non-Archimedean field of characteristic $p> 0$.
In \cite[Thms.\ 3.1 and 4.4]{DM}, the authors show that a necessary and sufficient
condition for $T_n(k)$ and $K_n(k)$ to be Frobenius split is for there to exist
a nonzero continuous $k$-linear functional $k^{1/p} \rightarrow k$. However,
we do not know if the existence of such a functional is also sufficient for
split $F$-regularity of $T_n(k)$ and $K_n(k)$. As far as we can ascertain,
it is not clear if the existence of a nonzero continuous functional $k^{1/p} \rightarrow k$ 
implies the existence of a nonzero continuous functional $\ell \rightarrow k$,
where $\ell = \widehat{k_\perf}$, as in the proof of Proposition 
\ref{prop:split-F-regular-Tate}. The main obstruction seems to be that if
$(M,\abs{})$ is a field equipped with a non-Archimedean valuation that is not
complete, then there may not be any nonzero continuous functionals 
$\widehat{M} \rightarrow M$. Indeed, this fails even when the value group
$\abs{M^\times} \simeq \ZZ$. For such discrete valuations, if $M^\circ$ is the
corresponding DVR, then $\widehat{M}$ is the fraction field of the 
$M^{\circ\circ}$-adic completion $\widehat{M^\circ}$ of $M^\circ$. Here 
$M^{\circ\circ}$ is the principal maximal ideal of $M^\circ$. However,
we show in Theorem \ref{thm:completion-not-solid} that there are no nonzero
$M^\circ$-linear maps $\widehat{M^\circ} \rightarrow M^\circ$. This implies that
are no nonzero continuous $M$-linear maps $\widehat{M} \rightarrow M$. Indeed,
assume for contradiction that $\phi\colon \widehat{M} \rightarrow M$ is a
nonzero continuous $M$-linear map. We may assume without loss of 
generality that $\phi(1) \neq 0$. By \cite[Lem.\ 2.11]{DM}, choose $B \in \RR_{> 0}$
such that for all $x \in \widehat{M}$, we have $\abs{\phi(x)} \leq B\abs{x}$. Since the value
group of $M$ is nontrivial, choose $a \in M^\times$ such that $B < \abs{a}$.  Then, the
composition
\[
\widehat{M} \xrightarrow{-\cdot a} \widehat{M} \overset{\phi}{\longrightarrow} M
\]
is a nonzero continuous $M$-linear map (because 
$\widehat{M} \xrightarrow{-\cdot a} \widehat{M}$ is an isomorphism), which
induces a $M^\circ$-linear map of the corresponding valuation
rings $\widehat{M^\circ} \rightarrow M^\circ$ since for all 
$x \in \widehat{M^\circ}$,
\[
\abs{\phi(xa^{-1})} = \abs{a^{-1}} \cdot \abs{\phi(x)} \leq \abs{a^{-1}} \cdot B\abs{x} < \abs{x} \leq 1.
\]
Note that the induced map $\widehat{M^\circ} \rightarrow M^\circ$
is nonzero because $\phi(1)\neq 0$. But this contradicts Theorem \ref{thm:completion-not-solid}, as observed above. This
remark motivates the interesting question of whether Frobenius splitting of a regular
ring is sufficient to imply split $F$-regularity (see Question \ref{q:Fsplit-reg}),
which to the best of our knowledge is not known. However, in the next subsection we
show that all Frobenius split DVRs are split $F$-regular 
(Proposition \ref{thm:split-F-regular-DVR}).
\end{remark}

\subsection{Split \textit{F}-regularity of discrete valuation rings} 
We have seen that even DVRs of prime characteristic exhibit varied
behvaior from the point of view of $F$-singularities despite being the 
simplest examples of regular local rings. In particular, while 
DVRs of prime characteristic are always F-pure regular, they are not 
always split $F$-regular, or even Frobenius split. Our goal in this subsection
 will be to show that Frobenius splitting,
or more generally, the existence of a nonzero $p^{-1}$-linear map
is the only obstruction to a DVR of characteristic $p$
being split $F$-regular. This 
removes generic $F$-finiteness assumptions from a result
of the first author and Smith \cite[Cor.\ 6.6.3]{DS16}.

\begin{proposition}
\label{thm:split-F-regular-DVR}
Let $(R,\fm,\kappa)$ be a DVR of prime characteristic
$p > 0$. Then, the following are equivalent:
\begin{enumerate}[label=$(\roman*)$,ref=\roman*]
	\item\label{prop:dvr-maps} $R$ has nonzero $p^{-1}$-linear map.
	\item\label{prop:dvr-F-split} $R$ is Frobenius split.
	\item\label{prop:dvr-split-F} $R$ is split $F$-regular.
\end{enumerate}
\end{proposition}

\begin{proof}
For the proof of $(\ref{prop:dvr-maps})\Rightarrow (\ref{prop:dvr-F-split})$
we will use that fact that $R$ is a principal ideal domain (PID). Let 
$\varphi\colon F_{R*}R \rightarrow R$ be a nonzero $R$-linear map. Since
$R$ is a PID, $\im(\varphi) = aR$, for some nonzero $a\in R$. Restricting
the codomain of $\varphi$ to $\im(\varphi) = aR$, and then using the fact that 
$aR \simeq R$ as $R$-modules gives us an $R$-linear surjection 
\[\widetilde{\varphi}\colon F_{R*}R \longtwoheadrightarrow R.\]
Choose $x \in F_{R*}R$ such that $\widetilde{\varphi}(x) = 1$. Then the
composition
\[
  F_{R*}R \xrightarrow{F_{R*}(-\cdot x)} F_{R*}R \overset{\widetilde{\varphi}}{\longrightarrow}
R
\]
is a Frobenius splitting of $R$. We have $(\ref{prop:dvr-split-F})\Rightarrow
(\ref{prop:dvr-maps})$ because any split $F$-regular ring is Frobenius split,
and a Frobenius splitting is a nonzero $p^{-1}$-linear map.

\par It remains to show $(\ref{prop:dvr-F-split})\Rightarrow(\ref{prop:dvr-split-F})$.
Let $\pi$ be a generator of the maximal ideal $\fm$. It suffices
to show that for any integer $n \geq 0$, the ring $R$ is Frobenius split along $\pi^n$.
Indeed, any $x \in R$ is of the form $x = u\pi^n$ for some unit $u \in R^\times$
and $n \geq 0$.
If $R$ is Frobenius split along $\pi^n$, then choose $e > 0$ and a map
$\varphi\colon F^e_{R*}R \rightarrow R$ such that $\varphi(\pi^n) = 1$. Then
the composition
\[F^e_{R*} \xrightarrow{F^e_{R*}(-\cdot u^{-1})} F^e_{R*}R
\overset{\varphi}{\longrightarrow} R\]
maps $x \mapsto 1$. Now to show that $R$ is Frobenius split along $\pi^n$, 
choose $e \gg 0$ such that
$p^e - n > n$, and let $\varphi_e\colon F^e_{R*}R \rightarrow R$ be a splitting of
$F^e_R\colon R \rightarrow F^e_{R*}R$ (such a splitting exists for
any $e > 0$ because $R$ is Frobenius split). Then consider the composition
\[\phi \colon F^e_{R*}R \xrightarrow{F^e_{R*}(-\cdot \pi^{p^e-n})} F^e_{R*}R
\overset{\varphi_e}{\longrightarrow} R.\]
We have $\phi(\pi^n) = \varphi_e(\pi^{p^e}) = \pi\varphi_e(1) = \pi$. Thus, $\im(\phi)$,
which is an ideal of $R$, is either $\pi R$ or $R$. If $\im(\phi) = \pi R$, then 
restricting the codomain of $\phi$ to $\pi R$ and composing with the canonical
isomorphism $\pi R \simeq R$ (that sends $\pi$ to $1$) shows that $R$ is Frobenius
split along $\pi^n$. On the other hand, if $\im(\phi) = R$, this means that
$\varphi_e(\pi^{p^e-n}R) = R$. Choose $b \in R$ such that
$\varphi_e(b\pi^{p^e-n}) = 1$. Since $p^e - n > n$ by our choice of $e$, we see that
$\pi^n$ divides $b\pi^{p^e-n}$. Thus, let $b\pi^{p^e-n} = a\pi^n$, for some $a \in R$. The 
composition
\[F^e_{R*}R \xrightarrow{F^e_{R*}(-\cdot a)} F^e_{R*}R
\overset{\varphi_e}{\longrightarrow} R\]
maps $\pi^n \mapsto \varphi_e(a\pi^n) = \varphi_e(b\pi^{p^e-n}) = 1$, which
again shows that $R$ is Frobenius split along $\pi^n$. This completes the 
proof of $(\ref{prop:dvr-F-split})\Rightarrow(\ref{prop:dvr-split-F})$.
\end{proof}

\begin{remark}
  \leavevmode
If $R$ is generically $F$-finite, then \cite[Cor.\ 6.6.3]{DS16}
	shows that the equivalent conditions of Proposition \ref{thm:split-F-regular-DVR}
	are further equivalent to $R$ being excellent. 
	Since there are excellent Henselian DVRs
	that are not split $F$-regular, we cannot hope to get rid of the generic
	$F$-finiteness assumption from \cite[Cor.\ 6.6.3]{DS16} to show that
	all excellent DVRs are split $F$-regular. However,
	we will show in Theorem \ref{thm:excellent-dvr} that any split $F$-regular
	DVR will be excellent, without any generic restrictions
	on such a ring.
\end{remark}

\section{\texorpdfstring{$F$}{F}-purity vs.\ Frobenius splitting: a contrast in permanence properties}
\label{sec:permanence}
\par In this section we continue exploring the closely related notions of
 $F$-purity and Frobenius splitting from the perspective of permanence
properties. Even though these two notions of singularity coincide for most
 rings arising in algebro-geometric
applications by Theorem \ref{thm:splittingwithgamma}, we 
will now illustrate some ways in which they differ. The general slogan is that 
$F$-purity is a more stable notion of singularity than Frobenius splitting.
As evidence of the better stability properties of
$F$-purity, we will show in this section that while $F$-purity 
descends under faithfully flat maps and ascends under regular maps, 
the same does not hold for Frobenius splitting in general.

\par Recall that given a field $k$, a Noetherian $k$-algebra $R$ is 
\emph{geometrically regular over $k$} if, for every finite field extension
$L/k$, the ring
$L \otimes_k R$ is regular. A map of rings $R \rightarrow S$ is \emph{regular} 
if it is flat and has geometrically regular fibers. Implicit in this definition is
the assertion that all the fibers of $R \rightarrow S$ are Noetherian even
if $R$ and $S$ are not. A map is 
\emph{(essentially) smooth} if it is regular and (essentially) of finite type.

\par Most notions of singularities such as reduced, normal, Cohen--Macaulay, 
Gorenstein ascend under regular maps, that is, if $R$ satisfies a certain 
notion of singularity then so does $S$ for a regular map $R \rightarrow S$. 
We will use a characterization of 
regular maps in prime characteristic due to Radu and Andr\'e  to show that 
while Frobenius splitting ascends under essentially smooth maps, it does not 
ascend under regular maps in general. However, $F$-purity always ascends under 
regular maps, providing evidence for the assertion that $F$-purity is a better 
behaved notion of singularity.

\subsection{Relative Frobenius and the Radu--Andr\'e theorem}
\par The result of Radu and Andr\'e that we will employ is a relative version of Kunz's 
celebrated theorem which characterizes regularity of a Noetherian ring in terms 
of flatness of the (absolute) Frobenius map \cite[Thm.\ 2.1]{Kun69}. To state 
Radu and Andr\'e's result we need to introduce the relative Frobenius map 
\cite[Exp.\ XV, Def.\ 3]{SGA5}, which is also known as the 
\emph{Radu--Andr\'e homomorphism} in the commutative algebraic literature.

\begin{definition}\label{def:radu-andre}
Let $\varphi\colon R \to S$ be a homomorphism of rings of prime characteristic 
$p > 0$. For every integer $e > 0$, consider the cocartesian diagram
  \[
    \begin{tikzcd}[column sep=4em]
      R \rar{F_R^e}\dar[swap]{\varphi} & F^e_{R*}R
      \arrow[bend left=30]{ddr}{F^e_{R*}\varphi}
      \dar{\varphi \otimes_R F^e_{R*}R}\\
      S \rar{\id_S \otimes_R F^e_R} \arrow[bend right=12,end
      anchor=west]{drr}[swap]{F^e_S} & S \otimes_R F^e_{R*}R
      \arrow[dashed]{dr}[description]{F^e_{S/R}}\\
      & & F^e_{S*}S
    \end{tikzcd}
  \]
in the category of rings. The \emph{$e$-th Radu--Andr\'e ring} 
is the ring $S \otimes_R F^e_{R*}R$, and the \emph{$e$-th relative Frobenius homomorphism} 
associated to $\varphi$ is the ring homomorphism
  \[
    \begin{tikzcd}[column sep=1.475em,row sep=0]
      \mathllap{F^e_{S/R}\colon} S \otimes_R F^e_{R*}R \rar & F^e_{S*}S\\
      s \otimes r \rar[mapsto] & s^{p^e} \varphi(r)
    \end{tikzcd}
  \]
If $e = 1$, we denote $F^1_{S/R}$ by $F_{S/R}$. We also sometimes denote 
$F^e_{S/R}$ by $F^e_\varphi$.
\end{definition}

Radu and Andr\'e's result is the following:

\begin{citedthm}[{\citeleft\citen{Rad92}\citemid Thm.\
  4\citepunct\citen{And93}\citemid Thm.\ 1\citeright}]\label{thm:raduandre}
A homomorphism $\varphi\colon R \to S$ of Noetherian rings of prime characteristic 
$p > 0$ is regular if and only if $F_{S/R}$ is faithfully flat.
\end{citedthm}

One then recovers \cite[Thm.\ 2.1]{Kun69} upon applying Theorem
\ref {thm:raduandre} to the
homomorphism $\FF_p \rightarrow R$, for a Noetherian ring $R$ of 
prime characteristic $p > 0$.

\subsection{Ascent under regular maps and faithfully flat descent}
Theorem \ref{thm:raduandre} has the following consequence for Frobenius splitting 
and $F$-purity:

\begin{proposition}
\label{prop:ascent-regular-maps}
Let $\varphi: R \rightarrow S$ be a regular map of Noetherian rings of prime 
characteristic $p > 0$. We have the following:
\begin{enumerate}[label=$(\roman*)$,ref=\roman*]
	\item\label{prop:Fpure-ascent} If $R$ is $F$-pure, then $S$ is $F$-pure.
	\item\label{prop:Fsplit-ascent} If $\varphi$ is essentially smooth and 
	$R$ is Frobenius split, then $S$ is Frobenius split.
	\item\label{prop:rel-Frob-iso} Suppose $R$ and $S$ are arbitrary rings 
	(not necessarily Noetherian). 
	If $F_{S/R}$ is an isomorphism and $R$ is Frobenius split, then $S$ is 
	Frobenius split. 
	\item\label{prop:Henselization} If $S$
	is a filtered colimit of \'etale $R$-algebras and $R$ is Frobenius split, 
	then $S$ is Frobenius split.
	Thus, the (strict) Henselization of a Frobenius split local ring is Frobenius 
	split.
	\item\label{prop:Fsplit-nonascent} There exists a regular map $\varphi$ such 
	that $R$ and $S$ are excellent local, $R$ is Frobenius split, but $S$ is not Frobenius
	split.
\end{enumerate}
\end{proposition}

\begin{proof}
Since $\varphi$ is a regular map, $F_{S/R}$ is faithfully flat by 
Theorem \ref{thm:raduandre}. 
It is well-known that faithfully flat maps are pure 
\cite[Ch.\ I, $\mathsection3$, n\textsuperscript{o} 5, Prop.\ 9]{BouCA}, and so, $F_{S/R}$ is a 
pure map. With this preliminary observation, we can now begin 
the proof of the Proposition.

\par$(\ref{prop:Fpure-ascent})$ Since purity is preserved under base change, 
$F$-purity of $R$ implies that the map $\varphi \otimes_R F_{R*}R$ is pure. 
Consequently, the composition
\[
F_{S} = F_{S/R} \circ (\varphi \otimes_R F_{R})
\]
is also pure, that is, $S$ is $F$-pure.
See also \cite[Prop.\ 2.4.4]{Has10}.

\par$(\ref{prop:Fsplit-ascent})$ If $\varphi$ is essentially of finite type, 
then $F_{S/R}$ is a finite map. Indeed, it is well-known that the relative Frobenius of a finite 
type map is always finite. Now, if $S$ is a localization of a finite 
type $R$-algebra $B$, then it is straightforward to check that $F_{S/R}$ 
is a localization of the finite map $F_{B/R}$, and consequently also finite. 
Thus, if $\varphi$ is essentially smooth, then $F_{S/R}$ is a finite map of 
Noetherian rings which is also pure. It follows by \cite[Cor.\ 5.2]{HR76} that 
$F_{S/R}$ splits. Moreover, since $R$ is Frobenius split, the map 
$\id_S\otimes_R F_R$ splits by base change. Therefore the composition 
$F_{S} = F_{S/R} \circ (\id_S \otimes_R F_{R})$ also splits, that is, 
$S$ is Frobenius split.

\par$(\ref{prop:rel-Frob-iso})$ If the relative Frobenius is an isomorphism, 
then 
\[
\begin{tikzcd}
R \arrow[r] \arrow[d, "F_R"'] &S \arrow[d, "F_S"]\\
F_{R*}R \arrow[r] &F_{S*}S
\end{tikzcd}
\]
is a pushout square and the assertion follows.

\par$(\ref{prop:Henselization})$ If $R \rightarrow A$ is an \'etale map of 
rings of characteristic
$p > 0$, then the relative Frobenius $F_{R/A}$ is always an isomorphism
\cite[Exp.\ XV, Prop.\ 1$(c)$]{SGA5}.
Thus, if $S$ is a filtered colimit of \'etale $R$-algebras, then $F_{S/R}$ is
a filtered colimit of relative Frobenii that are all isomorphisms. Consequently, 
$F_{S/R}$ is also an isomorphism. Then, Frobenius splitting is preserved under 
filtered colimits of \'etale maps
by what we just proved. Frobenius splitting is preserved by (strict) Henselizations 
because the (strict) Henselization
of a local ring is a filtered colimit of \'etale maps.

\par$(\ref{prop:Fsplit-nonascent})$ Choose a non-Archimedean field $(k,\abs{})$
of prime characteristic $p > 0$ such that the convergent power series ring
$K_1(k)$ is not Frobenius split; see \cite[Rem.\ 5.5]{DM}. Then $K_1(k)$ is an
excellent Henselian DVR (Proposition \ref{prop:properties}),
the unique ring homomorphism $\FF_p \rightarrow K_1(k)$ is regular
\cite[Prop.\ 6.7.7]{EGAIV2},
and $\FF_p$ is Frobenius split while $K_1(k)$ is not.
\end{proof}

\begin{remark}
  \leavevmode
\begin{enumerate}
	\item For Proposition \ref{prop:ascent-regular-maps}$(\ref{prop:Fsplit-nonascent})$,
	one can construct examples of $R$ and $S$ even in the function field of 
	$\mathbf{P}^2_{\overline{\FF_p}}$ if one relaxes the requirement for $S$
	to be excellent. Indeed, there exist discrete valuation rings 
	in the function field of $\mathbf{P}^2_{\overline{\FF_p}}$ that are not
	excellent \cite[Cor.\ 4.4]{DS18}, and consequently also not Frobenius split 
	\cite[Thm.\ 4.1]{DS18}.
	Take such a discrete valuation ring $S$. 
	Then the map $\overline{\FF_p} \rightarrow S$ is regular, and 
	$\overline{\FF_p}$ is Frobenius split, but $S$ itself is not. 
	
  \item Proposition \ref{prop:ascent-regular-maps}$(\ref{prop:Fpure-ascent})$ holds more generally for 
	any ring homomorphism $\varphi\colon R \rightarrow S$ of Noetherian rings for which 
  the relative Frobenius $F_{S/R}$ is a pure map \cite[Prop.\ 2.4.4]{Has10}. Such ring 
	homomorphisms are called \emph{$F$-pure homomorphisms} by Hashimoto 
	\cite[(2.3)]{Has10} and should be thought of as a relative version of the 
  notion of $F$-purity of a ring \cite[Prop.\ 2.4.3]{Has10}.

	\item The key assumption of Proposition
  \ref{prop:ascent-regular-maps}$(\ref{prop:rel-Frob-iso})$ is that 
	the relative Frobenius of $R \rightarrow S$ is an isomorphism. 
	Such maps are usually called \emph{relatively perfect}
	in the literature, because they are a relative version of perfect rings. 	Relative perfection plays an important role
	in the proofs of various invariance properties in the theory of 
	$F$-singularities under \'etale base change.
	Thus, it is natural to wonder
	how close the notion of relative perfection is to that of being \'etale. 
	If $R \rightarrow S$ is relatively perfect, then $R[S^p] = S$, which 
	immediately implies $\Omega_{S/R} = 0$ (since $d(s^p) = pd(s) = 0$). That is, a
	relatively perfect map is formally unramified. Moreover, the set
    $\{1\}$ is a $p$-basis of $S$ over $R$ in the sense of \cite[Ch.\ 0, D\'ef.\
    21.1.9]{EGAIV1}.
	Then injectivity of $F_{S/R}$ implies that $R \rightarrow S$ is formally smooth
  by \cite[Ch.\ 0, Thm.\ 21.2.7]{EGAIV1}. In other words, a relatively perfect map
	is formally \'etale. In fact, a flat relatively
	perfect map $R \rightarrow S$ satisfies the stronger property that the cotangent
	complex $L_{S/R}$ is acyclic \cite[Prop.\ 3.12]{Bha19}
  (formally \'etale is equivalent to acyclicity of the
	cotangent complex in degrees $0$ and $-1$; see 
\cite[\href{https://stacks.math.columbia.edu/tag/0D11}{Tags 0D11},
\href{https://stacks.math.columbia.edu/tag/08RB}{08RB}, 
\href{https://stacks.math.columbia.edu/tag/06E5}{06E5}]{stacks-project}).  
    Nevertheless, a relatively perfect map need not be the filtered colimit of
    \'etale maps.
	For example,
	$F_{S/R}$ is always an isomorphism if $R$ and $S$ are perfect. 
	However, if $S$ is not a flat $R$-algebra, then $S$
	cannot be a filtered colimit of \'etale (hence flat) $R$-algebras.	
	Failure of flatness is not the only obstruction that prevents a relatively perfect map  
	from being a filtered colimit of \'etale maps.\footnote{We thank Alapan Mukhopadhyay for this observation.} 
	Indeed, if $R \rightarrow S$ is an injective map of domains that is a filtered colimit of \'etale maps, 
	then the induced map on fraction fields has to be algebraic. Thus, 
	$\FF_p \hookrightarrow \FF_p[T]_{\perf}$ is a flat relatively perfect map 
	that cannot be a filtered colimit of \'etale maps.
	
	\item Relative perfection also shows that if $S$ is a filtered
	colimit of \'etale $R$-algebras and $R$ is $F$-finite, then $S$ is also
  $F$-finite. Thus, 
	(strict) Henselizations of $F$-finite rings are $F$-finite.
	\end{enumerate}
\end{remark}

Finally, while it is well-known that $F$-purity descends under faithfully flat maps, 
we show that Frobenius splitting does not.

\begin{proposition}
\label{prop:faithfully-flat-descent}
Let $\varphi: R \rightarrow S$ be a faithfully flat map of rings of prime 
characteristic $p > 0$. 
If $S$ is $F$-pure, then so is $R$. However, Frobenius splitting does not 
descend under faithfully 
flat maps even when $R$ and $S$ are both excellent, Henselian and regular.
\end{proposition}

\begin{proof}
The proof of the first statement is in \cite[Prop.\ 5.13]{HR76}, but we reprove
it here for convenience.
Note that $\varphi$ is pure since it is faithfully flat (Remark
\ref{rem:puremaps}$(\ref{rem:puremapsfflat})$). Therefore the 
composition $F_{S} \circ \varphi$ is pure as a map of $R$-modules if 
$S$ is $F$-pure. However,
\[
F_{S} \circ \varphi = \varphi \circ F_{R},
\]
and so, $F_{R}$ is pure, that is, $R$ is $F$-pure. This shows that $F$-purity
satisfies faithfully flat descent.

For the second statement, choose a non-Archimedean field $(k,\abs{})$ 
such that the convergent power series
ring $K_1(k)$ is not Frobenius split as in \cite[Rem.\ 5.5]{DM}. Take $R
\coloneqq K_1(k)$, and consider the completion
$S \coloneqq \widehat{R}
= k\llbracket X \rrbracket$ of $R$ with respect to the maximal ideal $(X)$.
Then, the canonical completion map $R \rightarrow S$ is regular 
(since $K_1(k)$ is excellent) and faithfully flat, and $R$ is not Frobenius 
split even though $S$ is. Thus, Frobenius splitting
does not satisfy faithfully flat descent even for faithfully flat maps
between excellent Henselian regular rings of prime characteristic.
\end{proof}

\begin{remark}
While the Henselization of a Frobenius split local ring $(R,\fm)$
of prime characteristic is always Frobenius split, as far as we are aware 
the converse
is not known except in special cases. For one such special case, 
suppose $R$ is a generically
$F$-finite Noetherian normal local domain such that $R^h$ is Frobenius split.
Then $R^h$ is also a Noetherian normal local domain 
\cite[Cor.\ to Thm.\ 23.9]{Mat89}.
Let $K$ be the fraction field of $R$ and $K^h$ be the fraction field of $R^h$.
Since the field extension $K \subseteq K^h$ is separable algebraic, the map $K
\to K^h$ is relatively perfect.
The diagram
\[
  \begin{tikzcd}
    K \rar\dar[swap]{F_K} & K^h\dar{F_{K^h}}\\
    F_{K*}K \rar & F_{K^h*}K^h
  \end{tikzcd}
\]
is therefore cocartesian, and hence $K^h$ is $F$-finite by base change (see also
\cite[Lem.\ 2(7)]{Has15}).
Now, because $R^h$ is
Frobenius split and generically $F$-finite, it follows that $R^h$ is also
$F$-finite by \cite[Thm.\ 3.2]{DS18}.
This implies that $R$ is also $F$-finite.
Indeed, since $R \rightarrow R^h$ is relatively perfect, it follows by 
faithfully flat descent of module finiteness that if
$F_{R^h*}R^h = F_{R*}R \otimes_R R^h$ is a finite $R^h$-module, then
$F_{R*}R$ is a finite $R$-module (see also \cite[Cor.\ 22]{Has15}).
Thus,
$R$ is $F$-finite and also $F$-pure by Proposition \ref{prop:faithfully-flat-descent}.
But an $F$-finite and $F$-pure ring is Frobenius split by \cite[Cor.\ 5.2]{HR76}.

\par If $(R,\fm)$ is a non-Henselian DVR such that $R^h$ 
is Frobenius split, then the only obstruction to Frobenius splitting of $R$
is for the latter to possess a nonzero $p^{-1}$-linear map 
by Theorem \ref{thm:split-F-regular-DVR}. Note that if $\Hom_R(R^h,R)$
is nontrivial, then there would exist a nonzero $R$-linear map 
$\phi\colon R^h \rightarrow R$ such that $\phi(1) \neq 0$. The composition
\[F_{R*}R \hooklongrightarrow F_{R^h*}R^h \xrightarrow{\text{$F$-splitting}} R^h 
\overset{\phi}{\longrightarrow} R\]
would consequently give a nonzero $p^{-1}$-linear map
of $R$, allowing us to conclude that $R$ is Frobenius split. However, we will
show soon that $\Hom_R(R^h,R)$
is trivial for any non-Henselian Noetherian local domain $R$ 
for which $R^h$ is a domain (Theorem \ref{thm:Henselization-not-solid}). Therefore,
the question of whether Frobenius splitting descends over Henselizations
seems nontrivial even for DVRs. 
\end{remark}

\addtocontents{toc}{\protect\smallskip}
\section{\texorpdfstring{$F$}{F}-solidity and non-\texorpdfstring{$F$}{F}-finite excellent rings}
\label{sec:F-solidity}
An underlying theme of the previous sections was an exploration of 
when nonzero $p^{-e}$-linear maps exist for excellent rings of prime characteristic.
The existence of a nonzero $p^{-e}$-linear map is a special instance of a
more general phenomenon studied by
Hochster in \cite{Hoc94}, with the goal of formulating a characteristic 
independent closure operation that satisfies properties similar to tight closure.
Hochster called this notion a \emph{solid} module, and our goal in this
section, and the rest of the paper, is to revisit 
his notion of solidity from the point of view of the excellence
condition. However,
our focus is not solid closure, for which we 
refer the reader to \cite{Hoc94}. 

\subsection{Solid modules and algebras}
We first define solid modules and algebras.
\begin{citeddef}[{\cite[Def.\ 1.1]{Hoc94}}]
\label{def:solid-modules}
Let $R$ be a ring. 
An $R$-module $M$ is \emph{solid} if 
there exists a nonzero $R$-linear map from $M \rightarrow R$. An 
$R$-algebra $S$ is \emph{solid} if it 
is solid as an $R$-module. 
\end{citeddef}

\begin{remark}\label{rem:useful-solid-facts}
  \leavevmode
\begin{enumerate}
	\item If $R$ is a ring of prime characteristic, then $F_{R*}R$ is a solid 
	$R$-algebra precisely when $R$ admits a nonzero $p^{-1}$-linear map.

	\item A simple, but surprisingly useful, observation is that an $R$-algebra 
	$S$ is solid if and only if there exists 
	an $R$-linear map $\varphi\colon S \rightarrow R$
	such that $\varphi(1) \neq 0$. Indeed, the existence of such a map
	implies $R$-solidity of $S$ by definition. Conversely, if $S$
	is a solid $R$-algebra, then choose any $R$-linear map 
	$\phi\colon S \rightarrow R$ such that $\phi(s) \neq 0$, for some $s \in S$.
	Then the composition
	\begin{equation}
	\label{rem:eval-1}
    S \xrightarrow{-\cdot s} S \overset{\phi}{\longrightarrow} R
    \end{equation}
	is a map that sends $1$ to $\phi(s) \neq 0$. Moreover, the same argument shows
	that the ideal $\sum_{\phi} \im(\phi)$ of $R$, as $\phi$ ranges over all
	elements of $\Hom_R(S,R)$, is the same as the ideal of $I_{S/R} \coloneqq
	\Set{\phi(1) \given \phi\in \Hom_R(S,R)}$. Indeed, the fact that $I_{S/R}$ is an ideal 
	follows from the observation that $I_{S/R}$ is the image of the $R$-linear map
	$\Hom_R(S,R) \rightarrow R$ given by evaluation at $1$. Furthermore, the equality 
	$\sum_{\phi} \im(\phi) = I_{S/R}$ follows 
	because $\Set{\phi(s) \given \phi \in \Hom_R(S,R), s \in S}$ is a generating
	set for $\sum_{\phi} \im(\phi)$, and each $\phi(s)$ can be rewritten as the image
  of $1$ under the $R$-linear map $S \rightarrow R$ from \eqref{rem:eval-1}.
	
	\item If $T$ is a solid $R$-algebra and the map $R \rightarrow T$ 
	factors through an $R$-algebra $S$, then $S$ is also a solid $R$-algebra.
	Indeed, choose an $R$-linear map $\varphi\colon T \rightarrow R$ such that
  $\varphi(1) \neq 0$. Then, the composition 
  $S \rightarrow T \overset{\varphi}{\rightarrow} R$ is a nonzero $R$-linear
  map.\label{rem:useful-solid-fact-composition}
\end{enumerate}
\end{remark}

\begin{example}
\label{ex:finite-solid}
Suppose $R \hookrightarrow S$ is a finite extension of rings (not necessarily
Noetherian) such that $R$ is reduced and has finitely many minimal
primes. 
Then, $S$ is a solid $R$-algebra. Indeed, if $K$
is the total ring of fractions of $R$ and $L \coloneqq K \otimes_R S$,
then one can choose
a \emph{splitting} $L \rightarrow K$ and restrict it to $S$, giving a
nonzero $R$-linear map $S \rightarrow K$. Note such a splitting
exists because $K$ is a finite direct product of fields by
\cite[\href{https://stacks.math.columbia.edu/tag/02LX}{Tag 02LX}]{stacks-project}
and the fact that the nonzerodivisors of a reduced ring are precisely the 
elements contained in minimal primes.
Since $S$ is finitely generated as an $R$-module,
the image of $S \rightarrow K$ is contained in the $R$-submodule
of $K$ generated by $1/f$, for some nonzerodivisor $f \in R$. Then, the 
composition
\[
  S \longrightarrow K \xrightarrow{-\cdot f} K
\]
is 
an $R$-linear map that sends $1$ to $f \neq 0$, and 
whose image lands inside $R$. Thus, $S$ is a solid $R$-algebra.
\end{example}

It is natural to ask how the notion of solidity behaves under compositions.
While we do not know the answer to this question in general, the following
special case will be useful in the sequel.

\begin{lemma}
\label{lem:solid-composition}
Let $R \rightarrow S$ be an extension of domains such that $\Frac(S)$
is an algebraic extension of $\Frac(R)$. If $S$ is a solid $R$-algebra
and $T$ is a solid $S$-algebra, then $T$ is a solid $R$-algebra.
\end{lemma}

\begin{proof}
Since $T$ is $S$-solid, choose an $S$-linear map $\varphi\colon T \rightarrow S$
such that $b \coloneqq \varphi(1) \neq 0$. By our hypothesis on the extension
$R \rightarrow S$, it follows by \cite[Prop.\ 3.7]{DS18} that the canonical map
\[
S \longrightarrow \Hom_R\bigl(\Hom_R(S,R),R\bigr)
\]
is injective. Since $b$ is a nonzero element of $S$, this means that we can find
some $\phi \in \Hom_R(S,R)$ such that $\phi(b) \neq 0$. Then the composition
\[
  T \overset{\varphi}{\longrightarrow} S \overset{\phi}{\longrightarrow} R
\]
is a nonzero $R$-linear map because $\phi \circ \varphi(1) = \phi(b) \neq 0$. 
In other words, $T$ is a solid $R$-algebra.
\end{proof}

\begin{remark}
Analyzing the proof of \cite[Prop.\ 3.7]{DS18}, one can show that the injectivity
of the canonical map $S \rightarrow \Hom_R(\Hom_R(S,R),R)$ holds as long as $S$ is 
$R$-solid and the extension $R \rightarrow S$ has the property that for any
nonzero ideal $J$ of $S$, the contraction $J \cap R$ is also a nonzero ideal of $R$.
Thus, solidity is preserved under composition as long as $R \rightarrow S$ is an
extension of domains that satisfies this ideal contraction property.
\end{remark}

\subsection{\textit{F}-solidity}
We now define and study a special case of a solid algebra in prime 
characteristic, which we call $F$-solidity.

\begin{definition}
\label{def:F-solidity}
Let $R$ be a ring of prime characteristic $p > 0$. We say $R$ is \emph{$F$-solid} 
if $F_{R*}R$ is a solid $R$-algebra.
\end{definition}

In other words, $R$ is $F$-solid precisely when it admits a nonzero
$p^{-1}$-linear map. We make 
a few preliminary observations about $F$-solidity.

\begin{remark}
\label{rem:F-solidity-obs}
Let $R$ be a ring of prime characteristic $p > 0$.
{\*}
\begin{enumerate}[label=$(\arabic*)$,ref=\arabic*]
\item\label{rem:Fsplit-Fsolid} If $R$ is Frobenius split, then $R$ is $F$-solid. 
This does not need $R$ to be Noetherian.

\item\label{rem:Ffinite-Fsolid} If $R$ is a reduced $F$-finite Noetherian ring, then 
$R$ is $F$-solid. Indeed, let 
$K$ be the total quotient ring of $R$. Then, $K$ is a finite direct product of 
$F$-finite fields. Thus, $\Hom_K(F_{K*}K, K) \neq 0$. But since $F_{R*}R$ is a 
finitely presented $R$-module and $K$ is a flat $R$-module, we have 
\[
\Hom_K(F_{K*}K,K) = \Hom_R(F_{R*}R,R) \otimes_R K.
\]
This implies that $\Hom_R(F_{R*}R,R) \neq 0$, or equivalently, that $R$ is $F$-solid.

\item\label{rem:Fsolid-all-e} If $R$ is a domain, then $R$ is $F$-solid if and only if for all 
(equivalently, for some) $e > 0$, $F^e_{R*}R$ is a solid $R$-module. This is 
well-known, but we provide a brief justification. Note that since $F^e_{R*}R$ is
always an $F_{R*}R$-algebra for any $e > 0$, $R$-solidity of $F^e_{R*}R$ implies
$R$-solidity of $F_{R*}R$ by Remark
\ref {rem:useful-solid-facts}$(\ref {rem:useful-solid-fact-composition})$.
Conversely, suppose $F_{R*}R$ is $R$-solid and let
$\phi\colon F_{R*}R \rightarrow R$ be a map such that $c \coloneqq \phi(1) \neq 0$. 
Assume by
induction that there exists a nonzero $R$-linear map 
$\varphi\colon F^e_{R*}R \rightarrow R$
such that $\varphi(1) \neq 0$. Then, the composition
\[
F^{e+1}_{R*}R \xrightarrow{F^e_{R*}\phi} F^e_{R*}R
\overset{\varphi}{\longrightarrow} R
\]
maps $c^{(p^e-1)p}$ to $c\varphi(1) \neq 0$, showing that
$F^{e+1}_{R*}R$ is $R$-solid and completing the proof by induction.
Note that we use $R$ is a domain in the step where
we conclude $c\varphi(1) \neq 0$ from the fact that both $c$ and $\varphi(1)$ 
are nonzero.

\item\label{rem:Fsolid-finite-ext} Suppose $R \rightarrow S$ is a finite 
extension of Noetherian domains. 
Then, $R$ is $F$-solid if and only if $S$ is $F$-solid. 
For the forward implication, since $R$ is Noetherian,
by \cite[Cor.\ 2.3]{Hoc94} it suffices to show that $F_{S*}S$ is a 
solid $R$-algebra. We will show that $F^e_{S*}S$ is $R$-solid for any $e > 0$. 
Consider the composition $R \rightarrow F^e_{R*}R \rightarrow F^e_{S*}S$,
where the first map is the $e$-th iterate of the Frobenius on $R$
and the second map is just restriction of scalars of the finite extension 
$R \rightarrow S$. Since $F^e_{S*}S$ is a finite $F^e_{R*}R$-algebra, we know
that $F^e_{S*}S$ is a solid $F^e_{R*}R$-algebra by Example \ref{ex:finite-solid}. 
Moreover, $R \rightarrow F^e_{R*}R$ is generically integral and 
$F^e_{R*}R$ is a solid $R$-algebra by 
$(\ref{rem:Fsolid-all-e})$ since $R$ is $F$-solid.
Therefore by Lemma \ref{lem:solid-composition} we conclude that
$F^e_{S*}S$ is a solid $R$-algebra, proving the forward implication.
Conversely, suppose $S$ is $F$-solid.
Since $R \rightarrow S$ is a finite extension, $S$ is $R$-solid
by Example \ref{ex:finite-solid}. Then $F_{S*}S$ is a solid $R$-algebra
by Lemma \ref{lem:solid-composition}. Since 
\[R \longrightarrow S \overset{F_S}{\longrightarrow} F_{S*}S\]
factors through $F_{R*}R$, it follows that $F_{R*}R$ is a solid $R$-algebra
(Remark \ref{rem:useful-solid-facts}$(\ref {rem:useful-solid-fact-composition})$),
that is, $R$ is $F$-solid.

\item\label{rem:Fsolid-Noether} Let $R$ be a Noetherian domain. 
Since polynomial algebras or power series rings over fields are
$F$-solid (they are even Frobenius split), it follows by part
$(\ref{rem:Fsolid-all-e})$ of this Remark
that $R$ is $F$-solid either when $R$ is of finite type over a field 
(via Noether normalization), or when $R$ is a complete local domain 
(by Cohen's structure theorem). If $R$ is finite type extension of an $F$-solid
domain $A$, then similar reasoning shows that there exists a
nonzero element $a \in A$ such that $R_a$ is $F$-solid by
\cite[\href{https://stacks.math.columbia.edu/tag/07NA}{Tag 07NA}]{stacks-project}.
Here, we use the fact that $F$-solidity of $A$ is preserved under localization,
and also that a polynomial ring over an $F$-solid ring is $F$-solid. However,
it is unclear whether a ring $R$ essentially of finite type over a complete local domain
is $F$-solid using this line of reasoning. We will show that such rings are 
indeed $F$-solid via
the gamma construction \ref{constr:gamma} (see Theorem \ref{thm:F-solid-eft}
below).

\item\label{rem:Fsolid-localization} If $R$ is an $F$-solid domain, then any localization of 
$R$ is also $F$-solid. This is a simple application of the fact that
for a multiplicative set $S \subseteq R$, we have $S^{-1}R \otimes_R F_{R*}R 
\simeq F_{S^{-1}R*}S^{-1}R$. However,
$F$-solidity is not a local property,
in the sense that if $R$ is a ring of prime characteristic whose local rings are all $F$-solid, it is not necessarily the case that $R$ is $F$-solid.
This follows 
from the failure of $F$-finiteness being a local property (see Example 
\ref{ex:excellence}$(\ref {ex:F-finite-not-local} )$). Indeed, choose a Noetherian domain 
$R$ of prime characteristic $p > 0$ such that $R_\fp$ is $F$-finite for all
prime ideals $\fp$, but $R$ itself is not as in \cite{DI}. Then the fraction field
of $R$ is $F$-finite, and so, if $R$ was $F$-solid, then \cite[Thm.\ 3.2]{DS18} would imply that $R$ is $F$-finite, a contradiction.
\end{enumerate}
\end{remark}

\subsection{\textit{F}-solidity for rings of finite type over complete local rings}
\par In Remark \ref {rem:F-solidity-obs}$(\ref {rem:Fsolid-Noether} )$ we
observed that rings to which one can apply a Noether normalization type 
result (such as finite type algebras over fields or complete local rings) are
often $F$-solid in prime characteristic. However, this normalization 
technique does not yield $F$-solidity for
algebras that are essentially of finite type over a complete local ring, which, as we saw in Section \ref {sec:F-purity-and-splitting}, often behave like 
Noetherian $F$-finite rings. Drawing inspiration from the $F$-solidity of 
reduced $F$-finite Noetherian rings, the main result of this section shows 
that $F$-solidity often holds even in a non-$F$-finite setting.

\begin{theorem}
\label{thm:F-solid-eft}
Let $S$ be a reduced ring which is essentially of finite type over a complete 
local ring $R$ of prime characteristic $p > 0$. Then, $S$ is $F$-solid.
\end{theorem}

\par The proof of this result proceeds by reducing to the $F$-finite case via an 
argument similar to the one in Theorem \ref{thm:splittingwithgamma}, and by using a 
descent result on $F$-solidity which we establish first.

\begin{lemma}
\label{lem:descent-F-solidity}
Let $\varphi\colon R \rightarrow S$ be an injective homomorphism of rings of prime 
characteristic $p > 0$, and suppose $S$ is a solid $R$-algebra.
\begin{enumerate}[label=$(\roman*)$,ref=\roman*]
	\item\label{lem:descent-algebraic} 
  If $R$ and $S$ are domains, $\varphi$ is injective, and the induced 
	map $\Frac(R) \rightarrow \Frac(S)$ is algebraic, then $R$ is $F$-solid.
	\item\label{lem:descent-inseparable} If $\varphi$ is purely inseparable, 
  $S$ is reduced, and there exists 
	an $R$-linear map $g\colon S \rightarrow R$ such that $g(1)$ is a nonzerodivisor, 
	then $R$ is $F$-solid.
\end{enumerate}
\end{lemma}

\begin{proof}
Throughout this proof, let $\phi\colon F_{S*}S \rightarrow S$ be an $S$-linear map such that
\[ 
\phi(1) \eqqcolon c \neq 0.
\]

\par $(\ref{lem:descent-algebraic})$ Since $\Frac(R) \hookrightarrow \Frac(S)$ is algebraic, 
clearing denominators 
from an equation of algebraic dependence of $c$ over $\Frac(R)$, it follows 
that there exists a nonzero $s \in S$ such that $sc \in R$. Let 
$\ell_s\colon S \rightarrow S$ denote left multiplication by $s$. Then 
$\ell_s \circ \phi$ maps $1$ to $sc \in R \smallsetminus \{0\}$. Since $S$ is a solid 
$R$-algebra, choose an $R$-linear map 
$f\colon S \rightarrow R$
such that $f(1) \neq 0$. Then the composition
\[
F_{R*}R \xrightarrow{F_*\varphi} F_{S*}S \xrightarrow{\ell_s \circ \phi} S
\overset{f}{\longrightarrow} R
\]
maps $1$ to $scf(1)$, and the latter is a nonzero element of $R$ because 
the elements $sc$ and $f(1)$ are nonzero, and $R$ is a domain. This shows $R$ is $F$-solid.

\par $(\ref{lem:descent-inseparable})$ We may assume without loss of generality 
that $\varphi$ is an inclusion. Since 
$\varphi$ is purely inseparable, choose $e > 0$ such that $c^{p^e} \in R$. Note 
that $c^{p^e} \neq 0$ because $S$ is reduced and $c$ is a nonzero element by choice. 
Let $\ell_{c^{p^e -1}}$ denote left multiplication by $c^{p^e-1}$. The composition
\[
F_{R*}R \hooklongrightarrow F_{S*}S \xrightarrow {\ell_{c^{p^e-1}} \circ \phi} S
\overset{g}{\longrightarrow} R
\]
maps $1$ to $c^{p^e}g(1)$, which is a nonzero element of $R$ because $g(1)$ 
is a nonzerodivisor on $R$ by hypothesis. This completes the proof.
\end{proof}

\par We can now prove Theorem \ref{thm:F-solid-eft}.

\begin{proof}[Proof of Theorem \ref{thm:F-solid-eft}]
Suppose $S$ is essentially of finite type over the complete Noetherian local ring 
$(R,\fm)$. Again, using the gamma construction \ref{constr:gamma}, there exists a faithfully flat 
purely inseparable ring extension $R \hookrightarrow R^{\Gamma}$ such that 
$S^{\Gamma} \coloneqq S \otimes_R R^{\Gamma}$ is $F$-finite and reduced 
(Lemma \ref{lem:gammafacts}$(\ref{lem:gammared})$). In particular, $S^\Gamma$ is 
then $F$-solid by Remark \ref{rem:F-solidity-obs}$(\ref{rem:Ffinite-Fsolid})$. Since $R$ is complete, 
the pure extension $R \hookrightarrow R^{\Gamma}$ splits (Lemma \ref{lem:fed12}), 
hence so does the purely inseparable extension 
$S \hookrightarrow S^{\Gamma}$. By $F$-solidity of $S^\Gamma$ and 
an application of Lemma \ref{lem:descent-F-solidity}$(\ref{lem:descent-inseparable})$, 
it follows that $S$ is also $F$-solid.
\end{proof}

\begin{remark}
In order to establish $F$-solidity of domains that are essentially of finite type over
any field of prime characteristic, one can avoid the gamma construction approach
of Theorem \ref{thm:F-solid-eft}. Namely, suppose $k$ is a field of characteristic
$p > 0$ and that $R$ is a domain that is essentially of finite type over $k$.
Let $T$ be a finite type $k$-algebra and $S \subset T$ a multiplicative set
such that $R = S^{-1}T$. Note that $T$ can be chosen to be a domain. Indeed,
let $P$ be a prime ideal of $T$ disjoint from $S$ such that $S^{-1}P = (0)$
(such a $P$ exists because $R$ is a domain). Since $P$ is finitely generated,
there exists some $f \in S$ such that $PT_f = (0)$. Upon replacing $T$ by $T_f$,
we can assume $T$ is a domain. Then $T$ is $F$-solid by 
Remark \ref{rem:F-solidity-obs}$(\ref{rem:Fsolid-Noether})$, and consequently,
$R$ is $F$-solid by Remark \ref{rem:F-solidity-obs}$(\ref{rem:Fsolid-localization})$.
This argument provides an alternate proof of Frobenius splitting of divisorial
valuation rings (Corollary \ref{cor:divisorial-Fsplit}). Namely, if 
$R$ is a divisorial valuation ring, then $R$ is essentially of finite
type over a field and hence $F$-solid. Then $R$ is Frobenius split
by Proposition \ref{thm:split-F-regular-DVR}.
\end{remark}

\subsection{Excellent rings are not always \textit{F}-solid}
A natural question is whether Theorem \ref{thm:F-solid-eft} can be
generalized to reduced excellent rings, or at least, to reduced
excellent rings that are of finite type over an an excellent local
ring. This turns out to be false, as shown in the next result. 
These are the first excellent
examples of prime characteristic domains that are not $F$-solid.

\begin{proposition}
\label{prop:non-F-solid}
There exists a complete non-Archimedean field $(k,\abs{})$ of characteristic
$p > 0$ such that for each integer $n > 0$, we have the following:
\begin{enumerate}[label=$(\roman*)$,ref=\roman*]
	\item\label{prop:Tate-not-Fsolid} The Tate algebra $T_n(k)$ is an 
	excellent regular ring of Krull dimension $n$ that is not $F$-solid.
	\item\label{prop:affinoid-not-Fsolid} Any $k$-affinoid domain of 
	Krull dimension $n$ is not $F$-solid.
	\item\label{prop:convergent-not-Fsolid} 
	The convergent power series ring $K_n(k)$ is an excellent
	Henselian regular local ring of Krull dimension $n$ that is
	not $F$-solid.
\end{enumerate}
\end{proposition}

Here by a \emph{$k$-affinoid} domain we mean an integral domain that is
the quotient of some Tate algebra.
These are the rigid analytic analogues of coordinate rings of
affine varieties.

\begin{proof}
$(\ref{prop:Tate-not-Fsolid})$ follows from \cite[Thm.\ A]{DM} 
and $(\ref{prop:convergent-not-Fsolid})$ follows from \cite[Rem.\ 5.5]{DM}
(see also \cite[Thm.\ 4.4]{DM}). 

\par For $(\ref{prop:affinoid-not-Fsolid})$,
choose a complete non-Archimedean field $(k,\abs{})$ for which 
$(\ref{prop:Tate-not-Fsolid})$ holds, and suppose that $A$
is an $k$-affinoid domain of Krull dimension $n > 0$. Then by
the rigid analytic analogue of Noether normalization
\cite[Cor.\ 2.2/11]{Bos14}, there exists a module-finite ring extension
$T_n(k) \hookrightarrow A$. Since $T_n(k)$ is not $F$-solid by 
$(\ref{prop:Tate-not-Fsolid})$, it follows that $A$ is not $F$-solid
by Remark \ref{rem:F-solidity-obs}$(\ref{rem:Fsolid-finite-ext})$.
\end{proof}

\begin{remark}
  \leavevmode
\begin{enumerate}
	\item Examples of non-$F$-solid regular local rings can be constructed 	even in
	the function field of $\mathbf{P}^2_{\overline{\FF_p}}$. But any such
  ring will not be excellent; see \cite[\S4.1]{DS18}.

	\item Let $k$ be as in Proposition \ref {prop:non-F-solid}. Then the 
	completion of the local ring $K_1(k)$ at its maximal ideal $(X)$ is 
	the power series ring $k\llbracket X \rrbracket$, which is $F$-solid. This illustrates that
	$F$-solidity does not descend over faithfully flat maps. The same example
	also shows that $F$-solidity does not ascend over regular maps,
	since $\FF_p \rightarrow K_1(k)$ is regular, $\FF_p$ is $F$-solid and
	$K_1(k)$ is not.
\end{enumerate}
\end{remark}

Despite the examples obtained in Proposition \ref{prop:non-F-solid},
affinoid domains of prime characteristic are often $F$-solid.

\begin{proposition}
\label{prop:some-Fsolid-affinoids}
Let $(k,\abs{})$ be a complete non-Archimedean field that satisfies
any of the following additional properties:
\begin{enumerate}[label=$(\roman*)$,ref=\roman*]
	\item\label{prop:F-split-usually-spher} $(k,\abs{})$ is spherically complete.
	\item\label{prop:F-split-usually-count} $k^{1/p}$ has a dense $k$-subspace $V$ 
	which has a countable $k$-basis, hence in particular
	if $[k^{1/p}:k] < \infty$.
	\item\label{prop:F-split-polar} $\abs{k^\times}$ is not discrete,
	and the norm on $k^{1/p}$ is polar with respect to the norm on $k$.
\end{enumerate}
Then any $k$-affinoid domain is $F$-solid.
\end{proposition}

\begin{proof}
Let $A$ be a $k$-affinoid domain. If $A$ has Krull dimension zero, then
$A = k$ and there is nothing to prove since any field of prime
characteristic is $F$-solid. If $A$ has Krull dimension $n > 0$,
then $A$ is a module-finite extension of $T_n(k)$ as in the proof of Proposition
\ref{prop:non-F-solid}$(\ref{prop:affinoid-not-Fsolid})$. When $k$ satisfies
any of the conditions of this Proposition, then $T_n(k)$
is Frobenius split by \cite[Cor.\ D]{DM}, hence also $F$-solid. 
Then $A$ is $F$-solid
by Remark \ref{rem:F-solidity-obs}$(\ref{rem:Fsolid-finite-ext})$.
\end{proof}

\begin{remark}
The observant reader will notice that the hypothesis of Proposition 
\ref{prop:some-Fsolid-affinoids}$(\ref{prop:F-split-polar})$ is 
weaker than the hypothesis of Proposition 
\ref{prop:split-F-regular-Tate}$(\ref{cor:F-split-polar})$.
This is because to show $F$-solidity of $k$-affinoid domains one only
needs Tate algebras to be $F$-solid, which is a 
weaker requirement than split $F$-regularity.
\end{remark}

\section{Some obstructions to solidity}
In the previous section, we showed that a reduced ring $R$ which is essentially of finite 
type over a complete local Noetherian ring of prime characteristic $p > 0$ is
$F$-solid (Theorem \ref{thm:F-solid-eft}). 
The proof proceeds by constructing an $F$-finite 
solid Noetherian $R$-algebra $R^\Gamma$, and then using non-trivial $p^{-e}$-linear 
maps on $R^\Gamma$ and its solidity as an $R$-algebra to produce non-trivial 
$p^{-e}$-linear maps on $R$. At the same time, we also exhibited examples of 
excellent Henselian regular local rings that are not $F$-solid (Proposition 
\ref{prop:non-F-solid}). The goal of this section is to highlight some obstructions
to $F$-solidity of excellent local rings which will more systematically
explain why the gamma construction
has little hope of producing examples of $F$-solid excellent rings
beyond those that are essentially of finite type over a complete local ring.

\subsection{Henselizations and completions are not solid} 
Suppose $S$ is a reduced ring which is essentially of finite type over an excellent 
local ring $(R,\fm)$ of prime characteristic $p > 0$ that is not necessarily complete. 
It is natural to wonder if nonzero $p^{-e}$-linear maps on the change of base ring 
$S_{\widehat{R}} \coloneqq \widehat{R} \otimes_R S$ can help produce nonzero 
$p^{-e}$-linear maps on $R$. As we now demonstrate, the main difficulty with this approach 
is that $S_{\widehat{R}}$ is rarely a solid $S$-algebra. Thus, descent type arguments 
in the spirit of Lemma \ref{lem:descent-F-solidity}, crucial in the proof of Theorem 
\ref{thm:F-solid-eft}, will almost never work. 

The following simple 
observation will be useful: if $R \rightarrow S \rightarrow T$ are ring maps such 
that $T$ is a solid $R$-algebra, then $S$ is also a solid $R$-algebra (Remark
\ref{rem:useful-solid-facts}$(\ref{rem:useful-solid-fact-composition})$). Said differently, 
if $S$ is not a solid $R$-algebra, then $T$ is not a solid $R$-algebra. We will use 
this observation to show that the completion of a local ring is rarely solid. 
In fact, we first prove the following stronger result for Noetherian local rings 
that are not Henselian.

\begin{theorem}
\label{thm:Henselization-not-solid}
Let $(R,\fm)$ be a Noetherian local ring of arbitrary characteristic such that 
the Henselization $R^h$ is a domain (for example, if $R$ is normal). 
Assume $R$ is not Henselian. Then, we have the following:
\begin{enumerate}[label=$(\roman*)$,ref=\roman*]
	\item\label{thm:fet-solid} If $\varphi\colon R \rightarrow S$ is an 	essentially \'etale extension of 
	domains, then $S$ is a solid $R$-algebra if and only if $\varphi$ is finite 	\'etale.

	\item\label{thm:Henselization-notsolid} The Henselization $R^h$ is not a 	solid $R$-algebra.

	\item\label{thm:any-Henselian-notsolid} 
	If $(R,\fm) \rightarrow (S,\fn)$ is a local ring homomorphism such that 	$S$ is
	Henselian (for example, if $S = R^{sh}$ or 
	$S = \widehat{R}$), then $S$ is not solid $R$-algebra.
\end{enumerate}
\end{theorem}

\par The main ingredient in the proof of Theorem \ref{thm:Henselization-not-solid} is 
the following observation of Smith and the first author:

\begin{citedprop}[{\cite[Prop.\ 3.7]{DS18}}]
\label{prop:generically-algebra-maps}
Let $\varphi\colon R \rightarrow S$ be an injective homomorphism of Noetherian domains 
which is generically finite. If $S$ is a solid $R$-algebra, then $\varphi$ 
is a finite map.
\end{citedprop}

\begin{proof}[Proof of Theorem \ref{thm:Henselization-not-solid}]
$(\ref{thm:any-Henselian-notsolid})$ follows from $(\ref{thm:Henselization-notsolid})$
by Remark \ref{rem:useful-solid-facts}$(\ref{rem:useful-solid-fact-composition})$, since the local map $R \rightarrow S$ factors through $R^h$ by the universal
property of Henselization. 

\par For $(\ref{thm:Henselization-notsolid})$, 
choose an essentially \'etale local homomorphism 
$\phi\colon (R,\fm) \rightarrow (S,\fn)$ such that the induced map on residue fields 
is an isomorphism and $\phi$ is not an isomorphism. Note that such an $S$ exists 
because $R$ is not Henselian. We claim that $S$ is not a finite $R$-algebra. 
Indeed, otherwise $S$ is a free $R$-module of finite rank because $\phi$ is flat 
and $R$ is Noetherian local. The free rank of $S$ must then equal $1$ 
because $\kappa(\fn) = S \otimes_R \kappa(\fm)$ and the extension 
$\kappa(\fm) \hookrightarrow \kappa(\fn)$ is trivial. But this means $\phi$ is an 
isomorphism, contrary to its choice. Thus $S$ is not a finite $R$-algebra.
Since $R^h$ is also the Henselization of $S$, it follows that $S$ is also a domain. 
Then $S$ is not a solid $R$-algebra by $(\ref{thm:fet-solid})$, and so, 
$R^h$ also cannot be a solid $R$-algebra. This proves 
$(\ref{thm:Henselization-notsolid})$. 

\par Thus, it remains to show $(\ref{thm:fet-solid})$. 
If $\varphi\colon R \rightarrow S$ is an essentially \'etale extension of domains, 
then the generic fiber is a finite separable extension of fields. If $S$ is a 
solid $R$-algebra, then $S$ must be module-finite over $R$ 
by Proposition \ref{prop:generically-algebra-maps}. 
Conversely, if $S$ is a finite $R$-algebra, then
$S$ is $R$-solid by Example \ref{ex:finite-solid} regardless of whether
$\varphi$ is \'etale.
\end{proof}

If $(R,\fm)$ is a Henselian Noetherian local domain that is not complete, then 
Theorem \ref{thm:Henselization-not-solid} gives no information on whether 
$\widehat{R}$ is a solid $R$-algebra. The goal in the remainder of this 
subsection is to give a more direct proof of the following result which, as a 
special case, implies that the completion of a Noetherian local domain, 
Henselian or not, is never solid.

\begin{theorem}
\label{thm:completion-not-solid}
Let $R$ be a Noetherian ring (not necessarily local) and let $I$ be a nonzero ideal 
of $R$ such that $R$ is $I$-adically separated. Denote by $\widehat{R}^I$ the 
$I$-adic completion of $R$. Consider an element $\varphi \in \Hom_R(\widehat{R}^I, R)$. 
Then, we have the following:
\begin{enumerate}[label=$(\roman*)$,ref=\roman*]
	\item\label{thm:image-1} $\varphi \neq 0$ if and only if $\varphi(1) 	\neq 0$.
	\item\label{thm:zerodiv} If $R$ is not $I$-adically complete and $\varphi 	\neq 0$, then $\varphi(1)$ is a zerodivisor on $R$.
	\item\label{thm:notsplit} If $R$ is not $I$-adically complete, the 	canonical map 
	$i\colon R \rightarrow \widehat{R}^I$ never splits.
	\item\label{thm:notsolid} If $R$ is a domain that is not $I$-adically 	complete, then 
	$\widehat{R}^I$ is not a solid $R$-algebra.
\end{enumerate}
\end{theorem}

\par For the proof of Theorem \ref{thm:completion-not-solid}, we first make a few 
preliminary observations about $I$-adic completions. Recall that for an ideal 
$I$ of a ring $R$, we say that an $R$-module $M$ is \emph{$I$-adically separated} 
if $\bigcap_{n \in \NN} I^nM = (0)$. Said differently, a module $M$
is $I$-adically separated if the $I$-adic topology on $M$ is Hausdorff.

\begin{lemma}
\label{lem:completion-map-basic-properties}
Let $I$ be an ideal of a Noetherian ring $R$. The canonical map
\[
i\colon R \longrightarrow \widehat{R}^I
\]
satisfies the following properties:
\begin{enumerate}[label=$(\roman*)$,ref=\roman*]
	\item\label{lem:dense} For all integers $n > 0$, we have 
	$\widehat{R}^I = i(R) + I^n\widehat{R}^I$.

	\item\label{lem:flat} The map $i$ is flat, and $i$ is faithfully flat if and
	only if $I$  is contained in the Jacobson radical of $R$.

	\item\label{lem:finite} The map $i\colon R \rightarrow \widehat{R}$ is
 	finite if and only if $i$ is surjective. In particular, if $R$ is $I$-adically
	separated but not $I$-adically complete, then $i$ is not a finite map.

	\item\label{lem:precomposition} If $M$ is an $I$-adically separated 
	$R$-module, then the map 
	\[
	i^*\colon \Hom_R(\widehat{R}^I, M) \longrightarrow \Hom_R(R, M)
	\]
	induced by pre-composition with $i$ is injective. 

	\item\label{lem:image1} For every 
	$\varphi \in \Hom_R(\widehat{R}^I, R)$, we have 
	$\varphi(1)\widehat{R}^I \subseteq i(R)$.
\end{enumerate}
\end{lemma}

\begin{proof}[Proof of Lemma \ref{lem:completion-map-basic-properties}]
$(\ref{lem:dense})$ follows from the isomorphisms $R/I^n \simeq \widehat{R}^I/I^n\widehat{R}^I$
for all integers $n > 0$. For a proof of $(\ref{lem:flat})$ see 
\cite[Thms.\ 8.8 and 8.14]{Mat89}.

\par $(\ref{lem:finite})$ If $i$ is surjective, then it is a finite map. 
Conversely, assume $i$ is finite. Tensoring the exact sequence
\[
  R \overset{i}{\longrightarrow} \widehat{R}^I \longrightarrow \coker(i) \longrightarrow 0
\]
by $R/I$, we get an exact sequence
\[
R/I \longrightarrow \widehat{R}^I/I\widehat{R}^I \longrightarrow \coker(i)/I\coker(i) \longrightarrow 0.
\]
As the induced map $R/I \rightarrow \widehat{R}^I/I\widehat{R}^I$ is an 
isomorphism, it follows that
\[
\coker(i) = I\coker(i).
\]
Since $i$ is finite, $\coker(i)$ is a finitely generated $R$-module, and 
consequently, $\coker(i)$ is annihilated by an element of the form $1 + a$
for some $a \in I$ by Nakayama's lemma \cite[Thm.\ 2.2]{Mat89}. This means 
that the ideal of $\widehat{R}^I$ generated by $1 + a$ is contained in $i(R)$. 
However, $1 + a$ is a unit in $\widehat{R}^I$ by 
\cite[Ch.\ III, \S2, n\textsuperscript{o} 13, Lem.\ 3]{BouCA}. Therefore
\[
\widehat{R}^I = (1+a)\widehat{R}^I \subseteq i(R) \subseteq \widehat{R}^I,
\]
which implies that the map $i\colon R \rightarrow \widehat{R}^I$ is surjective.
For the second assertion of $(\ref{lem:finite})$, if $A$ is $I$-adically 
separated, then  
$i$ is injective. If $i$ is also finite, then $i$ is surjective by what we just proved, 
which would imply $i$ is an isomorphism. But this is impossible because 
$A$ is not $I$-adically complete by hypothesis.

\par$(\ref{lem:precomposition})$ Let $\varphi\colon \widehat{R}^I \rightarrow M$ 
be an $R$-linear map such that 
$\varphi \circ i = 0$.
It suffices to show that $\varphi = 0$. Let $x \in \widehat{R}$. By $(\ref{lem:dense})$, for 
every integer $n > 0$, there exists $a_n \in R$ and $x_n \in I^n\widehat{R}^I$ such that
\[x = i(a_n) + x_n.\]
Then,
\[\varphi(x) = \varphi\bigl(i(a_n)\bigr) + \varphi(x_n) = \varphi(x_n),\]
and $\varphi(x_n) \in I^nM$ by $R$-linearity. This shows that
$$\varphi(x) \in \bigcap_{n \in \ZZ_{> 0}} I^nM = (0),$$
where the last equality follows because $M$ is $I$-adically separated. 
Thus, $\varphi = 0$.

\par $(\ref{lem:image1})$ Consider the $R$-linear map 
\[
i \circ \varphi\colon \widehat{R}^I \longrightarrow \widehat{R}^I,
\]
and let 
$a \coloneqq \varphi(1) \in R$.
We also have the $R$-linear map $\ell_a\colon \widehat{R}^I \rightarrow \widehat{R}^I$ 
given by left multiplication by $a$. Now for any $x \in R$, we have
\[
\ell_a\bigl(i(x)\bigr) = ai(x) = i(ax) = i\bigl(\varphi(1)x\bigr) = i \circ
\varphi \bigl(i(x)\bigr),
\]
that is,
\[
i^*(\ell_a) = i^*(i \circ \varphi),
\]
where $i^*\colon \Hom_R(\widehat{R}^I, \widehat{R}^I) \rightarrow \Hom_R(R, \widehat{R}^I)$ 
is the map induced by pre-composition with $i$. Since $\widehat{R}^I$ is an 
$I$-adically separated $R$-module, it follows from $(\ref{lem:precomposition})$ 
that $i^*$ is injective. Thus $\ell_a = i \circ \varphi$,
and so,
\[
  \varphi(1)\widehat{R}^I = \ell_a(\widehat{R}^I) = i \circ \varphi 
  (\widehat{R}^I) \subseteq i(R).\qedhere
\]
\end{proof}

\begin{remark}
  \leavevmode
\begin{enumerate}
\item If $I$ is an idempotent ideal of a Noetherian ring $R$, then 
$\widehat{R}^I = R/I$ and the completion map $R \rightarrow \widehat{R}^I$ 
is just the quotient map. In particular, the completion map is finite. 
Hence the second assertion of Lemma 
\ref{lem:completion-map-basic-properties}$(\ref{lem:finite})$ 
fails if $R$ is not $I$-adically separated.

\item The proofs of parts $(\ref{lem:dense}), (\ref{lem:finite}), 
  (\ref{lem:precomposition})$, and $(\ref{lem:image1})$ of Lemma 
\ref{lem:completion-map-basic-properties} work even when $R$ is not 
Noetherian as long as $I$ is a finitely generated ideal of $R$. We need 
$I$ to be finitely generated in order for the isomorphisms 
$R/I^n \simeq \widehat{R}^I/I^n\widehat{R}^I$ to hold 
(see \cite[\href{https://stacks.math.columbia.edu/tag/05GG}{Tag 05GG}]{stacks-project} 
for a nice proof of this due to Poonen).
\end{enumerate}
\end{remark}

\par We now prove Theorem \ref{thm:completion-not-solid} as follows.

\begin{proof}[Proof of Theorem \ref{thm:completion-not-solid}]
\par $(\ref{thm:image-1})$ follows by 
Lemma \ref{lem:completion-map-basic-properties}$(\ref{lem:precomposition})$ applied to 
$M = R$ because the pullback map
\[
i^*\colon \Hom_R(\widehat{R}^I, R) \longrightarrow \Hom_R(R,R)
\]
is injective. Note that here we need $R$ to be $I$-adically separated, and 
the latter follows by the hypotheses of the Theorem.

\par For $(\ref{thm:zerodiv})$, if $\varphi  \neq 0$, then by
$(\ref{thm:image-1})$, we have
$\varphi(1) \neq 0$. 
Assume that $\varphi(1)$ is a nonzerodivisor on $R$. Then, by flatness of 
$i\colon R \rightarrow \widehat{R}^I$ 
(Lemma \ref{lem:completion-map-basic-properties}$(\ref{lem:flat})$), we see that $\varphi(1)$ is 
also a nonzerodivisor on $\widehat{R}^I$. In particular,
\[
  \varphi(1)\widehat{R}^I \simeq \widehat{R}^I.
\]
At the same time, Lemma \ref{lem:completion-map-basic-properties}$(\ref{lem:image1})$ 
shows that
\[
\varphi(1)\widehat{R}^I \subseteq i(R),
\]
and so $\varphi(1)\widehat{R}^I$ is a finitely generated $R$-module 
since it is a submodule of the finitely generated $R$-module $i(R)$ 
and $R$ is Noetherian. Thus, $\widehat{R}^I$ is a finitely generated $R$-module, 
which is impossible by Lemma \ref{lem:completion-map-basic-properties}$(\ref{lem:finite})$
and the hypotheses of $(\ref{thm:zerodiv})$.

$(\ref{thm:notsplit})$ follows from $(\ref{thm:zerodiv})$ because a splitting 
of $R \rightarrow \widehat{R}^I$ 
maps $1$ to $1$, and $1$ is a nonzerodivisor. Similarly, $(\ref{thm:notsolid})$ 
also follows 
from $(\ref{thm:zerodiv})$. Indeed, if $R$ is a domain and $\widehat{R}^I$ is $R$-solid, 
then there exists an $R$-linear map $\phi\colon\widehat{R}^I \rightarrow R$ such 
that $\phi(1) \neq 0$ (Remark \ref{rem:useful-solid-facts}). 
But this is impossible because $(\ref{thm:zerodiv})$ 
implies $\phi(1)$ must be a zerodivisor and a domain has no nonzero zerodivisors.
\end{proof}

\begin{remark}
Let $R$ be a product of two rings $R_1 \times R_2$ and let $I$ be the 
idempotent ideal generated by $(1,0)$. Then, $\widehat{R}^I = R_2$ and the 
associated map $R \rightarrow \widehat{R}^I$ is just projection onto $R_2$. 
In this case $\widehat{R}^I$ is a solid $R$-algebra because one has the 
$R$-linear inclusion $\widehat{R}^I = R_2 \rightarrow R$ that sends $a$ to $(0,a)$.
Thus, Theorem \ref{thm:completion-not-solid}$(\ref{thm:notsolid})$ fails
for arbitrary Noetherian rings $R$.
\end{remark}

\subsection{Non-solidity of some big Cohen--Macaulay algebras}
Apart from providing obstructions to $F$-solidity of rings essentially of 
finite type over an excellent local ring, the results of the previous subsection
have implications for the non-solidity of certain algebra extensions
that are at the heart of the recent solution of Hochster's direct summand
conjecture in mixed characteritic \cite{And18a}.
Throughout this subsection, $(R,\fm)$ will denote a Noetherian local ring of dimension $d$.
Recall that an $R$-algebra $B$ is a \emph{big Cohen--Macaulay $R$-algebra} if there exists a 
system of parameters $x_1,x_2,\dots,x_d \in \fm$ which form a regular sequence on $B$. The adjective 
``big'' is meant to indicate that the $R$-algebra $B$ need not be finite over $R$ (in practice, 
$B$ is not even Noetherian). We will call a big Cohen--Macaulay $R$-algebra a BCM 
$R$-algebra for brevity. We say $B$ is a \emph{balanced BCM $R$-algebra} if every system of 
parameters of $R$ is a regular sequence on $B$. 

The existence of (balanced) BCM $R$-algebras has far-reaching consequences for the 
homological conjectures \cite{HH95} and has been recently used by Ma and Schwede to develop 
a singularity theory in mixed characteristic that draws inspiration from various notions 
of singularities in equal characteristic \cite{MS19} (see also \cite{MSTWW}).
Hochster and Huneke constructed BCM 
$R$-algebras when $R$ has equal characteristic using their characterization of the absolute 
integral closure $R^+$ as a BCM $R$-algebra for an excellent local domain $R$ of prime 
characteristic $p > 0$ \cite{HH93}. BCM $R$-algebras were constructed by Andr\'e in 
mixed characteristic in the same paper where he settled the direct summand conjecture 
\cite{And18a}, following which Shimomoto showed that one can even construct a 
BCM $R$-algebra that is simultaneously an $R^+$-algebra \cite{Shi18}. Andr\'e later 
showed that a (more useful) weakly functorial version of BCM $R$-algebras also exists in 
mixed characteristic \cite{And18b} (see also \cite{HM18}).

One of the pleasing properties of BCM $R$-algebras is that they are often $R$-solid, 
a well-known fact summarized in the following result.

\begin{proposition}
\label{prop:BCM-solid}
Let $(R, \fm, \kappa)$ be a complete local domain of dimension $d$ of arbitrary characteristic. 
Then, every BCM $R$-algebra is $R$-solid.
\end{proposition}

\begin{proof}[Indication of proof]
Let $B$ be a BCM $R$-algebra. Since there exists a system of parameters of $R$ that is a 
regular sequence on $B$, we see that $H^d_\fm(B) \neq 0$. It now follows that $B$ is a solid 
$R$-algebra using \cite[Cor.\ 2.4]{Hoc94}, where the argument proceeds by reducing to the 
case where $R$ is regular via the analogue of Noether normalization for complete local rings 
and then using Matlis duality in a way similar to how it is used in Lemma \ref{lem:fed12}.
\end{proof}

\begin{remark}\label{rem:fsolidityofrplus}
  \leavevmode
\begin{enumerate}[label=$(\arabic*)$,ref=\arabic*]
	\item Proposition \ref{prop:BCM-solid} provides a different perspective 	on $F$-solidity of a complete local domain $(R, \fm)$ in prime 	characteristic.
	Since $R$ is excellent,
  	\cite[Thm.\ 5.15]{HH93} shows that $R^+$ is a big Cohen--Macaulay 	$R$-algebra. Thus, $R^+$ is a 
	solid $R$-algebra by Proposition \ref {prop:BCM-solid}, 
	and since $R \rightarrow R^+	$ factors via $F_{R*}R$, we see
	that $F_{R*}R$ is also a solid $R$-algebra by Remark
	\ref{rem:useful-solid-facts}. 

\medskip 

	\item Now that we know the direct summand theorem, one can 
	give a simple proof of the fact that $R^+$ is always a solid $R$-algebra 
	for a complete local Noetherian domain $(R,\fm)$, even when $R^+$ is 	not known to be 
	BCM (for example, in equal characteristic $0$ and mixed characteristic). 	Note that we do not even need that there exist BCM $R$-algebras that 
	contain $R^+$. Using Cohen's 
	theorem, choose a module finite extension $A \hookrightarrow R$ where
	$A$ is a power series ring over a field or a mixed characteristic complete 	DVR. Then one can check that $A^+ = R^+$. 
	The map $A \rightarrow A^+$ is pure because $A$ is splinter 
	(here we use the direct summand theorem), and so, $A \rightarrow A^+$ 	splits by Lemma \ref{lem:fed12} because $A$ is complete. In other words,
 	the composition $A \hookrightarrow R \rightarrow R^+ = A^+$ is 
	$A$-solid. Then $R^+$ is $R$-solid by
	\cite[Cor.\ 2.3]{Hoc94} because $A \hookrightarrow R$ is a finite 	extension of Noetherian domains.\label{rem:directsummandrplus}
\end{enumerate}
\end{remark}

We now show that we lose solidity of BCM $R$-algebras if we drop the hypothesis that 
$R$ is complete. There exist excellent regular rings that behave like
polynomial or power series rings over fields for which $R^+$ is not $R$-solid. 
In fact, examples exist even for a class of excellent local rings that are 
closest in behavior to complete local rings, namely those that are Henselian.
As far as we are aware, these examples are the first of their kind.

\begin{proposition}
\label{prop:BCM-not-solid}
Let $R$ be a Noetherian ring. We have the following:
\begin{enumerate}[label=$(\roman*)$,ref=\roman*]
	\item\label{prop:hensel-R+-notsolid} 
	For each integer $n > 0$, there exist excellent regular rings $R$ of
	prime characteristic $p > 0$ and Krull dimension $n$ for which $R^+$
	is not a solid $R$-algebra. Moreover, $R$ can be chosen to be local and
	Henselian.

	\item\label{prop:arb-char-R+-not-solid} 
	Let $(R,\fm,\kappa)$ be a Noetherian local domain of arbitrary 	characteristic that 
	is not $\fm$-adically complete. For any BCM $R$-algebra $B$, the 
	$\fm$-adic completion 
	$\widehat{B}^\fm$ is a balanced BCM $R$-algebra that is not $R$-solid.
\end{enumerate}
\end{proposition}

\begin{proof}
$(\ref{prop:hensel-R+-notsolid})$ 
For each $n > 0$, there are Tate algebras $T_n(k)$ and convergent power series
rings $K_n(k)$ that are not $F$-solid \cite[Thm.\ A and Rem.\ 5.5]{DM} for
some appropriately chosen non-Archimedean field $(k,\abs{})$ of characteristic 
$p > 0$. Consequently, the absolute integral closures $R^+$ of such 
excellent regular rings $R$ cannot be $R$-solid. Note that the convergent
power series rings $K_n(k)$ are even Henselian, and both $T_n(k)$ and 
$K_n(k)$ have Krull dimension $n$ (see Proposition \ref{prop:properties}).

\par $(\ref{prop:arb-char-R+-not-solid})$ 
The fact that $\widehat{B}^\fm$ is a balanced BCM $R$-algebra follows from 
\cite[Cor.\ 8.5.3]{BH98}. By construction, there is a factorization 
$R \rightarrow \widehat{R}^\fm \rightarrow \widehat{B}^\fm$. Thus, if 
$\widehat{B}^\fm$ is $R$-solid, then $\widehat{R}^\fm$ is also a solid 
$R$-algebra, which is impossible by Theorem 
\ref{thm:completion-not-solid}$(\ref{thm:notsolid})$.
\end{proof}

\begin{remark}
Let $(R,\fm)$ be a Noetherian local domain of mixed characteristic that is 
not complete. As far as we are aware, all constructions of balanced BCM $R$-algebras
in the literature proceed by first passing to $\widehat{R}^\fm$
and then constructing a balanced BCM $\widehat{R}^\fm$-algebra that is
automatically also a balanced BCM $R$-algebra.
See \cite[n\textsuperscript{o} 4.2]{And18a}, \cite[Proof of Thm.\ 6.3]{Shi18}, and
\cite[n\textsuperscript{o} 17.5.51]{GR}.
None of these BCM $R$-algebras
can be $R$-solid because they contain $\widehat{R}^\fm$.
\end{remark}

\section{Solidity of absolute integral and
perfect closures -- a deeper analysis} 
In the previous subsection, one of our goals was to show that
$R^+$ fails to be a solid $R$-algebra for nice excellent local domains
(Proposition \ref{prop:BCM-not-solid}),
even though $R^+$ is always a solid $R$-algebra when $R$ is a complete local 
domain as a consequence of the direct summand theorem (Remark
\ref{rem:fsolidityofrplus}$(\ref{rem:directsummandrplus})$). 
In this section, we will examine the restrictions that
solidity of $R^+$ imposes on the ring $R$. For example, we will
see that solidity of absolute integral closures often implies excellence in 
prime characteristic (Theorem
\ref{thm:solid-N2}$(\ref{thm:A+-solid-excellent})$), whereas most domains of
equal characteristic zero have solid absolute integral closures because 
of the existence of splittings arising from trace maps (Proposition
\ref{prop:colimits}$(\ref{prop:A+-solid-char0})$). 
In prime characteristic, solidity of absolute integral closures implies
solidity of perfect closures of a domain. Hence we will also spend some
effort understanding when the perfect closure of a Noetherian domain of prime
characteristic is solid. 
We begin our investigation by exhibiting close connections between
solidity of absolute integral and perfect closures and the well-studied
notions of N-1 and Japanese (aka N-2) rings. 

\subsection{N-1 and Japanese rings} N-1 and Japanese rings,
defined below,
are rings for which normalizations satisfy the familiar finiteness properties
of finite type algebras over a field and rings of integers of number fields.

\begin{citeddef}[{\cite[(31.A), Defs.]{Mat80}}]
\label{def:Japanese-Nagata}
Let $R$ be an arbitrary integral domain with fraction field $K$. 
Then, we have the following notions:
\begin{itemize}
	\item $R$ is \emph{N-1} if the integral 
	closure of $R$ in $K$ is a finite $R$-algebra. 
	\item $R$ is \emph{Japanese} or \emph{N-2} if the integral closure of 	$R$ in every finite field extension of $K$ is a finite $R$-algebra. 
\end{itemize}
A ring $R$, not necessarily a domain, is \emph{Nagata} if $R$ is 
Noetherian and for every prime ideal 
$\fp$ of $R$, the quotient ring $R/\fp$ is Japanese.
\end{citeddef}

\begin{remark}
\par We will use the terminology ``Japanese ring,'' which is due to
Grothendieck and Dieudonn\'e \cite[Ch.\ 0, Def.\ 23.1.1]{EGAIV1}.
By a theorem of Nagata \cite[Thm.\ 7.7.2]{EGAIV2}, a Noetherian ring is Nagata
if and only if it is universally Japanese in the sense of Grothendieck and
Dieudonn\'e
\cite[Ch.\ 0, Def.\ 23.1.1]{EGAIV1}.
Note that (quasi-)excellent rings are Nagata by the Zariski--Nagata
theorem (see \cite[Cor.\ 7.7.3]{EGAIV2}).
\end{remark}

\par We first record some basic properties of 
the notions introduced in Definition \ref{def:Japanese-Nagata} 
for the reader's convenience, 
with brief indications for why the properties are true.
Most, if not all, of the properties appear in \cite[\S31]{Mat80} or \cite[Ch.\
0, \S23]{EGAIV1},
sometimes using slightly different terminology.
Note that we will 
only study the N-1 and Japanese conditions in the setting of Noetherian
rings. Thus, our statements will usually contain Noetherian hypotheses,
even though Noetherianity may not be strictly needed in some places.
\begin{lemma}
\label{lem:field-theory}
Let $R$ be a Noetherian domain with fraction field $K$. We have the 
following:
\begin{enumerate}[label=$(\roman*)$,ref=\roman*]
\item\label{lem:N1-finite-ext}
If every finite extension domain of $R$ is N-1, then $R$ is
Japanese.

\item\label{lem:Japanese-purely-inseparable}
Suppose that
for every finite purely inseparable field extension $K \subseteq L$,
the integral closure of $R$ in $L$ is a module-finite $R$-algebra.
Then $R$ is Japanese.

\item\label{lem:N1-N2-char0}
When $K$ has characteristic zero, then $R$ is N-1 if and only 
if $R$ is Japanese.

\item\label{lem:N1-normal-locus}
If $R$ is N-1, then its normal locus in $\Spec(R)$ is open.
\end{enumerate}
\end{lemma}

\par\noindent See \cite[(31.F), Lem.\ 4]{Mat80} for a partial converse to
$(\ref{lem:N1-normal-locus})$.

\begin{proof}
$(\ref{lem:N1-finite-ext})$ 
Let $K \subseteq L$ be a finite field extension. We have to show that
the integral closure $S$ of $R$ in $L$ is a finite $R$-algebra. 
Choose a basis $\{l_1,l_2,\dots,l_n\}$ of $L/K$ 
such that the $l_i$ are integral over $R$. Let $S'$ be the finite $R$-algebra
$R[l_1,l_2,\dots,l_n]$. Note that the fraction field of $S'$ is $L$. 
Since $S'$ is N-1 by hypothesis, it follows that the integral closure of
$S'$ in $L$ is a finite $S'$, hence also, a finite $R$-algebra.
But the integral closure of $S'$ in $L$ is $S$.

\par $(\ref{lem:Japanese-purely-inseparable})$ 
We have to show that if $K \subseteq M$ is a finite field 
extension, then the integral closure of $R$ in $M$ is module-finite
over $R$. Since a submodule of a finitely generated module over a 
Noetherian ring is finitely generated, we may enlarge 
$M$ to assume $K \hookrightarrow M$ is normal.
The normality of $M$ then implies that there is a 
factorization $K \hookrightarrow L \hookrightarrow M$ such that $L/K$ is
purely inseparable and $M/L$ is separable
\cite[\href{https://stacks.math.columbia.edu/tag/030M}{Tag
030M}]{stacks-project}.
The proof then follows 
by the hypothesis of the Lemma and the fact that integral
closures of Noetherian domains in finite separable extensions of their 
fraction fields are always module-finite (see \cite[(31.B), Prop.]{Mat80}).

\par $(\ref{lem:N1-N2-char0})$ 
It is clear that a Japanese ring is N-1. Conversely, an N-1 ring
whose fraction field has characteristic zero is Japanese by 
$(\ref{lem:Japanese-purely-inseparable})$.

$(\ref{lem:N1-normal-locus})$ 
Let $R^N$ be the normalization of $R$ in $K$. Since
$R$ is N-1, $R \hookrightarrow R^N$ is a finite map, and hence the cokernel $Q$
is also a finitely generated $R$-module.
The support of $Q$ as an $R$-module is therefore closed, and it consists
of those points $\fp \in \Spec(R)$ for which $R_\fp$ is not normal by the fact
that taking integral closures commutes with localization \cite[Ch.\ V, \S1,
n\textsuperscript{o} 5, Prop.\ 16]{BouCA}.
The normal locus in $\Spec(R)$ is the complement of the support of $Q$, and is
therefore open.
\end{proof}

\par A simple sufficient condition for a ring of prime characteristic to be N-1
is the following:

\begin{lemma}[{cf.\ {\cite[Proof of Cor.\ 2.2]{EG}}}]\label{lem:fpuren1}
  Let $(R,\fm)$ be a Noetherian local ring of prime characteristic $p > 0$.
  If $R$ is $F$-injective (for example, if $R$ is $F$-pure), then $R$ is N-1.
\end{lemma}
\begin{proof}
  If $R$ is $F$-injective, then $\widehat{R}$ is $F$-injective as well by \cite[Rem.\ on p.\ 473]{Fed83}.
  Since $F$-injective rings are reduced \cite[Lem.\ 3.11]{QS17}, 
this means that $R$ is analytically
  unramified, and hence $R$ is N-1 by \cite[(31.E)]{Mat80}.
\end{proof}

\subsection{Solidity of absolute integral closures and Japanese rings}
\par Our main result of this subsection relates the solidity of absolute
integral closures of Noetherian domains to the Japanese condition. The notion
of $F$-solidity also allows us to prove a prime characteristic analogue of
Lemma \ref{lem:field-theory}$(\ref{lem:N1-N2-char0})$.

\begin{theorem}
\label{thm:solid-N2}
Let $R$ be a Noetherian domain of arbitrary characteristic and let $R^+$ be the absolute
integral closure of $R$. Then, we have the following:
\begin{enumerate}[label=$(\roman*)$,ref=\roman*]
	\item\label{thm:A+-solid-Japanese} 
	If $R^+$ is a solid $R$-algebra, then $R$ is a Japanese ring.
	
	\item\label{thm:F-solid-N1-N2}
	Suppose $R$ has prime characteristic. If the perfect closure
	$R_\perf$ of $R$ is a solid $R$-algebra, or more generally, if $R$ is $F$-solid,
	then $R$ is N-1 if and only if $R$ is Japanese.
	
	\item\label{thm:A+-solid-excellent} 
  	If $R$ has prime characteristic, is generically $F$-finite, and the condition
  	in $(\ref{thm:A+-solid-Japanese})$ holds, then $R$ is excellent. 
\end{enumerate}
\end{theorem}

\begin{proof}
Throughout this proof, let $K$ denote the fraction field of $R$.

$(\ref{thm:A+-solid-Japanese})$ Let $L$ be a finite extension of $K$.
Let $S$ be the integral closure of $R$ in $L$. We may assume without loss of
generality that $L$ is a subfield of $\Frac(R^+)$ and $S$ is a subring of 
$R^+$ (since $S$ is an integral extension of $R$). Therefore $R$-solidity of $R^+$
implies that $S$ is a solid $R$-algebra (Remark \ref{rem:useful-solid-facts}). 
Since $R$ is Noetherian and the extension
$R \hookrightarrow S$ is generically finite (the fraction field of 
$S$ is $L$), it follows by Proposition \ref{prop:generically-algebra-maps} 
that $S$ is a finite $R$-algebra.

\par$(\ref{thm:F-solid-N1-N2})$ 
The  backward implication is clear because a Japanese ring is always N-1. 
To show the forward implication, by Lemma 
\ref{lem:field-theory}$(\ref{lem:Japanese-purely-inseparable})$
it suffices to show that if $L$ is a finite purely inseparable field extension of $K$,
then the integral closure of $R$ in $L$ is a finite $R$-algebra. If $R_\perf$ is a solid
$R$-algebra, then $F_{R*}R$ is also a solid $R$-algebra because it embeds in $R_\perf$
(Remark \ref{rem:useful-solid-facts}).
Therefore, it suffices to show that every $F$-solid Noetherian N-1 domain is Japanese. 
Let $R^N$
be the integral closure of $R$ in $K$. Note that for every integer $e > 0$, 
$F^e_{R^N*}R^N$ is the integral
closure of $R$ in $F^e_{K*}K$. Since $R$ is N-1, $R \hookrightarrow R^N$ is a finite 
extension of domains. Therefore, $F^e_{R^N*}R^N$ is a solid $R^N$-algebra for every $e > 0$
by Remark \ref{rem:F-solidity-obs}, $(\ref{rem:Fsolid-all-e})$ and
$(\ref{rem:Fsolid-finite-ext})$,
because $R$ is $F$-solid.
Then for every integer $e > 0$, $F^e_{R^N*}R^N$ is a solid $R$-algebra 
by Lemma \ref{lem:solid-composition}
because $R^N$ is a solid $R$-algebra by module-finiteness 
(see Example \ref{ex:finite-solid}).
Returning to our finite inseparable extension $L/K$, 
there exists $e_0 \gg 0$ and a $K$-algebra embedding of $L$ in $F^{e_0}_{K*}K$.
The integral closure $S$ of $R$ in $L$ is then a subring of $F^{e_0}_{R^N*}R^N$. Then 
$S$ is a solid $R$-algebra because $F^{e_0}_{R^N*}R^N$ is a solid $R$-algebra. 
By Proposition \ref{prop:generically-algebra-maps} we again conclude that
 $S$ is module-finite over $R$ because $R \hookrightarrow S$ is generically finite.

\par We will use $(\ref{thm:A+-solid-Japanese})$ to prove $(\ref{thm:A+-solid-excellent})$ 
directly, although $(\ref{thm:A+-solid-Japanese})$  implies
that $R$ is $F$-solid (after embedding $F_{R*}R$ in $R^+$), which then
implies $(\ref{thm:A+-solid-excellent})$ by \cite[Thm.\ 3.2]{DS18}. Note that by 
$(\ref{thm:A+-solid-Japanese})$, $R$ is a Japanese ring. Since $R^p$
is isomorphic to $R$, it follows that $R^p$ is also Japanese. Using the 
assumption that $R$ is a generically $F$-finite, the extension
$K^p \hookrightarrow K$ is finite. 
Thus, the integral closure $S$ of $R^p$ in $K$ is module
finite over $R^p$. But $R$ is contained in $S$, so that $R$ is also module 
finite over $R^p$ by Noetherianity. Thus, $R$ is $F$-finite, and consequently, also excellent by \cite[Thm. 2.5]{Kun76}.
\end{proof}

\begin{remark}
{\*}
\begin{enumerate}
	\item Lemma \ref{lem:field-theory}$(\ref{lem:N1-N2-char0})$ indicates 
	that the difference between N-1 and the 
	Japanese property arises because of inseparability. Thus, Theorem 
	\ref{thm:solid-N2}$(\ref{thm:F-solid-N1-N2})$
	shows that $F$-solidity can be viewed as an antidote for the 
	bad behavior of inseparable extensions in
	prime characteristic.

	\item It is natural to ask to what extent the $R$-solidity of $R^+$
	differs from the $R$-solidity of $R_\perf$ for a ring $R$ of prime 
	characteristic. 
	We will see in Example \ref{ex:F-solid-not-N1} that there exist 
	one-dimensional
	locally excellent (but not excellent) Noetherian domains $R$ that are not 	N-1 (hence
	also not Japanese), but such that $R_\perf$ is a solid $R$-algebra.
	Thus, $R$-solidity of $R_\perf$,
 	or more generally, $F$-solidity of $R$, is weaker than $R$-solidity 
	of $R^+$ by Theorem \ref{thm:solid-N2}$(\ref{thm:A+-solid-Japanese})$.
\end{enumerate}
\end{remark}

\subsection{Solidity of filtered colimits and Japanese rings of characteristic
zero}
Let $R$ be domain. Both $R^+$ and the perfect closure $R_\perf$ 
(if $R$ has positive characteristic) can be expressed as filtered colimits
of module-finite $R$-subalgebras, each of which are $R$-solid by
Example \ref{ex:finite-solid}. Thus, it is natural to wonder how the notion
of solidity behaves under filtered colimits of rings. This topic is pursued
in this subsection. Our analysis (Proposition \ref{prop:colimits}) shows,
for example, that $R_\perf$ is always $R$-solid for a Frobenius split ring.
We also find a natural class of excellent DVRs of
prime characteristic that are Frobenius split, but that are not necessarily 
$F$-finite or complete. A partial converse of Theorem 
\ref{thm:solid-N2}$(\ref{thm:A+-solid-Japanese})$ is obtained for
normal domains that contain the rational numbers. Although this 
latter observation is known to experts, the 
conclusion we draw using it (Corollary \ref{cor:N-1-char0})
gives a new characterization of Japanese Noetherian rings in equal 
characteristic zero.

\begin{proposition}
\label{prop:colimits}
Let $R$ be a ring, not necessarily Noetherian. 
\begin{enumerate}[label=$(\roman*)$,ref=\roman*]
\item\label{prop:colimit-solid}
 Let $\{R_i\}_{i \in I}$
be a collection of $R$-algebras indexed by a filtered poset. 
Suppose there exists an index
$i_0 \in I$ such that for all $i \geq i_0$ there exists an $R$-linear map
$\varphi_i\colon R_i \rightarrow R$. Assume the maps $\varphi_i$ are compatible 
with the transition maps
$\phi_{ij}\colon R_i \rightarrow R_j$ in the obvious sense. If
$\varphi_{i_0}$ is nonzero map, then $\colim_i R_i$ is a solid $R$-algebra.
In particular, if $\varphi_{i_0}$ is a splitting of $R \rightarrow R_{i_0}$,
then $R$ is a direct summand of $\colim_i R_i$.

\item\label{prop:perf-split}
 If $R$ has prime characteristic, $R$ is Frobenius split if and only
if $R \rightarrow R_\perf$ splits. In particular, $R_\perf$ is a solid
$R$-algebra when $R$ is Frobenius split.

\item\label{prop:Dedekind-solid-countgen}
Suppose $R$ is a Japanese Dedekind domain with fraction field 
$K$. Let $L$ be an algebraic extension of $K$ that is countably generated 
over $K$. If $R^L$ is the integral closure of $R$ in $L$, then $R$ is
 a direct summand of $R^L$. In particular, $R^L$ is a solid $R$-algebra.
 
\item\label{prop:Dedekind-Fsolid}
  If $R$ is a Japanese DVR (or more generally, a
  Japanese Dedekind domain) of prime characteristic such that $F_{K*}K$ is countably
generated over the fraction field $K$ of $R$, then $R$ is Frobenius split.\footnote{This result was observed by Karl Schwede and the first author 
during the 2015 Math Research Communities in commutative algebra at
Snowbird, Utah.}

\item\label{prop:A+-solid-char0} 
Suppose $R$ is a domain (not necessarily Noetherian) 
containing the rational numbers $\mathbf{Q}$. Then, $R$ is normal
if and only if the natural map $R \rightarrow R^+$ splits. 
\end{enumerate}
\end{proposition}

\begin{proof}
For $(\ref{prop:colimit-solid})$, the existence of an $R$-linear map 
$\Psi\colon \colim_i R_i \rightarrow R$ 
follows by the universal property of colimits and the compatibility
condition on the $\varphi_i$'s. Since the composition
\[
  R_{i_0} \longrightarrow \colim_{i \in I} R_i \overset{\Psi}{\longrightarrow} R
\]
coincides with $\varphi_{i_0}$, it follows that if the latter is nonzero,
then $\Psi$ must be as well. Consequently, $\colim_i R_i$ is a solid $R$-algebra
if $\varphi_{i_0} \neq 0$. If $\varphi_{i_0}$ is additionally a splitting
of $R \rightarrow R_{i_0}$, then $\varphi_{i_0}$ maps $1$ to $1$. Then, $\Psi$
also maps $1$ to $1$, completing the proof of $(\ref{prop:colimit-solid})$.

\par For $(\ref{prop:perf-split})$, we apply $(\ref{prop:colimit-solid})$ 
to the filtered system of $R$-algebras $\{F^e_{R*}R\}_{e \in \ZZ_{> 0}}$
where the transition maps $F^e_{R*}R \rightarrow F^{e+1}_{R*}R$ are just the $p$-th power
Frobenius map. If $R$ is Frobenius split, let $\varphi_1\colon F_{R*}R \rightarrow R$ be a splitting
of the Frobenius map. Then for each $e \geq 1$, define $\varphi_e\colon F^e_{R*}R \rightarrow R$
to be the composition
\[
\varphi_e \colon F^e_{R*}R \xrightarrow{F^{e-1}_{R*}\varphi_1}F^{e-1}_{R*}R 
\xrightarrow{F^{e-2}_{R*}\varphi_1} F^{e-2}_{R*}R \longrightarrow \dots \longrightarrow
F_{R*}R \overset{\varphi_1}{\longrightarrow} R.
\]
Said more simply, $\varphi_e$ as a map of sets from $R \rightarrow R$ is just the composition
of $\varphi_1$ with itself $e$-times. Then, using the fact that $\varphi_1$ is a left inverse 
of the Frobenius map, it is easy to verify that the collection $\{\varphi_e\}$ is compatible
with the transition maps of the filtered system $\{F^e_{R*}R\}_{e \in \ZZ_{> 0}}$. Now taking colimit 
and using $(\ref{prop:colimit-solid})$, we get that $R \rightarrow R_\perf$ splits. Conversely, if $R \rightarrow R_\perf$
splits, then the factorization $R \xrightarrow{F} F_{R*}R \rightarrow R_\perf$ shows that
$R \rightarrow F_{R*}R$ splits as well, that is, $R$ is Frobenius split.

\par $(\ref{prop:Dedekind-Fsolid})$ follows readily from 
$(\ref{prop:Dedekind-solid-countgen})$ by taking $L$ to be $F_{K*}K$ and using 
the observation that since $R$ is normal, the integral closure of $R$
in $F_{K*}K$ is precisely $F_{R*}R$. Thus, we now prove 
$(\ref{prop:Dedekind-solid-countgen})$. Fix
a countable set of generators $\Set{\ell_n \given n \in \ZZ_{> 0}}$ of $L$ over
$K$. For each $n \in \ZZ_{\geq 0}$, 
let
\[
  K_n \coloneqq K(\ell_1,\ell_2,\dots,\ell_n) = K[\ell_1,\ell_2,\dots,\ell_n].
\]
Moreover, let $R_n$ be the integral closure of $R$ in $K_n$. Note that
$K_0 = K$ and $R_0 = R$ because $R$ is normal.
Furthermore, $R_{n+1}$ is the integral closure of $R_n$ in $K_{n+1}$, hence
the collection of rings $\Set{R_n \given n \in \ZZ_{\geq 0}}$ is filtered with
inclusions as the transition maps. Since $L$ is the 
union of the subfields $K_n$, it follows that $R^L$ is the colimit of the
rings $R_n$, which are themselves Dedekind domains. As $K_n$
is a finite extension of $K$, it follows by the Japanese property 
that $R_n$ is a module-finite $R$-algebra, for all $n \geq 0$. 
Thus, $R_n$ is also module-finite over $R_{n-1}$, for all
$n \geq 1$.
Now using the fact that $R_n$ is a regular ring, 
fix an $R_{n-1}$-linear splitting 
\[
s_n\colon R_n \longrightarrow R_{n-1}
\]
of the inclusion $R_{n-1} \hookrightarrow R_n$ for all $n \geq 1$ via the direct summand theorem.
Then define 
\[
\varphi_n \coloneqq s_1 \circ s_{2} \circ \dots \circ s_n.
\]
Thus, for $n \geq 1$, the map $\varphi_n\colon R_n \rightarrow R$ is an $R$-linear splitting 
of the inclusion $R \hookrightarrow R_n$ chosen in a way so that the splittings
are compatible with the transition maps of the filtered system $\{R_n\}$. Therefore,
the induced map $R^L \rightarrow R$ is a splitting of $R \hookrightarrow R^L$
by $(\ref{prop:colimit-solid})$.

$(\ref{prop:A+-solid-char0})$ 
We first show that if $R$ is normal, then $R \rightarrow R^+$ splits. 
Write $R^+$ as a filtered colimit of its module-finite $R$-subalgebras $R_i$
with the transition maps given by inclusions. 
Choose $R_{i_0} = R$ and
$\varphi_{i_0}\colon R_{i_0} \rightarrow R$ to be the identity map on $R$. Given 
any finite subalgebra $R \rightarrow R_i$ of $R^+$, define 
$\varphi_i\colon R_i \rightarrow R$ 
to be the restriction of the normalized trace map 
\[
\frac{1}{[K_i:K]} \Tr_{K_i/K}\colon K_i \longrightarrow K
\]
to $R_i$, where $K = \Frac(R)$ and $K_i = \Frac(R_i)$. The restricted normalized
trace maps $R_i$ into $R$. Indeed, if $a \in R_i$, then
the minimal polynomial $f_a(x)$ of $a$
over $K$ has coefficients in $R$ by normality of $R$. Since 
$\Tr_{K_i/K}(a)$ is an integer multiple of a coefficient of
$f_a(x)$ 
\cite[\href{https://stacks.math.columbia.edu/tag/0BIH}{Tag 0BIH}]{stacks-project},
and since $R$ contains $\QQ$, 
it follows that for all $a \in R_i$,
\[
\frac{1}{[K_i:K]} \Tr_{K_i/K}(a) \in R.
\]
One can check that $\varphi_i$ is a splitting of $R \rightarrow R_i$. 
We now verify that the maps $\varphi_i$ are compatible with the transition inclusions 
$R_i \hookrightarrow R_j$. So let $a \in R_i$. Then,
\begin{align*}
\varphi_j(a) \coloneqq{}& \frac{1}{[K_j:K]}\Tr_{K_j/K}(a)\\ 
={}& \frac{1}{[K_j:K_i][K_i:K]}\Tr_{K_i/K}\bigl(\Tr_{K_j/K_i}(a)\bigr)\\
={}& \frac{1}{[K_j:K_i][K_i:K]}\Tr_{K_i/K}\bigl([K_j:K_i]\,a\bigr)\\
={}& \frac{1}{[K_i:K]}\Tr_{K_i/K}(a) =: \varphi_i(a).
\end{align*}
Here, the second equality follows by the behavior of trace maps under a tower
of field extensions, and the third equality follows because $a$ is already an element 
of $K_i$. Thus, the $\varphi_i$ give a compatible system of splittings 
$R_i \rightarrow R$, and so, by $(\ref{prop:colimit-solid})$ 
we get a splitting of $R^+ \rightarrow R$.

\par Conversely suppose $R \hookrightarrow R^+$ splits. Let $R^N$ be the normalization
of $R$. Using the factorization $R \hookrightarrow R^N \hookrightarrow R^+$ we see
that $R \hookrightarrow R^N$ also splits. As $R \hookrightarrow R^N$ is 
generically an isomorphism, this implies $R \hookrightarrow R^N$ must be an
isomorphism as well. Indeed, $\coker(R \hookrightarrow R^N)$ is $R$-torsion
free as it can identified as a submodule of the $R$-torsion free module $R^N$,
and so, $\coker(R \hookrightarrow R^N) = 0$ since it is generically trivial.
\end{proof}

As a consequence of Proposition \ref{prop:colimits} and Theorem \ref{thm:solid-N2},
we obtain the following characterization of the N-1 property for Noetherian
domains of equal characteristic zero. 
Note that this corollary also gives a new proof of
the fact that the N-1 condition coincides with the Japanese property
in equal characteristic zero
(cf.\ the proof of Lemma \ref{lem:field-theory}$(\ref{lem:N1-N2-char0})$).

\begin{corollary}
\label{cor:N-1-char0}
Let $R$ be a Noetherian domain containing the rational numbers $\mathbf{Q}$.
Then, the following are equivalent.
\begin{enumerate}[label=$(\roman*)$,ref=\roman*]
\item\label{cor:+-solid} $R^+$ is a solid $R$-algebra.
\item\label{cor:Japanese-char0} $R$ is Japanese.
\item\label{cor:N1-char0} $R$ is N-1.
\end{enumerate}
\end{corollary}

\begin{proof}
We have $(\ref{cor:+-solid}) \Rightarrow (\ref{cor:Japanese-char0})$ 
by Theorem \ref{thm:solid-N2}$(\ref{thm:A+-solid-Japanese})$. Note
this is the only place where we will use that $R$ is Noetherian. The 
implication $(\ref{cor:Japanese-char0}) \Rightarrow (\ref{cor:N1-char0})$ 
follows by the definitions of N-1 and Japanese rings 
(Definition \ref{def:Japanese-Nagata}). 

\par Thus, it remains to show that $(\ref{cor:N1-char0}) \Rightarrow (\ref{cor:+-solid})$. Let $R^N$ be the normalization of $R$.
Then by assumption, $R^N$ is a module-finite $R$-algebra, hence
$R$-solid by Example \ref{ex:finite-solid}. Moreover, it is clear that
$R^+ = (R^N)^+$. That $(R^N)^+$ is a solid $R^N$-algebra
follows by Proposition \ref{prop:colimits}$(\ref{prop:A+-solid-char0})$ 
because $R^N$ is normal. Then,
Lemma \ref{lem:solid-composition} implies that $R^+ = (R^N)^+$ is 
also $R$-solid.
\end{proof}

\subsection{Solidity and excellence in Krull dimension one}
Theorem \ref{thm:solid-N2} and Proposition \ref{prop:colimits}  
raise the question of whether there exist non-excellent
domains $R$ of prime characteristic $p > 0$ such that $R^+$ is $R$-solid,
or more generally, if $R$ is $F$-solid. In this subsection our
main result shows that N-1 
$F$-solid Noetherian domains of prime characteristic 
$p > 0$ and Krull dimension one are always excellent
(Theorem \ref{thm:excellent-dvr}). In particular,
we show as a consequence that all $F$-solid Dedekind domains of prime
characteristic are excellent, although the converse fails
even for excellent Henselian discrete valuation rings by 
Proposition \ref{prop:non-F-solid}.
The novel aspect of Theorem \ref{thm:excellent-dvr} is that 
unlike \cite{DS18}, no generic $F$-finiteness assumptions are made. 
Note that it is well-known that a
Dedekind domain whose field of fractions has characteristic zero is 
excellent (we prove it below for the reader's convenience), 
making the question of excellence for Dedekind domains only interesting
in prime characteristic.

\begin{theorem}
\label{thm:excellent-dvr}
Let $R$ be a Noetherian N-1 domain of
Krull dimension one. 
Suppose $R^+$
is a solid $R$-algebra (in which case one does not need $R$ to be N-1), 
or that $R$ is an $F$-solid domain of prime characteristic
$p > 0$. 
Then, $R$ is excellent. In particular, a Frobenius split (or $F$-solid) 
Dedekind domain is always excellent.
\end{theorem}

\begin{proof}
We directly verify all the axioms for $R$ to be excellent
to illustrate to the reader that these axioms are quite
checkable for low-dimensional rings.

\par We will focus on verifying the axioms when $R$ has prime characteristic $p > 0$, and mention how to adapt the proof when the fraction field $K$ of $R$ has characteristic zero.
We have to 
check the following three axioms that characterize an excellent ring (see Definition
\ref{def:excellent}).

(1) \emph{$R$ is universally catenary:} This is immediate because $R$ is a one-dimensional
domain, hence
Cohen--Macaulay. And it is well-known that Cohen--Macaulay rings are 
universally catenary \cite[Thm.\ 33]{Mat80}.

(2) \emph{If $S$ is a finite type $R$-algebra, the regular locus of $S$ is open:} By
\cite[Thm.\ 73]{Mat80}, it suffices to show that for every $\fp \in \Spec(R)$ and for
every finite field extension $K'$ of the residue
field $\kappa(\fp)$ at $\fp$, there exists a finite $R$-algebra $R'$ with
fraction field $K'$, such that $R'$ contains $R/\fp$ and such that the regular
locus in $\Spec(R')$ contains a nonempty open set.
Since $R$ is one-dimensional,
$\fp$ is either a maximal ideal or the zero ideal. If $\fp$ is maximal, then
one can just take $R'$ to equal $K'$. If $\fp = (0)$, then $\kappa(\fp)$ is the
fraction field $K$ of $R$. By $F$-solidity and the fact that $A$ is
N-1, Theorem \ref{thm:solid-N2}$(\ref{thm:F-solid-N1-N2})$ 
implies that $R$ is a Japanese ring. Then
one can take $R'$ to be the integral closure of $R$ in $K'$. Note $R'$ is a
one-dimensional normal domain (hence regular) that is module-finite over $R$
with fraction field $K'$. 

The argument for property (2) only uses that $R$
has Krull dimension one and is Japanese. Note that if $K$
has characteristic zero, then $R$ is automatically Japanese when $R^+$
is a solid $R$-algebra by Theorem \ref{thm:solid-N2}$(\ref{thm:A+-solid-Japanese})$.

(3) \emph{The formal fibers of the local rings of $R$ are geometrically regular:} 
Let $\fp \in \Spec(R)$. We have to check that the map
\[
R_\fp \longrightarrow \widehat{R_\fp}
\]
from $R_\fp$ to its $\fp R_\fp$-adic completion has geometrically regular fibers.
If $\fp = (0)$, the completion map is an isomorphism and there is
nothing to verify. Suppose $\fp$ is a nonzero maximal ideal of $R$. Then
$R_\fp$ is one-dimensional local domain. Thus, $R_\fp$ has two formal fibers, 
the closed fiber and the
generic fiber. The closed formal fiber is just an isomorphism of the residue fields
of $R_\fp$ and $\widehat{R_\fp}$, hence it is always geometrically regular. Let $K$
be the fraction field of $R_\fp$ (which is also the fraction field of $R$), 
and consider the generic formal fiber
\[
(\widehat{R_\fp})_K \coloneqq K \otimes_{R_\fp} \widehat{R_\fp}.
\]
Observe that $(\widehat{R_\fp})_K$ is a zero-dimensional Noetherian ring. 
Thus, to show that $K \hookrightarrow (\widehat{R_\fp})_K$ is geometrically
regular, it suffices to show that $(\widehat{R}_\fp)_K$ 
is a geometrically reduced $K$-algebra because the notions ``geometrically
regular'' and ``geometrically reduced'' agree for zero{-\penalty0\hskip0pt\relax}dimensional Noetherian 
algebras over a field. Indeed, a geometrically regular Noetherian algebra
over a field is always geometrically reduced
 because regular rings are reduced.
Conversely, if $R$ is a zero-dimensional Noetherian geometrically
reduced $k$-algebra, then for every finite field extension $k'$ of $k$,
$k' \otimes_k R$ is also zero-dimensional, Noetherian and reduced.
However, we know that a reduced Noetherian ring of dimension zero
always decomposes as a finite direct product of fields by the Chinese 
Remainder Theorem. Thus,
$k' \otimes_k R$ is regular, which implies that $R$ is geometrically
regular over $k$.

\par To show that the formal fibers of $R_\fp$ are geometrically reduced, 
we will again use that $R$ is Japanese
by $F$-solidity and the N-1 property (Theorem
\ref{thm:solid-N2}$(\ref{thm:F-solid-N1-N2})$).
Since taking integral closures 
commutes with localization,
$R_\fp$ is also a Japanese ring. Now because $R_\fp$ 
has only two prime ideals (it is a one-dimensional local domain),
this implies that $R_\fp$ is in fact a Nagata ring by Definition 
\ref{def:Japanese-Nagata}. 
But a Noetherian local ring is a Nagata ring if and only if
its formal fibers are geometrically reduced by a result of 
Zariski and Nagata \cite[Thm.\ 7.6.4]{EGAIV2}. This proves that all formal 
fibers of all local rings of $R$ are geometrically regular, completing
the proof that an $F$-solid N-1 Noetherian domain of positive
characteristic is excellent.

\par Again we only used that $R$ is Japanese in the verification
of property (3). 
 This shows that one-dimensional Japanese domains whose fraction fields are of
characteristic zero are always excellent. In particular, Dedekind domains
whose fraction fields are of characteristic zero are excellent because 
such rings are N-1, and hence Japanese by Lemma \ref{lem:field-theory}$(\ref{lem:N1-N2-char0})$.
\end{proof}

\par Said differently, Theorem \ref{thm:excellent-dvr} shows that non-excellent
Dedekind domains of prime characteristic $p > 0$ have no nontrivial $p^{-e}$-linear maps.
In particular,
Nagata's non-excellent discrete valuation ring 
$k \otimes_{k^p} k^p\llbracket t\rrbracket$ where $[k:k^p] = \infty$ (see
\cite[App.\ A1, Ex.\ 3]{Nag75}), has no nonzero 
$p^{-e}$-linear maps.

\begin{remark}
  \leavevmode
\begin{enumerate}
  \item Combined with Lemma \ref{lem:fpuren1}, Theorem \ref{thm:excellent-dvr}
    shows that if a Noetherian domain of Krull dimension one is $F$-pure (or
    more generally, $F$-injective) and $F$-solid, then it is locally excellent
without any N-1 hypotheses. 
This is because both $F$-solidity
and $F$-injectivity are preserved under localization (see 
\cite[Prop.\ 3.3]{DMfinj} for the latter fact). Thus, such a ring is locally N-1 and locally F-solid, and hence, locally excellent. Note that F-solidity is crucial and just F-injectivity 
(or even F-purity) alone does not guarantee that a Noetherian domain of Krull dimension $1$ is locally excellent. For example, there exist non-excellent 
DVRs in the function field of $\textbf{P}^2_{\overline{\FF_p}}$ by 
\cite[\S4.1]{DS18}.

\item The N-1 hypothesis in Theorem \ref{thm:excellent-dvr} is necessary 
even if the one-dimensional domain is $F$-solid.
In the next subsection we will construct a one-dimensional locally
excellent Noetherian domain $R$ for which $R_\perf$ is $R$-solid
but $R$ is not N-1, hence also not excellent (Example \ref{ex:F-solid-not-N1}).

\item Theorem \ref{thm:excellent-dvr} has the interesting consequence that if 
$R$ is a normal Noetherian domain of prime characteristic such that $R^+$ 
is $R$-solid, or
more generally, if $R$ is $F$-solid, then for every
height one prime $\fp$ of $R$, the localization $R_\fp$ is excellent.
Indeed, since the formation of absolute integral closures
commutes with localization, and since all rings involved are domains,
one can check that if $S \subset R$ is a multiplicative set, then
$(S^{-1}R)^+$ is also a solid $R$-algebra (resp.\ $S^{-1}R$ is $F$-solid)
by just localizing a nonzero $R$-linear map from $R^+$ (resp.\ $F_{R*}R$) 
to $R$. Thus, Theorem \ref{thm:excellent-dvr} applied to the localization $R_\fp$
implies $R_\fp$ is excellent. In other words, an $F$-solid normal (or, more
generally, N-1) Noetherian domain $R$ of prime characteristic is ``excellent 
in codimension 1.''
\end{enumerate}
\end{remark}

Theorem \ref{thm:excellent-dvr}, the results of this paper,
and those of \cite{DM} (see also \cite{DS18})
allow us to obtain a
classification of $F$-solid DVRs, which we now
collect and summarize for the reader's convenience.

\begin{corollary}
\label{cor:Fsolid-DVRs}
Let $(R,\fm)$ be a DVR of prime characteristic
$p > 0$. Consider the following statements:
\begin{enumerate}[label=$(\roman*)$,ref=\roman*]
	\item\label{cor:dvr-maps} $R$ is $F$-solid.
	\item\label{cor:dvr-F-split} $R$ is Frobenius split.
	\item\label{cor:dvr-split-F} $R$ is split $F$-regular.
	\item\label{cor:excellent-dvr} $R$ is excellent.
\end{enumerate}
Then, $(\ref{cor:dvr-maps})$, $(\ref{cor:dvr-F-split})$, and 
$(\ref{cor:dvr-split-F})$ are equivalent, and always imply
$(\ref{cor:excellent-dvr})$. All four assertions are equivalent
if $R$ is essentially of finite type over a complete local ring,
or the fraction field $K$ of $R$ is such that $K^{1/p}$ is
countably generated over $K$. However, there exist excellent
Henselian DVRs that are not $F$-solid.
\end{corollary}

\begin{proof}
The equivalence of $(\ref{cor:dvr-maps})$, $(\ref{cor:dvr-F-split})$, and 
$(\ref{cor:dvr-split-F})$ follow from Proposition 
\ref{thm:split-F-regular-DVR}, while each of these equivalent statements
implies $(\ref{cor:excellent-dvr})$ by Theorem \ref{thm:excellent-dvr}.

\par That all four statements are equivalent when $R$ is essentially of
finite type over a complete local ring, or if the fraction field $K$ of
$R$ has the property that $K^{1/p}$ is countably generated over $K$
follows from Theorem \ref{thm:splittingwithgamma} and Proposition
\ref{prop:colimits}$(\ref{prop:Dedekind-Fsolid})$. Finally, the
existence of non-$F$-solid excellent Henselian discrete valuation rings
follows from Proposition \ref{prop:non-F-solid}, which in turn depends on
the results of \cite{DM}.
\end{proof}

\subsection{A meta construction of Hochster and solidity of perfect closures}
In the previous Subsection we proved that
Noetherian N-1 domains of Krull dimension one and prime characteristic
$p > 0$ that are $F$-solid are also excellent (Theorem \ref{thm:excellent-dvr}). 
Our goal in 
this subsection will be to use the following meta construction due to Hochster to
give examples of non-excellent one-dimensional domains that are $F$-solid, 
and even Frobenius split. In other words, we will show that the N-1 assumption 
cannot be weakened in Theorem \ref{thm:excellent-dvr}.  
Moreover, we will also obtain examples of Noetherian
domains of prime characteristic whose perfect closures are solid, but whose
absolute integral closures are not (Example \ref{ex:F-solid-not-N1}).
In a sense Hochster's construction, and consequently our examples, are 
as nice as possible 
because the rings will be locally excellent.

\begin{citedthm}[{\cite[Props.\ 1 and 2]{Hoc73}}]
\label{thm:meta-Mel}
Let $\mathcal{P}$ be a property of Noetherian local rings. Let $k$ 
be a field, and let $(R, \fm)$ be a local 
ring essentially of finite type over $k$ such that
\begin{enumerate}[label=$(\roman*)$,ref=\roman*]
\item\label{thm:g-integ} $R$ is geometrically integral;
\item\label{thm:trivial-residue} $R/\fm = k$; and
\item\label{thm:not-P} for every field extension $L \supseteq k$, the ring 
$(L \otimes_k R)_{\fm(L \otimes_k R)}$ fails to satisfy $\mathcal{P}$.
\end{enumerate}
Moreover, suppose every field extension $L \supseteq k$ satisfies
$\mathcal{P}$. For all $n \in \ZZ_{> 0}$, let $R_n$ be a copy of 
$R$ with maximal ideal $\fm_n = \fm$. Let  
$R' \coloneqq \bigotimes_{n \in \ZZ_{> 0}} R_n$, where the infinite 
tensor product is taken over $k$. Then, each $\fm_nR'$ is a prime 
ideal of $R'$. Moreover, if 
$S = R' \smallsetminus (\bigcup_n \fm_nR')$, then the ring
\[
T \coloneqq S^{-1}R'
\]
is a Noetherian domain whose locus of primes that 
satisfy $\mathcal{P}$ is not open in $\Spec(T)$. Furthermore, 
the map $n \mapsto \fm_nT$ induces a one-to-one correspondence
between $\ZZ_{> 0}$ and the maximal ideals of $T$, and
\[
T_{\fm_nT} \simeq (L_n \otimes_k R_n)_{\fm_n(L_n \otimes_k R_n)},
\]
where $L_n$ is the fraction field of the domain $\bigotimes_{m \neq n} R_m$. 
Thus, 
each local ring of $T$ is essentially of finite type 
over 
some appropriate field extension of $k$. 
In particular, all local rings 
of $T$ are excellent, that is, $T$ is locally excellent.
\end{citedthm}

\begin{remark}
\label{rem:Hoc-examp-analysis}
By definition, the infinite tensor product $R' = \bigotimes_{n \in \ZZ_{>0}} R_n$
is the filtered colimit 
\[
\colim_{j \in \ZZ_{> 0}} \biggl( \bigotimes_{n=1}^j R_1 \otimes_k R_2 \otimes_k
\dots \otimes_k R_j \biggr)
\]
of finite tensor products. Also, implicit in Theorem \ref{thm:meta-Mel} 
is the assertion that for each 
$n \in \ZZ_{> 0}$, the ideal $\fm_n(L_n \otimes_k R_n)$ is a prime ideal of
$L_n \otimes_k R_n$ (Hochster calls this property \emph{absolutely prime} in his paper). 
This follows from the fact that since $R_n/\fm_n$
is isomorphic to the base field $k$, 
for ever field extension $L$ of $k$, we have
\[L \otimes_k R_n/\fm_n \simeq L.\]
Thus, $\fm_n(L_n \otimes_k R_n)$ is in fact a maximal ideal of $L_n \otimes_k R_n$ 
for all $n \in \ZZ_{> 0}$.
Moreover, as a consequence of Hochster's construction, it follows
that the dimension of each local ring
$T_{\fm_nT}$ equals the dimension of $R$. Indeed, the extension
$R_n \hookrightarrow L_n \otimes_k R_n$ is faithfully flat, and induces
a faithfully flat local extension 
$R_n = (R_n)_{\fm_n} \hookrightarrow (L_n \otimes_k R_n)_{\fm_n(L_n \otimes_k R_n)}$.
But the closed fiber of this local extension is a field (the maximal ideal of $R_n$ 
expands to the maximal ideal of $(L_n \otimes_k R_n)_{\fm_n(L_n \otimes_k R_n)}$), 
hence is zero-dimensional.
Therefore,
\[
\dim(R) = \dim(R_n) = \dim\bigl((L_n \otimes_k R_n)_{\fm_n(L_n \otimes_k
R_n)}\bigr) = \dim(T_{m_nT}),
\]
where the first equality follows because $R_n$ is just a copy of $R$, the second equality
follows by faithful flatness of the local extension and
\cite[Thm.\ 15.1$(ii)$]{Mat89}, and
the third equality follows by the isomorphism in Hochster's result. In other words, 
each maximal ideal of $T$ has the same height, that is, $T$ is equicodimensional
and $\dim(T) = \dim(R)$. In particular, $T$ has finite Krull dimension.
\end{remark}

\par We can now construct many non-excellent Frobenius split (hence $F$-solid) rings.

\begin{corollary}
\label{cor:non-exc-Fsplit}
Let $(R,\fm)$ be a Noetherian local $\mathbf{F}_p$-algebra 
such that
\begin{enumerate}[label=$(\arabic*)$,ref=\arabic*]
	\item $R$ is essentially of finite type over $\mathbf{F}_p$,\label{item:non-exc-Fsplit-eft}
	\item $R$ is geometrically integral over $\mathbf{F}_p$,
	\item $R/\fm = \mathbf{F}_p$,
	\item $R$ is Frobenius split, and\label{item:non-exc-Fsplit-Fsplit}
	\item $R$ is not regular.\label{item:non-exc-Fsplit-notreg}
\end{enumerate}
Then, the ring $T$ from Theorem \ref{thm:meta-Mel} constructed
using any $R$ satisfying properties $(\ref{item:non-exc-Fsplit-eft})$--$(\ref{item:non-exc-Fsplit-notreg})$ is a locally excellent Frobenius 
split ring whose regular locus
is not open. In particular, $T_\perf$ is a solid $T$-algebra and $T$ is not
excellent.
\end{corollary}

\begin{proof}
First, we give an explicit example of a ring $R$ satisfying properties $(\ref{item:non-exc-Fsplit-eft})$--$(\ref{item:non-exc-Fsplit-notreg})$.
Let $k$ be a field, and consider the homogeneous coordinate ring
\[
  \frac{k[x,y,z]}{(y^2 - xz)}
\]
of a conic in $\mathbf{P}^2_{k}$.
Note that this ring is an integral domain for any $k$ because it can
be identified with the Veronese subring $k[x^2,xy,y^2]$ of $k[x,y]$. 
Moreover, for any field $k$ of prime characteristic $p > 0$, this ring
is split $F$-regular because $k[x^2,xy,y^2]$
is a direct summand of the split $F$-regular ring $k[x,y]$. 
In particular, $k[x,y,z]/(y^2-xz)$ is Frobenius 
split for any $k$, and consequently, so are all its local rings.
Let 
\[
  R = \biggl(\frac{\mathbf{F}_p[x,y,z]}{(y^2-xz)}\biggr)_{(x,y,z)}.
\] 
Then by the above discussion,
it follows that $R$ satisfies properties $(\ref{item:non-exc-Fsplit-eft})$--$(\ref{item:non-exc-Fsplit-Fsplit})$. Furthermore, $R$ is not
regular since it is the local ring at the cone point. Therefore $R$ satisfies 
all five properties in the statement of the corollary. From here on, we do
not need the specifics of what $R$ actually is, but only that it has the 
aforementioned five properties.

Let $\mathcal{P}$ be the property of a Noetherian local ring being regular, and
fix any $\FF_p$-algebra $R$ that satisfies $(\ref{item:non-exc-Fsplit-eft})$--$(\ref{item:non-exc-Fsplit-notreg})$.
Since $R$ is not regular and since regularity descends under faithfully flat
maps \cite[Thm.\ 23.7$(i)$]{Mat89},
it follows that $R$
also satisfies properties $(\ref{thm:g-integ})$-$(\ref{thm:not-P})$ 
of Theorem \ref{thm:meta-Mel}. Construct
$T$ as in Theorem \ref{thm:meta-Mel} using a countable number of copies of $R$. 
By construction, $T$ is locally excellent and the regular locus of $T$ is not
open (we know this locus is nonempty because $T$ is generically regular). 
We will now
show that $T$ is Frobenius split. For this it suffices to show that the ring
\[
  R' \coloneqq \bigotimes_{n \in \mathbf{Z}_{>0}} R_n \coloneqq \colim_{j \in
  \mathbf{Z}_{>0}}\bigl( 
  R_1 \otimes_{\mathbf{F}_p} R_2 \otimes_{\mathbf{F}_p} \dots
  \otimes_{\mathbf{F_p}} R_j\bigr).
\]
is Frobenius split, because $T$ is a localization of $R'$ and Frobenius
splittings are preserved under localization.
Note each $R_n$, being a copy of $R$, is Frobenius split by hypothesis. 
Fix a Frobenius splitting $\varphi\colon F_{R*}R \rightarrow R$ of $R$,
and denote by $\varphi_n\colon F_{R_n*}R_n \rightarrow R_n$ the corresponding splitting of 
the $n$-th copy of $R$. Then, for each $j$, we get an 
$(R_1 \otimes_{\mathbf{F}_p} \dots \otimes_{\mathbf{F}_p} R_j)$-linear
map
\[
\varphi_1 \otimes \dots \otimes \varphi_j\colon
F_{R_1*}R_1 \otimes_{\mathbf{F}_p} \dots \otimes_{\mathbf{F}_p} F_{R_j*}R_j
\longrightarrow
R_1 \otimes_{\mathbf{F}_p} \dots \otimes_{\mathbf{F}_p} R_j
\] 
that maps $1$ to $1$. But note that since the Frobenius map of $\mathbf{F}_p$
is just the identity,
\[
F_{R_1*}R_1 \otimes_{\mathbf{F}_p} \dots \otimes_{\mathbf{F}_p} F_{R_j*}R_j =
F_*(R_1 \otimes_{\mathbf{F}_p} \dots \otimes_{\mathbf{F}_p} R_j).
\]
Thus, for each $j$, $\varphi_1 \otimes \dots \otimes \varphi_j$
is a Frobenius splitting of $R_1 \otimes_{\mathbf{F}_p} \dots \otimes_{\mathbf{F}_p} R_j$.
Moreover, these Frobenius splittings are compatible with the 
transition maps 
 of the filtered system 
$\{R_1 \otimes_{\mathbf{F}_p} \dots \otimes_{\mathbf{F}_p} R_j\}_j$
in the sense that for every pair of indices $\ell > j$, we have a commutative diagram
\[
\begin{tikzcd}
   F_*(R_1 \otimes_{\mathbf{F}_p} \dots \otimes_{\mathbf{F}_p} R_j) 
   \arrow[r] \arrow[d, "\varphi_1 \otimes \dots \otimes \varphi_j"']
    &  F_*(R_1 \otimes_{\mathbf{F}_p} \dots \otimes_{\mathbf{F}_p} R_\ell)  
    \arrow[d, "\varphi_1 \otimes \dots \otimes \varphi_\ell"] \\
   R_1 \otimes_{\mathbf{F}_p} \dots \otimes_{\mathbf{F}_p} R_j \arrow[r]
&R_1 \otimes_{\mathbf{F}_p} \dots \otimes_{\mathbf{F}_p} R_\ell. 
\end{tikzcd}
\]
where the horizontal maps come from the transition map
$R_1 \otimes_{\mathbf{F}_p} \dots \otimes_{\mathbf{F}_p} R_j 
\rightarrow R_1 \otimes_{\mathbf{F}_p} \dots \otimes_{\mathbf{F}_p} R_\ell$
that sends $x_1 \otimes \dots \otimes x_j$ to 
$x_1 \otimes \dots \otimes x_j \otimes 1 \otimes \dots \otimes 1$. 
Taking colimits now gives a Frobenius splitting of the ring $R'$,
which is what we wanted. Note that if $T$ is Frobenius split, 
then $T_\perf$ is a solid
$T$-algebra by Proposition \ref{prop:colimits}$(\ref{prop:perf-split})$. Hence
the last assertion of the Corollary follows.
\end{proof}

Using Corollary \ref{cor:non-exc-Fsplit} we will now show that there exists
a locally excellent Noetherian domain $R$ of Krull dimension one which is 
Frobenius split (in which case $R_\perf$ is a solid $R$-algebra), 
but for which $R^+$ is not a solid $R$-algebra. In particular, this 
example will also show that $F$-solidity alone cannot imply the Japanese 
property without the N-1 hypothesis in 
Theorem \ref{thm:solid-N2}$(\ref{thm:F-solid-N1-N2})$, and that, similarly,
$F$-solidity alone cannot imply excellence in Krull dimension one 
(see Theorem \ref {thm:excellent-dvr}).

\begin{example}
\label{ex:F-solid-not-N1}
Let $R$ be the local ring at the origin of the node over $\mathbf{F}_p$, that is,
\[
  R = \biggl(\frac{\mathbf{F}_p[x,y]}{(y^2 - x^3 - x)}\biggr)_{(x,y)}.
\]
Since the coordinate ring of the node is a domain over any field, it follows
that $R$ is geometrically integral over $\mathbf{F}_p$. We know that $R$ is not
regular (hence not normal since it is one-dimensional). However, an application of
Fedder's criterion \cite[Thm.\ 1.12]{Fed83} shows that $R$ is Frobenius split. By Corollary \ref{cor:non-exc-Fsplit},
the ring $T$ constructed using $R$ is locally excellent, Frobenius split, but not
excellent. In particular, $T_\perf$ is a solid $T$-algebra. That $T$ has
Krull dimension one follows from Remark \ref{rem:Hoc-examp-analysis}. We claim 
that $T^+$ is not a solid $T$-algebra. Indeed, if it is, then
Theorem \ref{thm:solid-N2}$(\ref{thm:A+-solid-Japanese})$ will imply that $T$ is Japanese,
and hence N-1. However, this is impossible by Theorem \ref {thm:excellent-dvr} because an $F$-solid N-1 Noetherian domain of Krull dimension one
is always excellent, whereas $T$ is not.
\end{example}

\begin{remark}
Hochster's construction provides a wealth of examples of locally excellent 
domains that are not excellent (Theorem \ref{thm:meta-Mel}). Said differently, 
excellence in not a local property. Example \ref{ex:F-solid-not-N1}
shows that the properties of being N-1 and Japanese are also not local.
At the same time, the N-1 property is \emph{often} local by \cite[(31.G), Lem.\ 4]{Mat80}.
\end{remark}

Example \ref{ex:F-solid-not-N1} illustrates that $R_\perf$ can be a solid 
$R$-algebra even when $R^+$ is not. 
We now show solidity of perfect closures
for most rings that arise in algebro-geometric applications. To establish this, 
we exploit the following observation of Hochster and Huneke that
is related to the notion of test elements in tight closure theory.

\begin{citedlem}[{\cite[Lem.\ 6.5]{HH90}}]
\label{lem:perf-containment}
Let $\varphi\colon A \rightarrow R$ be a module-finite and generically smooth (equivalently,
generically geometrically reduced) map, where $A$ is a normal domain and
$R$ is torsion-free as an $A$-module. 
Let $r_1,r_2,\dots,r_d \in R$ be a vector space basis for $K' = K \otimes_A R$
over the fraction field $K$ of $A$, and let 
\[c \coloneqq \det\bigl(\Tr_{K'/K}(r_ir_j)\bigr).\]
Then, $c$ is a nonzero element of $A$ such that $cR_\perf \subseteq A_\perf[R]$. 
\end{citedlem}

In the above lemma, $A_\perf[R]$ is the image of the unique map 
$A_\perf \otimes_A R \rightarrow R_\perf$ obtained by the universal property 
of the tensor product that makes the diagram
\[
\begin{tikzcd}
  A \arrow[r] \arrow[d, "\varphi"']
    & A_\perf \arrow[d, "\varphi_\perf"] \\
  R \arrow[r]
&R_\perf
\end{tikzcd}
\]
commute. Moreover, since $\varphi$ is a finite map, having a smooth generic
fiber is the same as having a geometrically reduced generic fiber, which is
in turn the same as having an \'etale generic fiber. 

Using Lemma \ref{lem:perf-containment}, we can show solidity of the 
perfect closure in many geometric situations.

\begin{proposition}
\label{prop:perf-solid}
Let $\varphi\colon A \rightarrow R$ be a module-finite extension of domains of prime
characteristic $p > 0$ such that $A$ is normal and Noetherian. 
If $A_\perf$ is a solid $A$-algebra (for example, if $A$ is Frobenius split),
then $R_\perf$ is a solid $R$-algebra.
\end{proposition}

\begin{proof}
Since $A \rightarrow R$ is a module-finite extension of Noetherian domains,
to show that $R_\perf$ is a solid $R$-algebra, it suffices to show that
$R_\perf$ is a solid $A$-algebra by \cite[Cor.\ 2.3]{Hoc94}. Consider
the chain of maps
\[A \longrightarrow A_\perf \xrightarrow{\varphi_\perf} R_\perf.\]
Since $A \rightarrow A_\perf$ induces an algebraic extension of fraction
fields, by Lemma \ref{lem:solid-composition} it further suffices to
show that $R_\perf$ is a solid $A_\perf$-algebra. Note that 
$A_\perf \rightarrow R_\perf$ need not be a finite map even though $A \rightarrow R$
is. Hence more work has to be done to establish $A_\perf$-solidity of $R_\perf$.
For every $e > 0$, let $A^{1/p^e}[R]$ denote the image of the relative Frobenius map
\[F^e_{R/A}\colon A^{1/p^e} \otimes_A R \longrightarrow R^{1/p^e}.\]
We claim that the kernel of $F^e_{R/A}$ is precisely the nilradical of 
$A^{1/p^e} \otimes_A R$.
Indeed, $F^e_{R/A}$ induces a homeomorphism on spectra, which means that
$\ker(F^e_{R/A})$ is contained in the nilradical of $A^{1/p^e} \otimes_A R$.
On the other hand,
the nilradical of $A^{1/p^e} \otimes_A R$ must be contained in $\ker(F^e_{R/A})$
because $R$ is reduced.
The upshot of this argument is that for each $e > 0$, 
\[(A^{1/p^e} \otimes_A R)_{\red} \simeq A^{1/p^e}[R].\]
Note also that the module-finiteness of $\varphi\colon A \rightarrow R$ implies that
for each $e > 0$, the map $A^{1/p^e} \rightarrow A^{1/p^e}[R]$ is a module-finite extension.
By \cite[Prop.\ 2.4.2.1]{DS19}, we may choose $e \gg 0$ such that the induced map
\[A^{1/p^e} \longrightarrow (A^{1/p^e} \otimes_A R)_{\red} \simeq A^{1/p^e}[R]\]
has geometrically reduced generic fiber. Module-finiteness implies
that having a geometrically reduced generic fiber is the same as having a smooth generic
fiber (see the proof of (3) in Theorem \ref{thm:excellent-dvr}). Moreover, $A^{1/p^e}[R]$ is
a torsion-free module over the normal domain $A^{1/p^e}$. Hence, by
Lemma \ref{lem:perf-containment} applied to the finite extension 
$A^{1/p^e} \hookrightarrow A^{1/p^e}[R]$, 
there exists a nonzero $c \in A^{1/p^e}$ such that
\[c\bigl(A^{1/p^e}[R]_\perf\bigr) \subseteq A^{1/p^e}_\perf\bigl[A^{1/p^e}[R]\bigr] =
A^{1/p^e}_\perf[R].\]
But $A^{1/p^e}[R]_\perf = R_\perf$ and $A^{1/p^e}_\perf = A_\perf$, which means that
left multiplication by $c$ induces a nonzero map $R_\perf \rightarrow R_\perf$
whose image lands inside $A_\perf[R]$. Restricting the codomain of this map
shows that $R_\perf$ is a solid $A_\perf[R]$-algebra. On the other hand,
$A_\perf \hookrightarrow A_\perf[R]$
is a module-finite extension of domains
 because it is the composition of the module-finite
maps 
\[A_\perf \xrightarrow{\id \otimes \varphi} A_\perf \otimes_{A} R 
\longtwoheadrightarrow A_\perf[R].\]
Then, $A_\perf[R]$ is $A_\perf$-solid
by Example \ref{ex:finite-solid}. Thus, $R_\perf$ is a solid $A_\perf$-algebra
upon applying Lemma \ref{lem:solid-composition} to the composition
$A_\perf \hookrightarrow A_\perf[R] \hookrightarrow R_\perf$.
\end{proof}

Upon taking a Noether normalization one can now conclude that the perfect closure
of any algebra that is of finite type over a prime characteristic field is 
solid over the algebra. In fact we obtain the following more general consequence.

\begin{corollary}
\label{cor:module-finite-regular}
Let $A \rightarrow R$ be a module-finite extension of domains such that
$A$ is regular and essentially of finite type over a complete local
ring of prime characteristic $p > 0$. Then, $R_\perf$ is a solid $R$-algebra.
\end{corollary}

\begin{proof}
Since regular rings are $F$-pure, $A$ is Frobenius split by Theorem 
\ref{thm:splittingwithgamma}, and hence $A_\perf$ is $A$-solid by
Proposition \ref{prop:colimits}$(\ref{prop:perf-split})$. 
Since $A$ is also normal, we are done by Proposition \ref{prop:perf-solid}.
\end{proof}

Finally, we observe that the analogue of Proposition \ref{prop:perf-solid} for
the solidity of the absolute integral closure has a significantly simpler proof.

\begin{proposition}
\label{prop:plus-solid}
Let $A \rightarrow R$ be a module-finite extension of Noetherian domains. 
If $A^+$ is a solid $A$-algebra, then $R^+$ is a solid $R$-algebra.
\end{proposition}

\begin{proof}
The module-finite extension $A \rightarrow R$ induces an 
an isomorphism $A^+ \overset{\sim}{\to} R^+$. Thus, $R^+$
is a solid $A$-algebra, and so, $R^+$ is a solid $R$-algebra
by \cite[Cor.\ 2.3]{Hoc94}.
\end{proof}

\begin{remark}
If $R$ is a domain that is also a finite type algebra over a field $k$, 
then Proposition \ref{prop:plus-solid} shows that $R$-solidity of 
$R^+$ reduces to the question of the solidity of the absolute
integral closure of a polynomial ring over $k$ upon taking a suitable 
Noether normalization. However, somewhat surprisingly, 
the question of whether $k[x_1,x_2,\dots,x_n]^+$
is a solid $k[x_1,x_2,\dots,x_n]$-algebra seems to be completely open unless
unless $k$ has characteristic zero, in which case  
solidity follows by Proposition \ref{prop:colimits}$(\ref{prop:A+-solid-char0})$. 
See Question \ref{q:absolute-solid} and Proposition \ref{prop:evaluation-1}
for some further details regarding this.
In fact, the authors do not know of a single non-excellent example
of a Noetherian domain $A$ of positive characteristic for which $A^+$
is a solid $A$-algebra (Question \ref{q:A+-solid-nonexc}).
\end{remark}

\addtocontents{toc}{\protect\smallskip}
\section{Some open questions}\label{sect:openquestions}
\counterwithout{subsubsection}{subsection}
\counterwithin{subsubsection}{section}
For the reader's convenience, we collect a few questions 
that to the best of our knowledge are open and 
related to the results of this paper and \cite{DM}. We will
summarize what 
is known about these questions.\medskip

\par We begin with some questions 
on the variants of the usual notion of strong $F$-regularity.
While split $F$-regularity and $F$-pure regularity do not coincide even 
for excellent Henselian 
regular local rings (Theorem \ref{thm:F-pure-not-split-F}),
we showed that every regular ring essentially of finite type over an
excellent local ring is $F$-pure regular (Theorem \ref{eft-G-ring}). 
This raises the following question.

\begin{question}
\label{q:regular-F-pure}
Is every excellent regular ring of prime characteristic $F$-pure regular?
\end{question}

\noindent We expect (though cannot show) that there are regular rings that 
do not satisfy the ideal-theoretic characterization of $F$-pure regularity 
from Proposition \ref{prop:ideal-theoretic}.
However, we do not know if any such examples can be excellent.
Theorem \ref{eft-G-ring} and Remark \ref{rem:whyexcellentinq1} show that the
characterization of $F$-pure regularity in Proposition
\ref{prop:ideal-theoretic} is satisfied for regular or even strongly $F$-regular
rings of prime characteristic that are essentially of finite type over an
excellent local ring.
This suggests that while Question \ref{q:regular-F-pure} may be true for
excellent regular rings, there may be counterexamples for non-excellent regular
rings.
On the other hand, any example of a regular ring
that is not $F$-pure regular, excellent or not, would be interesting in 
its own right because it will show that Hochster's tight closure notion of
strong $F$-regularity (Definition \ref{def:freg}$(\ref{def:fsingspureregapp})$)
does not coincide with $F$-pure regularity for regular rings. In particular,
this would indicate that Definition \ref{def:freg}$(\ref{def:fsingspureregapp})$
is the most appropriate variant of the usual notion of strong $F$-regularity
in a non-$F$-finite setting.

\par In \cite{DM}, we obtained examples of excellent $F$-pure rings that are not
Frobenius split. However, the rings that we construct satisfy the stronger property 
that they do not admit any nonzero $p^{-1}$-linear maps. Thus the following 
question arises.

\begin{question}
\label{q:Fpure-Fsplit}
Suppose $R$ is a Noetherian $F$-pure ring of prime characteristic $p > 0$ that admits a nonzero
$p^{-1}$-linear map. Is $R$ Frobenius split?
\end{question}

\noindent We showed in Proposition \ref{thm:split-F-regular-DVR} that Question
\ref{q:Fpure-Fsplit} has an affirmative answer when $R$ is a
DVR, and in \cite{DS18} it was shown
that the question has an affirmative answer when $R$ is a generically
$F$-finite domain. However, to the best of our knowledge Question
\ref{q:Fpure-Fsplit} is open even for excellent local rings of Krull 
dimension one that are not
regular, and for regular local rings of Krull dimension greater than one. We expect 
any progress on Question \ref{q:Fpure-Fsplit} to also shed light on the 
following question.

\begin{question}
\label{q:Fsplit-reg}
Suppose $R$ is a regular local ring that is Frobenius split. Is $R$ split
$F$-regular?
\end{question}

\noindent This question also seems to be open when $R$ has Krull dimension
greater than one, while the case of Krull dimension one follows by
Proposition \ref{thm:split-F-regular-DVR}. A likely first step is to try and
see if the convergent power series rings $K_n(k)$ are split $F$-regular 
whenever $k$ is a complete non-Archimedean field that admits nonzero continuous
functionals $k^{1/p} \rightarrow k$ and $n > 1$. 
Note that we know that $K_n(k)$ is Frobenius
split for any such $k$ by \cite[Thm.\ 4.4]{DM}, while split $F$-regularity
is known in a few cases;  see Proposition \ref {prop:split-F-regular-Tate}
and  Remark \ref{rem:split-Freg-TnKn}.\medskip

\par Switching gears, we now highlight some open questions on the solidity
of absolute integral closures. The next question is particularly
surprising.

\begin{question}
\label{q:absolute-solid}
Let $k$ be a field of prime characteristic $p > 0$ (one can even take
$k = \FF_p$). Is the absolute integral closure $k[x]^+$ a solid
$k[x]$ algebra? 
\end{question}

\noindent Note that the question has an affirmative answer if we replace
the polynomial ring over $k$ by the power series ring $k\llbracket x\rrbracket$ by
Lemma \ref{lem:fed12}, because the map 
$k\llbracket x\rrbracket \rightarrow k\llbracket x\rrbracket^+$ is pure since
$k\llbracket x\rrbracket$ is regular,
and hence a splinter. We would now like to show that if $k$
is $F$-finite, then solidity of $k[x]^+$ implies that the map
$k[x] \rightarrow k[x]^+$ splits. 

\begin{proposition}
\label{prop:evaluation-1}
Let $R$ be any domain of prime characteristic $p > 0$. Then, we 
have the following:
\begin{enumerate}[label=$(\roman*)$,ref=\roman*]
  \item\label{prop:im1} Let $S$ be any $R$-algebra. Then,
    \[
	I_{S/R} \coloneqq \im\bigl(\Hom_R(S,R) \xrightarrow{\ev @ 1} R\bigr)
	\]
	is a uniformly $F$-compatible ideal of $R$.
	
	\item\label{prop:splits}
    	Suppose $R$ is Noetherian, generically $F$-finite, and split $F$-regular.
	If $S$ is a solid $R$-algebra, then $R \rightarrow S$ splits.
	
	\item\label{prop:A+-splits} 
	If $R$ is Noetherian, generically $F$-finite, split $F$-regular,
	and $R^+$ is a solid $R$-algebra, then $R \rightarrow R^+$ splits.
	\end{enumerate}
\end{proposition}

\noindent Recall that an ideal $I$ of a ring $R$ of prime characteristic is 
\emph{uniformly $F$-compatible} if, for all $e \in \ZZ_{> 0}$ and
for every $p^{-e}$-linear map $\varphi$, we have 
$\varphi(F^e_{R*}I) \subseteq I$.
See \cite[Def.\ 3.1]{Sch10}.

\begin{proof}
The basic idea of the proof of $(\ref{prop:im1})$ comes from the 
proof of \cite[Thm.\ 5]{MS19}.
If $R$ is not $F$-solid, then $I_{S/R}$ is uniformly $F$-compatible because 
the only $p^{-e}$-linear map on $R$ is the trivial one for any $e > 0$ by
Remark \ref{rem:F-solidity-obs}$(\ref {rem:Fsolid-all-e})$.
It therefore suffices to consider the case when $R$ is $F$-solid. Consider an
integer $e > 0$, and a $p^{-1}$-linear map $\phi \in \Hom_R(F^e_{R*}R, R)$. 
Note that we have a natural map
\[
\eta_\phi\colon \Hom_{F^e_{R*}R}(F^e_{S*}S, F^e_{R*}R) \longrightarrow \Hom_R(S,R)
\]
given by $\eta_\phi(\psi) \coloneqq \phi \circ \psi \circ F^e_S$, 
where $F^e_S$ is the $e$-th
iterate of the Frobenius on $S$. We then get a commutative square
\[
\begin{tikzcd}
  \Hom_{F^e_{R*}R}(F^e_{S*}S, F^e_{R*}R) \arrow[r, "\eta_\phi"] \arrow[d, "\ev @
  1"']
    &  \Hom_R(S,R) \arrow[d, "\ev @ 1"] \\
  F_{R*}R \arrow[r, "\phi"]
&R \end{tikzcd}
\]
The image of the left vertical evaluation map is just $F^e_{R*}I_{S/R}$. By the
commutativity of the diagram and the definition of $I_{S/R}$, it now follows that 
\[
\phi(F^e_{R*}I_{S/R}) \subseteq I_{S/R}.
\]
Since $\phi$ was an arbitrary $p^{-e}$-linear map of $R$ for an arbitrary $e > 0$, 
this proves $(\ref{prop:im1})$.

\par $(\ref{prop:A+-splits})$ clearly follows from  $(\ref{prop:splits})$, and hence
it remains to show $(\ref{prop:splits})$. If $R$ is Noetherian,
generically $F$-finite and split $F$-regular, then $R$ is $F$-solid, and hence,
$F$-finite by \cite[Thm.\ 3.2]{DS18}. Then, $R$ has a smallest nonzero uniformly $F$-compatible ideal
with respect to inclusion,
namely the big or non-finitistic test ideal $\tau$ \cite[Thm.\ 6.3]{Sch10}. However, 
$F$-finiteness and split $F$-regularity implies that $\tau = R$. Now if
$S$ is a solid $R$-algebra, then the ideal $I_{S/R}$ from 
$(\ref{prop:im1})$ is nonzero by Remark \ref{rem:useful-solid-facts}.
Since $I_{S/R}$ is a uniformly $F$-compatible ideal by $(\ref{prop:im1})$ ,
it follows that $R = \tau \subseteq I_{S/R}$ by the minimality of $\tau$. 
Thus, $1 \in I_{S/R}$, which is equivalent to the splitting of  $R \rightarrow S$.
\end{proof}

\begin{remark}
A result of Hochster and Huneke shows that every weakly $F$-regular Noetherian domain $R$ is 
a splinter, in the sense that every module-finite extension of domains $R \hookrightarrow S$ 
splits \cite[Thm.\ 5.17]{HH94splitting}. Note that $S$ is a solid $R$-algebra for any such finite extension by Example
\ref{ex:finite-solid}. Thus, it appears that the property of being a
Noetherian $F$-finite split $F$-regular domain is quite a bit stronger than that
of being a splinter by Proposition \ref{prop:evaluation-1}$(\ref{prop:splits})$
because this property implies splitting of any solid algebra, module-finite or not.
Proposition \ref{prop:evaluation-1} and Lemma \ref{lem:fed12}
also have the interesting consequence
that when $(R,\fm)$ is $F$-finite, split $F$-regular, and complete
local, then purity of an $R$-algebra $S$ is equivalent to the 
$R$-solidity of $S$.
\end{remark}

\par Since conjecturally Noetherian $F$-finite splinters are split $F$-regular,
the following question then arises.

\begin{question}
\label{q:splinters}
Let $R$ be a Noetherian $F$-finite splinter. If $S$ is a solid $R$-algebra 
(not necessarily module-finite),
then does $R \rightarrow S$ split?
\end{question}

\noindent A negative answer to Question \ref{q:splinters} will imply that Noetherian
$F$-finite splinters are not split $F$-regular in general, although both notions
are known to coincide for $\QQ$-Gorenstein rings \cite[Thm.\ 1.1]{Sin99}. 
On the other hand, a positive answer
will likely shed light on the equivalence of weak and split $F$-regularity for
Noetherian $F$-finite rings, which is one of the outstanding open problems in
prime characteristic commutative algebra.

\par The uniform $F$-compatibility of the images of evaluation at $1$ maps
raises the question of characterizing the images of such maps, at least for
some special extensions of rings.

\begin{question}
\label{q:image-ev1}
Suppose $R$ is a Noetherian excellent domain of prime characteristic
$p > 0$ (one may assume $R$ is $F$-finite or complete local).
If $R^+$ is a solid $R$-algebra, then what is the image of evaluation
at $1$ map
\[
  \Hom_R(R^+,R) \xrightarrow{\ev @ 1} R\,?
\]
\end{question}

\noindent Note that the solidity assumption is important in Question
\ref{q:image-ev1} because Proposition \ref{prop:BCM-not-solid} shows that there are excellent 
Henselian regular local rings $R$ of prime characteristic for which $R^+$
is not a solid $R$-algebra.
Using the notation of Proposition \ref{prop:evaluation-1},
suppose $I_{R^+/R}$ is the image of this evaluation at $1$ map.
If $I_{R^+/R} = R$, then $R$ is a splinter because the map 
$R \hookrightarrow R^+$ splits, and so, must also be pure. However,
we do not know if the converse holds,
even in the case where $R$ is $F$-finite. 
In Appendix \ref{Karen-appendix},
Karen E. Smith will provide an answer to Question \ref{q:image-ev1}
when $R$ is Gorenstein and either complete local or $\NN$-graded and 
finitely generated over $R_0 = k$ a field, by reinterpreting results
from her fundamental thesis (see Theorems \ref {theorem:Smith-thesis} and
\ref {theorem:Smith-thesis2}).

\par One can replace $R^+$ by $R_\perf$ in Question \ref{q:image-ev1} and
ask what the image $I_{R_\perf/R}$ of the evaluation at $1$ map
$\Hom_R(R_\perf,R) \rightarrow R$ is. In this case, regardless of whether 
$R$ is excellent,  $I_{R_\perf/R} = R$ if and only if $R$
is Frobenius split. This equivalence follows by Proposition 
\ref{prop:colimits}$(\ref{prop:perf-split})$. Now suppose $R$ is not necessarily
Frobenius split. Then $I_{R_\perf/R} \neq R$. Let $\fp$ be a prime ideal
of $R$ such that $I_{R_\perf/R} \nsubseteq\fp$ (no such prime ideal will exist
if $R_\perf$ is not a $R$-solid). Choose an element 
$a \in I_{R_\perf/R}$ such that $a \notin \fp$. Then there exists an $R$-linear map
$\phi\colon R_\perf \rightarrow R$ that sends $1$ to $a$. Localizing at $\fp$
and using the fact that perfect closures commute with localization, we then
get a map $\phi_\fp\colon (R_\fp)_\perf \rightarrow R_\fp$ that maps $1$ to a 
unit in $R_\fp$ because $a \notin \fp$. Then the composition
\[
  (R_\fp)_\perf \overset{\phi_\fp}{\longrightarrow} R_\fp \xrightarrow{-\cdot a^{-1}} R_\fp
\]
sends $1$ to $1$. Consequently, $R_\fp$ is Frobenius split. 
Thus, the closed subset $Z$ of $\Spec(R)$ defined by $I_{R_\perf/R}$ contains the
non-Frobenius split locus of $\Spec(R)$. However, since 
\[
R_\fp \otimes_R \Hom_R(R_\perf,R) \neq \Hom_{R_\fp}\bigl((R_\fp)_\perf,
R_\fp\bigr)
\]
in general,
it is not clear to the authors if $Z$ coincides with the non-Frobenius split 
locus of $\Spec(R)$. Thus, we raise the following question:

\begin{question}
\label{q:non-F-split-locus}
Let $R$ be a Noetherian domain of prime characteristic $p > 0$. Suppose $R_\perf$
is a solid $R$-algebra. Then, does the image of the evaluation at $1$ map
\[
  \Hom(R_\perf, R) \xrightarrow{\ev @ 1} R
\]
define the non-Frobenius split
locus of $R$?
\end{question}

\par We showed that a Noetherian domain $R$
with solid $R^+$ is Japanese. However, we do not know any examples 
in prime characteristic of Noetherian domains with solid absolute integral closure
when such domains are not complete local.
This raises the final question in our summary of open questions.

\begin{question}
\label{q:A+-solid-nonexc}
Is there a non-excellent Noetherian domain of prime characteristic $p > 0$
for which $R^+$ is a solid $R$-algebra?
\end{question}

\noindent Examples of non-excellent Noetherian domains $R$ with solid $R^+$
abound when $R$ contains $\QQ$ because solidity of $R^+$ is then equivalent
to $R$ being Japanese by Corollary \ref{cor:N-1-char0}. In prime characteristic,
$R^+$ can be replaced by $R_\perf$ and then the analogue of 
Question \ref{q:A+-solid-nonexc} for $R_\perf$ has an affirmative answer by Corollary \ref{cor:non-exc-Fsplit}.

\addtocontents{toc}{\protect\medskip}
\makeatletter
\addtocontents{toc}{\@tocline{-1}{0pt}{0pt}{}{\itshape}{Appendix by Karen E. Smith}{}}
\makeatother
\appendix
\section{Solidity of absolute integral closures and the test ideal\texorpdfstring{\except{toc}{\\{\footnotesize\MakeUppercase{Karen E.
Smith}}}}{ (by Karen E. Smith)}}
\label{Karen-appendix}
\counterwithout{subsubsection}{section}
\counterwithin{subsubsection}{subsection}

The purpose of this appendix is to 
 address the following question raised in Section \ref{sect:openquestions}:

\begin{question}[see Question \ref {q:image-ev1}]
\label{q1}
Suppose  $R$ is a complete local domain of prime characteristic $p > 0$. What is the 
 image of  the ``evaluation at 1"  map
 \[
\Hom_R(R^+, R) \longrightarrow R?
\]
\end{question}

Here, $R^+$ denotes the {\it absolute integral closure}  of $R$ -- that is,  the integral closure of $R$ in an algebraic closure of its fraction field. 

\medskip
We prove the following satisfying answer, at least in the Gorenstein case: 

\begin{theorem}
\label{theorem:Smith-thesis}
Let $(R,\fm)$ be a complete local Noetherian Gorenstein domain of prime characteristic
$p > 0$. Then the image of the evaluation at $1$ map
\[
\Hom_R(R^+, R) \xrightarrow{\ev @ 1} R
\]
is the test ideal $\tau(R)$ of $R$.
\end{theorem}

A graded version also holds; see Theorem \ref{theorem:Smith-thesis2}.
\medskip

The image of the  evaluation map in Question \ref{q1}  was shown to be a  {\it
  uniformly $F$-compatible ideal}  in general in  Proposition \ref{prop:evaluation-1} of the main paper. 
  Uniformly $F$-compatible ideals,
introduced by Schwede as a  prime characteristic analog of  {\it log canonical centers }  (see \cite[Def.\ 3.1]{Sch10}), are a generalization of the $F$-ideals introduced in \cite{Sm94} as annihilators of Frobenius-stable submodules of the top local cohomology module of a Noetherian 
local ring $(R, \fm)$ and later generalized in 
\cite{Sm95b} and \cite{Gennady01}.
In the Gorenstein case, there is a unique smallest nonzero $F$-ideal (or uniformly  $F$-compatible ideal), which was shown in \cite{Sm94} to be the famous {\bf test ideal}  $\tau(A) $ of Hochster and Huneke's tight closure theory.  This was a key step in the proof that the tight closure of a parameter ideal in a local excellent ring is equal to its plus closure \cite{Sm94}.
Our proof of Theorem \ref{theorem:Smith-thesis} invokes the methods of  \cite{Sm94}.

\begin{remark}
Theorem \ref{theorem:Smith-thesis}, in particular, implies that $R^+$ is a {\it solid} $R$-algebra (see Definition \ref{def:solid-modules}), in light of 
Hochster and Huneke's result that a complete local domain always admits a  {\it test element} 
\cite[\S6]{HH90}.
\end{remark}

\subsection{Tight Closure and Test Ideals}

Tight closure, introduced by Hochster and Huneke, is a closure operation on submodules of a fixed ambient module $M$ over an excellent ring of prime characteristic $p > 0$. One property is that, for modules over a regular ring, all submodules are tightly closed. Another property is that local domains  for which all ideals are tightly closed are Cohen-Macaulay \cite[Rem.\ 7.1]{HH90}.  Using these two properties,  Hochster and Huneke gave  a simple conceptual proof that direct summands of regular rings are Cohen-Macaulay in prime characteristic, and many other important results in commutative algebra.  The basic theory is developed in  \cite{HH90}.

\begin{definition}
Let $R$ be a Noetherian domain of prime characteristic $p > 0$. Let $M$ be any $R$-module, and let $N\subset M$ be a submodule. 
The \emph{tight closure of $N$ in $M$}, denoted $N^*_M$, is the 
collection of elements $m \in M$ for which there exists a nonzero element $c \in R$ such
that $c \otimes m$ lies in the image of the canonical map
\begin{equation}\label{eq1}
F^e_*R \otimes_R N \rightarrow F^e_*R \otimes_R M
\end{equation}
for all $e \gg 0$.
\end{definition}

Two special cases are of particular interest. One is when the ambient module $M$ is $R$, so we are computing tight closures of {\it ideals}, and the other is when $M$ is arbitrary but $N$ is zero.

To verify a module is tightly closed using the definition above, we would need to consider maps of the form  (\ref{eq1}) above {\it for all nonzero} $c$. 
Fortunately, it turns out that so-called {\bf test elements} exist that can be used to check an element is in the  tight closure  by checking relations (\ref{eq1})  for {\it just one ``test element" $c$}.

\begin{definition}
\label{def:test-element}
Let $R$ be a Noetherian domain of prime characteristic $p > 0$. A nonzero
 element $c \in R$ is a \emph{test element} if for all  ideals $I$ of $R$, 
$$cI^* \subseteq I$$ (equivalently, for all $a \in I^*$, $ca^{p^e} \in I^{[p^e]}$ for all
$e > 0$). 
Equivalently, $c$ is a test element if and only if  $c$ annihilates the tight closure of zero in every finitely generated module $M$. 
\end{definition}
\begin{definition}
\label{def:test-ideal}
 The \emph{test ideal of $R$}, denoted $\tau(R)$, is the ideal of $R$
generated by all test elements.
\end{definition}

A deep theorem of Hochster and Huneke ensures that the test ideal $\tau(R)$ is
nonzero when $(R, \mathfrak{m}) $ is a complete local domain (and quite a bit more generally);  see 
\cite[$\mathsection 6$]{HH94}.

The  behavior of tight closure  is a bit murky for non-finitely generated ambient modules.  For this reason, the following definition is needed:   For any module $M$, the
{\it  finitistic tight closure of zero}  in $M$ is
$$0^{*\fg}_M \coloneqq \bigcup_{M'} 0^{*}_{M'},$$ 
where $M'$ ranges over all finitely
generated submodules of $M$. Note that $0^{*\fg}_M \subseteq 0^*_M$, and that these are equal if $M$ is finitely generated. 

\medskip

The test ideal  can be characterized  as follows:

\begin{citedprop}[{\cite[Prop.\ (8.23)(d)]{HH90}}]
\label{prop:test-ideal-tight-closure}
Let $(R,\fm)$ be a Noetherian local domain of prime characteristic $p > 0$.
Let $E$ denote an injective hull of its residue field. Then
\[
\tau(R) = \Ann_R(0^{*\fg}_E).
\]
\end{citedprop}

\subsection{The Proof of Theorem
\except{toc}{\ref{theorem:Smith-thesis}}\for{toc}{\ref*{theorem:Smith-thesis}}}

We will deduce Theorem \ref{theorem:Smith-thesis} from the following  result:

\begin{citedthm}[{\cite[Thm.\ 5.6.1.]{Sm94}}]
\label{prop:smith-key-insight}
Let $(R,\fm)$ be an excellent local domain of prime characteristic $p > 0$
and dimension $d$. Then the kernel of the canonical map
\[
H^d_\fm(R) \rightarrow H^d_\fm(R^+)
\]
is precisely $0^{*\fg}_{H^d_\fm(R)}$.
\end{citedthm}

\begin{remark}
A key point in the proof of Theorem \ref{prop:smith-key-insight} above is that
\[
 0^{*\fg}_{H^d_\fm(R)} = 0^*_{H^d_\fm(R)}
\] 
in the Gorenstein case. 
An open question predicts that  $0^*_E = 0^{*\fg}_E$ in general for the injective hull of the residue field $E$ of a local ring \cite[\S 7]{Sm93}. This is equivalent to the long-standing open question about the equivalence of 
strong and weak $F$-regularity; see \cite[Prop.\  7.1.2]{Sm93}, \cite[Prop.\ 2.9]{Gennady01}. 
In general, since $0^{*\fg}_E \subseteq 0^*_E$, this conjecture is equivalent to saying that the corresponding inclusion of annihilators  
$\Ann_R(0^*_E) \subseteq \Ann_R(0^{*\fg}_E)$ is an equality by 
\cite[Lem.\ 3.1(v)]{Sm94}.{\footnote{
The ideal $\Ann_R(0^*_E)$ of $R$ is called the {\emph{non-finitistic test ideal} or the \emph{big test ideal}}, although the latter term is a bit confusing since 
it is {\it a priori} {\bf smaller} than the usual (finitistic) test ideal $\Ann_R(0^{*\fg}_E)$. The name ``big" comes from the fact that the elements of $\Ann_R(0^*_E)$ can be used as test elements even for non-finitely generated modules whereas the 
elements of the usual test ideal $\Ann_R(0^{*\fg}_E)$ works ({\it a priori}) only for tests where the ambient module is Noetherian.
 }}
\end{remark}

Theorem  \ref{theorem:Smith-thesis} will follow by applying Matlis duality to Theorem \ref{prop:smith-key-insight}.

\subsubsection{Matlis Duality}  Let $(R, \fm)$ be a complete Noetherian
 local ring of dimension $d$, and let  $E$ be an injective hull of its residue field. 
Matlis duality is the exact contravariant  functor on $R$-modules sending
each $M$ to $M^{\vee} := \Hom_R(M, R)$. This functor takes Noetherian modules to Artinian modules, and vice versa, and is involutive when
 restricted to either class of modules: that is $(M^{\vee})^{\vee} \cong M$ if $M$ satisfies either the ACC or DCC condition on submodules.

 \par The {\bf canonical module} of $(R, \fm)$ in this context is defined as any  $R$-module $\omega_R$ that is a Matlis dual to the
 local cohomology module $H^d_{\fm}(R)$. Since $H^d_{\fm}(R)$ is Artinian,  $\omega_R$ is Noetherian. 
The ring $R$ is a canonical module of itself when $R$ is Gorenstein -- a 
sister fact to the fact that $E$ can be taken to be  $H^d_{\fm}(R)$. See \cite{BH98, HK71} for basics on Matlis duality and the canonical module.
 
 Matlis duality gives us a perfect pairing
$$
\omega_R \times H^d_{\fm}(R) \rightarrow E
$$
which can be concretely understood by viewing  $w \in \omega_R$  as inducing the  $R$-linear map ``evaluation map"
\[
\varphi_w: H^d_\fm(R) \rightarrow   E 
\]
given that, by definition, $\omega_R = \Hom_R(H^d_m(R), E)$. 
In particular, for an $R$-submodule $M$ of $H^d_\fm(R)$, we can define 
 the {\bf annihilator of $M$ in $\omega_R$}, denoted $\Ann_{\omega_R} M$,
\[
\{w \in \omega_R \,\, | \,\,  \varphi_w(m) = 0  \,\,\, {\rm{for \,\,all\,\, }} m \in M\} \, \subset \,\omega_R
\]
where 
$ \varphi_w: H^d_\fm(R) \rightarrow E$ is as defined above. 
See \cite[\S 2.4]{Sm93}  for details on this point of view. 

\medskip

With this definition and notation, we collect a few facts.

\begin{lemma}
\label{lem:Matlis-dual}
Let $(R,\fm)$ be a Noetherian, $S_2$, complete local ring of dimension $d$, and  let $E$ be an injective hull of its residue field. 
 If $M$ is a submodule of $H^d_\fm(R)$, then $M^\vee \cong \omega_R/\Ann_{\omega_R} M$.
\end{lemma}

\begin{proof}
See  \cite[Prop. 2.4.1.(i)]{Sm93}. 
\end{proof}

\begin{lemma}\label{lem2}
$(R,\fm)$ be a Noetherian complete local ring of dimension $d$, and  let $E$ be an injective hull of its residue field. For any $R$-module map $\phi: R \rightarrow S$, the Matlis dual of the induced map 
$H^d_m(R) \overset{\phi^*}\longrightarrow H^d_m(S)$ can be identified with the map 
$$
\Hom_R(S, \omega_R) \rightarrow  \omega_R \,\,\,\,\,\,\,\,\,\, \,\,\,\,\,\,\,\,\,\,  \Psi \mapsto \Psi(\phi(1)).
$$
\end{lemma}

\begin{proof} This is really a much more general fact, following from the adjointness of tensor and Hom.
 First note that $H^d_\fm(R) \overset{\phi^*}\longrightarrow H^d_\fm(S)$ can be identified with the map
$$
R \otimes_R H^d_\fm(R) \xrightarrow {\phi \otimes id} S \otimes_R H^d_\fm(R).
$$
So applying the Matlis dual functor $\Hom_R(-, E)$ we have 
$$
\Hom_R(S \otimes_R H^d_{\fm}(R), E) \rightarrow \Hom_R(R \otimes_R H^d_\fm(R), E) \,\,\,\,\,\,\,\, \,\, {\text{sending}}  \,\,\,\,\,\,\, \Phi  \mapsto   \Phi \circ (\phi \otimes id),
$$
which, under adjunction, becomes   the map
$$
\Hom_R\big{(}S, \, \Hom_R( H^d_\fm(R), E)\big{)} \rightarrow \Hom_R\big{(}R, \, \Hom_R( H^d_\fm(R), E) \big{)} \,\,\,\,\,\,\,\,{\text{sending}} \,\,\,\,\,\,\, \Psi  \mapsto   \Psi \circ \phi.
$$
Now using the identification of  $ \Hom_R( H^d_\fm(R), E)$ with $\omega_R$, we have a map 
 $$
\Hom_R(S, \omega_R) \rightarrow \Hom_R(R, \omega_R) \,\,\,\,\,\,\,\, {\text{sending}} \,\,\,\,\,\,\, \Psi  \mapsto   \Psi \circ \phi,
$$
 and since the latter is identified with $\omega_R$ via the map $f\mapsto f(1)$, this becomes
 $$
\Hom_R(S, \omega_R) \rightarrow  \omega_R \,\,\,\,\,\,\,\, \,\,\,\,\,\,\, \Psi  \mapsto   \Psi (\phi (1)),
$$
that is, the ``evaluation at $\phi(1)$" map.
\end{proof}

 \begin{proof}[Proof of Theorem \ref{theorem:Smith-thesis}]  Let $(R, \fm)$ be a complete local Gorenstein domain of dimension $d$. We first recall that because $(R, \fm)$ is Gorenstein,  the local cohomology module $H^d_{\fm}(R)$ is an injective hull of the residue field of $R$. Thus the test ideal for $R$ is the annihilator of the $R$-module 
$0^{*\fg}_{H^d_{\fm}(R)}$.  By Theorem \ref{prop:smith-key-insight}, the
 test ideal is then also the annihilator of the kernel of the natural map
 \[
H^d_\fm(R) \rightarrow H^d_\fm(R^+).
\]
To relate this to the evaluation at $1$ map, we apply Matlis duality to the  exact sequence
\begin{equation}
\label{eq:exact}
0 \rightarrow 0^{*\fg}_{H^d_\fm(R)} \rightarrow H^d_\fm(R) \rightarrow H^d_\fm(R^+).
\end{equation}
Using Lemma 
\ref{lem2} to analyze the Matlis  dual of this  sequence, we get an exact sequence
\[
\Hom_R(R^+, \omega_R)\xrightarrow{\ev @1} \omega_R \rightarrow 
(0^{*\fg}_{H^d_\fm(R)})^\vee \rightarrow 0.
\]
On the other hand, this sequence can be rewritten as 
\begin{equation*}
\Hom_R(R^+, \omega_R)\xrightarrow{\ev @1} \omega_R \rightarrow  
\omega_R/\Ann_{\omega_R} 0^{*\fg}_{H^d_\fm(R)} \rightarrow 0,
\end{equation*}
where we have used Lemma \ref{lem:Matlis-dual} for the identification
\[
(0^{*\fg}_{H^d_\fm(R)})^\vee \cong \omega_R/\Ann_{\omega_R} 0^*_{H^d_\fm(R)}.
\]
 Thus, by exactness, the evaluation at $1$ map 
\[
 \Hom_R(R^+, \omega_R)\xrightarrow{\ev @1} \omega_R   \,\,\,\,\,\,\,\,\,\,\,\,\,\,\,\,\,\,  \,\,\,\,\,\,\,\,\,\,\,\,\,\,\,\,\,\, \phi\mapsto \phi(1)
\]
 has image
 \[
  \Ann_{\omega_R} 0^{*\fg}_{H^d_\fm(R)},
  \]
 an important submodule of $\omega_R$ called the {\it parameter test module} in \cite{Sm95b}.
 In the Gorenstein case, using the fact that $\omega_R \cong R$ and $H^d_m(R) \cong E$, we finally conclude that  
 the image of 
\[
\Hom_R(R^+, R)\xrightarrow{\ev @1} r  \,\,\,\,\,\,\,  \,\,\,\,\,\,\,\,\,\,\,\,\,\,\,\,\,\, \phi \mapsto \phi(1)
\]
is the test ideal by Proposition \ref {prop:test-ideal-tight-closure}. 
 Theorem \ref{theorem:Smith-thesis} now follows.
\end{proof}

\subsection{Graded variants of absolute integral closures}

Suppose that $(R, \fm)$ is an $\NN$-graded domain, finitely generated over $R_0  = k$, a field, and with unique homogeneous maximal ideal $\fm$.  In this case, there are two graded variants of $R^{+}$.

Take any $z\in R^+$, and let 
\begin{equation}\label{eq2}
f(x) = x^n + a_1 x^{n-1} + \cdots + a_{n-1}x + a_n
\end{equation}
be a monic polynomial over $R$ witnessing the integrality of $z$ over $R$. We say  that $f(x)$ is {\bf homogenous} if 
 each $a_i$ is homogeneous in $R$ and there is a choice of rational number $\beta \in \QQ$ such that, setting $\deg x = \beta$, the polynomial $f(x)$ becomes homogeneous. Equivalently, 
this says that there exists a rational number $\beta$ such that
 $\deg a_i = i \beta $ for all $i$.  We say that $z \in R^+$ has degree  
$\beta$ in this case; it is not hard to see that this is well-defined \cite[Lemma 4.1]{HH93}.
Using this, Hochster and Huneke defined the following two natural subrings
 of  $R^+$ \cite{HH93}:

\begin{definition}
\label{def:graded-abs-int-cl}
 Let $R$ be an $\NN$-graded domain, finitely generated over $R_0  = k$, where $k$ is a field of characteristic $p>0$. Then 
\begin{itemize}
\item $R^{+\GR} $ is the subring of $T^+$ generated by all elements $a\in T^+$ that satisfy some {\it homogeneous} equation (\ref{eq2}) of integral dependence. Note that $R^{+\GR} $ is $\QQ$-graded. 
\item $R^{+\gr} $ is the subring of $R^+$ generated by all elements $a\in R^{+\GR}$ of {\it integral degree}. Note that $R^{+\gr}$ is $\NN$-graded.
\end{itemize}
\end{definition}

\noindent Both $R^{+\gr}$ and $R^{+\GR}$ are big Cohen-Macaulay $R$-algebras in prime characteristic $p > 0$, by \cite[Main Thm.\ 5.15]{HH93}.

\medskip

We work in the category of $\ZZ$-graded $R$-modules. In particular, 
for $\ZZ$-graded $R$-modules $M$ and $N$, and $d \in \ZZ$,
 we say that an $R$-linear homomorphism $f:M\rightarrow N$ has \emph{degree $d$} 
if $f(M_n) \subset N_{n+d}$ for all $n\in \ZZ$.
The set of all homomorphisms of degree $d$ is written $\Hom_d(M, N)$. 
We are interested in the module of graded homomorphisms 
\[
  \grHom_R(M, N) = \bigoplus_{d\in \ZZ} \Hom_d(M, N),
\]
which is a $\ZZ$-graded $R$-module. If  $M$ is finitely generated, it is not
hard to see that $\grHom_R(M, N) $ is  the same as $\Hom_R(M, N)$, but in general the former module may be  strictly smaller.

The following analogue of Theorem \ref{theorem:Smith-thesis} holds:

\begin{theorem}
\label{theorem:Smith-thesis2}
Let $R$ be an $\NN$-graded  Gorenstein domain, finitely generated over $R_0  = k$, where $k$ is a field of characteristic $p>0$.  Then the image of the evaluation at $1$ map
\[
\grHom_R(R^{+\gr},\,\, R) \xrightarrow{\ev @ 1} R
\]
is the test ideal $\tau(R)$ of $R$.
The same holds if we replace $R^{+\gr}$ by $R^{+\GR}$, and define 
\[
\grHom_R(R^{+\GR}, R) = \bigoplus_{d\in \QQ} \Hom_d(R^{+\GR}, R).
\]
\end{theorem}

\begin{proof}[Proof] First note that for simple degree reasons, $R^{+\gr}\hookrightarrow R^{+\GR}$ splits over $R^{+\gr}$ and hence over $R$. This implies that the maps
 $$
\grHom_R(R^{+\gr}, R) \xrightarrow{\ev @ 1} R \,\,\,\,\,\,\,\,\, {\text{and}}\,\,\,\,\,\,\,\,\  \grHom_R(R^{+\GR}, R) \xrightarrow{\ev @ 1} R
$$ have the same image. So it suffices to prove the statement for the $\ZZ$-graded $R$-module  $R^{+\gr}$. 

Now the  proof is essentially the same as in the complete case, but we work in the graded  (${}^*$) category instead, using graded Matlis duality and a graded analog of Theorem \ref{prop:smith-key-insight}. 

\subsubsection{Graded Matlis duality} [For details, see \cite{BH98}.]
If $E$ denotes an injective hull of the residue field of the graded ring $(R, \fm)$ at the unique homogeneous maximal ideal, then $E$ is $\ZZ$-graded, and we can define a {\bf graded Matlis dual functor}
$$
M\mapsto  M^{\vee} =  \grHom_R(M, E) 
$$
which has all the same properties we reviewed for complete local Noetherian rings, provided we restrict to graded $R$-modules.
In particular, because $H^d_\fm(R)$ is graded, its graded Matlis dual is the canonical module $\omega_R$, which is a graded $R$-module. We  still have a perfect 
 pairing   $\omega_R \times H^d_{\fm}(R) \rightarrow E$ which respects the grading, so the annihilator of a graded submodule of 
 $H^d_{\fm}(R) $ in 
 $\omega_R$ can be defined. With this set-up,
 the following graded analogues of Lemmas \ref{lem:Matlis-dual} and 
\ref{lem2} hold, whose proof we omit.

\begin{lemma}
\label{gradedMatlis-dual}
Let $(R,\fm)$ be an  $S_2$  $\NN$-graded domain of dimension $d$,  finitely generated over $R_0  = k$, a field, and with unique homogeneous maximal ideal $\fm$. Let $E$ be an injective hull of its residue field $R/\fm$. 
\begin{enumerate}
\item  If $M$ is a graded submodule of $H^d_\fm(R)$, then $M^\vee \cong \omega_R/\Ann_{\omega_R} M$.
\item  For any graded map of graded $R$-modules  $\phi: R \rightarrow S$, the graded Matlis dual of the induced map 
$H^d_m(R) \overset{\phi^*}\longrightarrow H^d_m(S)$ can be identified with the map 
$$
\grHom_R(S, \omega_R) \rightarrow  \omega_R \,\,\,\,\,\,\,\,\,\, \,\,\,\,\,\,\,\,\,\,  \Psi \mapsto \Psi(\phi(1)).
$$
\end{enumerate}
\end{lemma}

\medskip
Returning to the proof of Theorem \ref {theorem:Smith-thesis2}, 
the main theorem of \cite{Sm95} states  that 
$I^* = IR^{+\gr}\cap R = IR^{+\GR}\cap R$ for homogeneous parameter
 ideals $I$ in $R$. So using  the method of \cite{Sm94}, we see also that 
the finitistic tight closure of zero in $H^d_{\fm}(R)$  is graded, and is  the kernel of either of the natural maps $H^d_{\fm}(R) \rightarrow H^d_{\fm}(R^{+\gr})$ or $H^d_{\fm}(R) \rightarrow H^d_{\fm}(R^{+\GR})$.
In particular,  we have an exact sequence of graded modules
$$
0 \rightarrow 0^{*\fg}_{H^d_{\fm}(R)} \rightarrow H_{\fm}^d(R) \rightarrow H_{\fm}^d(R^{+\gr})
$$
whose graded Matlis dual produces an exact sequence of graded $R$-modules
$$
\grHom_R(R^{+\gr}, \omega_R) \xrightarrow{\ev @ 1} \omega_R \rightarrow \omega_R/\Ann_{\omega_R}(0^{*\fg}_{H^d_{\fm}(R)} )
\rightarrow 0 
$$
as before. Using the fact that $R$ is Gorenstein, we conclude that the image of  
$$\grHom_R(R^{+\gr}, R) \xrightarrow{\ev @ 1} R$$  is the test ideal  $\tau(R)$ following the same argument as in the complete case, finishing the proof of Theorem \ref {theorem:Smith-thesis2}.
 \end{proof}

\addtocontents{toc}{\protect\medskip}

\end{document}